\documentclass[leqno,11pt,a4paper]{amsart}
\usepackage{jc}
\usepackage{comment}
\usepackage{todonotes}
\usepackage{amsmath}
\usepackage{mathtools,amsfonts,stmaryrd}
\usepackage{color}
\usepackage{enumitem}
\usepackage{tikz-cd}
\usepackage[normalem]{ulem}

\newcommand{\GW}{\on{GW}}
\newcommand{\id}{\on{Id}}

\newcommand{\g}{\mathfrak{g}}
\renewcommand{\c}{\mathfrak{c}}
\newcommand{\f}{\mathfrak{f}}
\newcommand{\cM}{\mathcal{M}}
\newcommand{\cE}{\mathcal{E}}
\newcommand{\cJ}{\mathcal{J}}
\renewcommand{\cB}{\mathcal{B}}
\newcommand{\cL}{\mathcal{L}}

\newcommand{\T}{\mathbb{T}}
\newcommand{\close}{\on{Close}}
\newcommand{\moduli}{\mathcal{M}_{\sigma}( X;J)}
\newcommand{\moduliP}{\mathcal{M}_{\sigma}( X;J,P)}

\newcommand{\moduliLEJ}{\mathcal{M}_{\sigma}( X;J,P, S_T,S_T)}
\newcommand{\moduliLEJO}[2]{\mathcal{M}_{\sigma}( X;J_{#1},p_{#2}, S_T, S_T)}
\newcommand{\modulitrans}{\cM(X;J,S_+,S_-)}
\newcommand{\modulireg}{\cM^{\on{reg}}(X;J,S_+,S_-)}
\newcommand{\moduliparam}{\underline\cM}

\newcommand{\PHam}{\on{CHam}}

\newcommand{\indRS}{\on{RS}}

\newcommand{\bA}{\mathbf{A}}
\newcommand{\bJ}{\underline{J}}

\newcommand{\supp}{\on{supp}}

\newcommand{\smallneg}{\text{-}}

\newcommand{\Sp}{\on{Sp}}

\newcommand{\ASpec}{\on{ASpec}}
\newcommand{\LSpec}{\on{LSpec}}

\newcommand{\wY}{\mathfrak{w}}

\newcommand{\cGr}{\mathfrak{c}_{\on{Gr}}}

\renewcommand{\hat}{\widehat}

\graphicspath{{images/}}
\begin{document}
\title[Elementary SFT Spectral Gaps And The Strong Closing Property]{Elementary SFT Spectral Gaps And The Strong Closing Property}

\author{J.~Chaidez}
\address{Department of Mathematics\\University of Southern California\\Los Angeles, CA\\90089\\USA}
\email{julian.chaidez@usc.edu}

\author{S. Tanny}
\address{School of Mathematics\\Institute for Advanced Study\\Princeton, NJ\\08540\\USA}
\email{tanny.shira@gmail.com}

\maketitle

\begin{abstract} We formulate elementary SFT spectral invariants of a large class of symplectic cobordisms and stable Hamiltonian manifolds, in any dimension. We give criteria for the strong closing property using these invariants, and verify these criteria for Hofer near periodic systems. This extends the class of symplectic dynamical systems in any dimension that satisfy the strong closing property.
\end{abstract}

\tableofcontents

\section{Introduction} \label{sec:intro} A smooth dynamical system on a manifold has the \emph{$C^k$-closing property} if every given nearly periodic orbit can be made periodic by an arbitrarily $C^k$-small perturbation of the system. This property dates back to Poincar\'{e} \cite{p1899}, who asked if the dynamical systems that arise in celestial mechanics satisfy a closing property. In modern dynamics, this problem was revisited by Peixoto \cite{px1988}, Pugh \cite{p1967a,p1967b} and Pugh-Robinson \cite{pr1983}, who proved the $C^1$-closing property for smooth systems, Hamiltonian systems and in several other settings.

\begin{figure}[h]
\centering
\includegraphics[width=.6\textwidth]{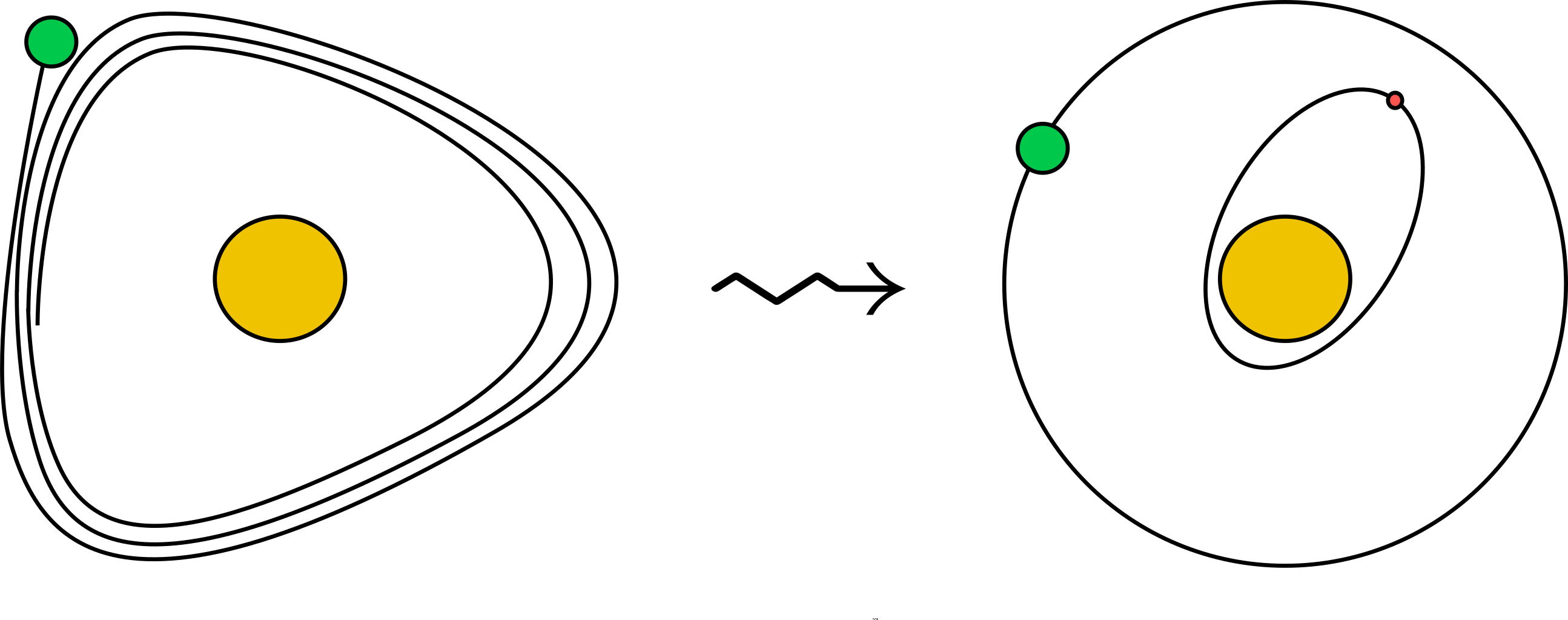}
\caption{The closing property states that nearly closed orbits can be closed by some small perturbation, akin to the introduction of a planet with small mass.}
\label{fig:closing_lemma}
\end{figure}

These landmark results were followed by a relative lull, as examples due to Pugh \cite{p1975}, Hernandez \cite{h1979}, Robinson \cite{g1987} and Herman \cite{herman1991exemples,herman1991differentiabilite} showed that $C^2$ and higher closing properties do not hold at the level of generality of the $C^1$ version. Despite this, Smale included the $C^\infty$-closing property as the 10th problem in his influential problem list \cite{s1998}.

\vspace{3pt}
Over the last decade, the situation changed dramatically due to work of Irie \cite{i2015}, who used Floer theory to prove a strong version of the $C^\infty$ closing property for Reeb flows on closed contact $3$-manifolds. This breakthrough spurred a flurry of recent activity at the intersection of dynamics and symplectic geometry, resulting in a smooth closing property for Hamiltonian surface maps \cite{a2016} and area preserving surface maps \cite{cpz2021,eh2021}, along with several equidistribution results \cite{i2018,p2021}.

\vspace{3pt}

The main tools in the works following Irie are spectral invariants coming from embedded contact homology (ECH)  \cite{h2010} and its relatives, which can only be used in low dimensions. In \cite{chaidez2022contact}, the present authors, Datta and Prasad took the first steps in applying Floer theoretic methods to the study of higher dimensional closing properties, and in particular addressed a conjecture of Irie \cite{irie2022strong}. However, the methods of \cite{chaidez2022contact} have many limitations related to the difficulty of defining and computing contact homology \cite{p2015,p2016}. Further advances on higher dimensional closing lemmas were achieved by \cite{x2022,cineli2022strong}

\vspace{3pt}

In this paper, we give a new construction of spectral invariants, called \emph{ESFT spectral gaps}, for a large class of symplectic cobordisms and stable Hamiltonian manifolds, in any dimension. They may be viewed as a unification and generalization of recent invariants due to McDuff-Siegel \cite{mcduff2021symplectic}, Hutchings \cite{hutchings2022elementary} and Edtmair \cite{edtmair2022elementary}. In particular, the construction involves  a min-max procedure  over moduli-spaces of pseudo-holomorphic curves in symplectic field theory, and does not require Floer homology or transversality technology. These spectral invariants admit many useful properties that permit their use as a useful black box tool in symplectic dynamics.

\vspace{3pt}

As an application, we provide spectral gap criteria for several strong variants of the smooth closing property. We prove that this criterion holds for flows that are near-periodic in a very low-regularity (Hofer or Banach-Mazur) sense, yielding a rich source of symplectic dynamical systems in any dimension that satisfy a strong closing property. In particular, this proves a substantial generalization of Irie's conjecture in \cite{irie2022strong}. We hope to further expand the known cases of the smooth closing property using these tools, in future works.

\vspace{3pt}

\subsection{Conformal Hamiltonian Manifolds} We start by introducing the objects for which our invariants are defined. These are types of stable Hamiltonian manifolds and symplectic cobordisms, familiar from the construction of symplectic field theory  \cite{egh2000,wendlsymplectic}.

\begin{definition*} A \emph{conformal Hamiltonian manifold} $Y$ is an odd-dimensional manifold equipped with a maximally non-degenerate, closed $2$-form $\omega$ and a $1$-form $\theta$ that satisfies
\[
\theta|_{\on{ker}(\omega)} > 0 \qquad\text{and}\qquad d\theta = c_Y \cdot \omega \qquad\text{ for a \emph{conformal constant} }c_Y \ge 0
\]
Similarly, a \emph{conformal symplectic cobordism} $X:Y_+ \to Y_-$ is a symplectic cobordism $(X,\Omega)$ equipped with $1$-form $\Theta$ whose symplectic dual vector-field $Z$ satisfies
\[
Z \text{ points outward at $Y_+$ and inward at $Y_-$} \qquad\text{and}\qquad \mathcal{L}_Z\Omega = c_X \cdot \Omega\text{ for some }c_X \ge 0
\] \end{definition*}

Contact manifolds and mapping tori of symplectomorphisms are the basic examples of conformal Hamiltonian manifolds, and these are (essentially) the only two types. This unifying notion allows us to simultaneously construct invariants for contact manifolds and symplectomorphisms throughout the paper, and illuminates  the common features of the two settings that makes the construction possible. We will introduce these spaces in more detail in Section \ref{subsec:conformal_hamiltonian_manifolds}.

\vspace{3pt}

The fundamental attribute of conformal Hamiltonian manifolds, differentiating them from arbitrary stable Hamiltonian manifolds, is the existence of a \emph{canonical} symplectic structure
\[
\hat{\omega} \qquad\text{on the symplectization}\qquad \R_r \times Y
\]
It is, roughly, the unique symplectic structure that restricts to the conformal Hamiltonian structure on $Y = 0 \times Y$ and that expands conformally in the $\partial_r$-direction. Similarly, conformal symplectic cobordisms admit a canonical symplectic structure on the completion $\hat{X}$ with a natural conformal vector-field.

\begin{figure}[h]
\centering
\includegraphics[width=.8\textwidth]{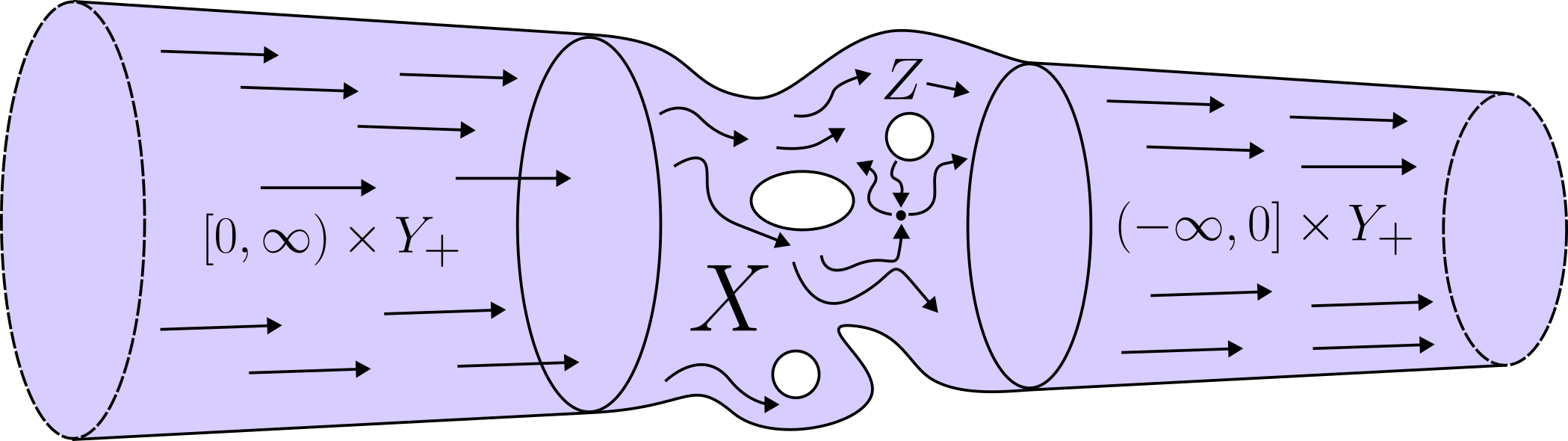}
\caption{A depiction of the completion of a conformal cobordism, equipped with its conformal vector-field.}
\label{fig:conformal_cobordism}
\end{figure}

\noindent Let us briefly point out some of the most important consequences of the existence of $\hat{\omega}$. 

\vspace{3pt}

First, the graph $Y_f \subset \R \times Y$ of a function $f$ on $Y$ has a natural conformal Hamiltonian structure acquired by restricting $\hat{\omega}$ and $\hat{\omega}(\partial_r,-)$. This gives rise to a good theory of families and deformations for conformal Hamiltonian manifolds and conformal cobordisms. This issue is, notoriously, much more subtle for arbitrary stable Hamiltonian structures (cf. \cite{cv2014}).

\vspace{3pt}

Second, the existence of a symplectic structure over the entire completion $\hat{X}$ of a conformal cobordism permits the use of a slightly expanded class of complex structures from the usual one in SFT, which need only be translation invariant in a neighborhood of infinity. The use of this class, and the action of the conformal vector-field on it, is key for demonstrating the continuity properties of our spectral invariants.

\begin{remark} The authors originally set out to formulate spectral gaps for \emph{all} stable Hamiltonian manifolds and symplectic cobordisms. It is interesting to ask if there is a more general class of symplectic objects, with similar properties, for which these constructions work.\end{remark}

\vspace{3pt}

\subsection{ESFT Spectral Gaps} \label{subsec:ESFT_spectral_gaps} The Elementary SFT spectral gaps are constructed using a min-max process over spaces of pseudo-holomorphic curves in symplectic field theory. In SFT, as introduced by Eliashberg-Givental-Hofer \cite{egh2000}, one considers proper, finite energy $J$-holomorphic maps
\begin{equation} \label{eqn:J_curve_intro}
u:(\Sigma,j) \to (\hat{X},J)\end{equation}
from a punctured Riemann surface to the completion of a symplectic cobordism $(X,\Omega)$. The almost complex structure $J$ on $\hat{X}$ is chosen to be compatible with $\Omega$ and cylindrical on the ends of $\hat{X}$, in the appropriate sense. At the punctures, these curves are asymptotic to cylinders over closed orbits of the Reeb vector-fields of $Y_+$ and $Y_-$. Note that the Reeb vector-field $R$ of a stable Hamiltonian manifold $(Y,\omega,\theta)$ is the unique vector-field satisfying
\[\iota_R\theta = 1 \qquad\text{and}\qquad \iota_R\omega = 0\]
The set of Reeb orbits $\Gamma_+$ in $Y_+$ appearing at the punctures of $u$ are called the \emph{positive ends} of $u$. 

\vspace{3pt}

The ESFT spectral gaps of a conformal cobordism $X$ depend on a choice of length $T \ge 0$ and a curve type $\sigma$. A curve type tracks topological information about the maps in the min-max process, and consists of a \emph{genus} $g$, a \emph{point number} $m$ and a homology class $A \in H_2(X,\partial X)$. If $Y_+$ and $Y_-$ have sufficiently nice Reeb flow, the ESFT spectral gap $\g_{\sigma,T}(X)$ is, roughly, given by
\[
\g_{\sigma,T}(X) := \sup_{J,P}\Big( \underset{u}{\on{min}} \; E_\Omega(u) \Big)
\]
The supremum is taken over a certain space of compatible almost complex structures $J$ on $\hat{X}$ and choices of $m$ points $P$ in $\hat{X}$. The minimum is over $(j,J)$-holomorphic maps $u$ as in (\ref{eqn:J_curve_intro}) where
\begin{itemize}
    \item The total period of the orbits in the positive ends of $u$ is bounded by $T$.
    \item The total genus of the components of the domain $\Sigma$ is bounded by $g$.
    \item The relative homology class represented by $u$ is $A$.
    \item The curve contains the $m$ points in $P$ in its image.
\end{itemize}
We provide a cartoon of the maps that could appear in this min-max process in Figure \ref{fig:elementary_spec}.

\begin{figure}[h]
\centering
\includegraphics[width=.8\textwidth]{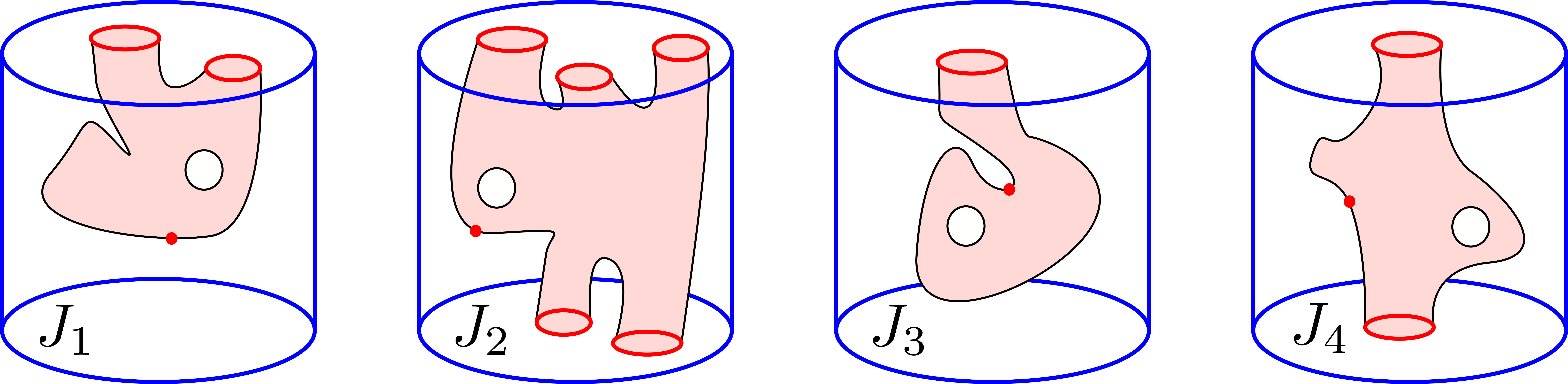}
\caption{A depiction of the maximization process of the ESFT spectral gaps.}
\label{fig:elementary_spec}
\end{figure}

\noindent The ESFT spectral gaps satisfy an extensive list of axioms that we state and prove in Section \ref{sec:spectral_gaps}. In this introduction, we provide only an informal statement of the most important axioms.

\begin{thm*}[Proposition \ref{prop:axioms_of_spectral_gaps}] \label{thm:axioms_intro} The ESFT spectral gaps $\g_{\sigma,T}(X)$ satisfy the following axioms.
\begin{itemize}
\vspace{3pt}
    \item (Hofer Continuity) \label{itm:hofer_intro} The ESFT spectral gap $\g_{\sigma,T}(X)$ varies continuously in $X$ with respect to a version of the Hofer metric for conformal cobordisms, if $T \not\in \LSpec(Y_+)$.
    \vspace{3pt}
    \item (Liouville Embedding) \label{itm:liouville_intro}  If $W \subset X$ is an embedded Liouville domain in $X$, then
    \[c_{\sigma|_W}(W) \le \g_{\sigma,T}(X) \qquad\text{where}\qquad c_\sigma(W) = \lim_{T \to \infty} g_{\sigma|_W,T}(W) \]
    \item (Spectrality) \label{itm:spectrality_intro} The ESFT spectral gap $\g_{\sigma,T}(X)$ is valued in the area spectrum $\ASpec_T(X)$, i.e. the set of all possible areas of curves between Reeb orbits in $\partial X$ of length bounded by $T$.
\end{itemize}\end{thm*}

\vspace{3pt}

\subsection{Applications To Closing Properties} \label{subsec:intro_periodicity_and_closing} Let us now outline our results that relate spectral gaps to closing properties. We use a framework developed by Hutchings \cite{hutchings2022elementary} in the contact case. 

\vspace{3pt}

Any non-negative function $f$ on a conformal Hamiltonian manifold $Y$ has a natural \emph{width} $\wY(Y,f)$, given by the Gromov width $\cGr(W)$ of the region $W \subset \R \times Y$ between the graph $Y_f$ of $f$ and the graph of $0$, denoted simply by $Y$. The \emph{$T$-closing width} of $Y$, denoted by
\[\close^T(Y) \in [0,\infty]\]
is the smallest number with the following property: Any family of  functions
\begin{equation} \label{eqn:intro_f_family} f(t):Y \to [0,\infty) \qquad\text{such that}\qquad f(0) = 0 \quad \text{and}\quad \wY(Y,f(1)) > \close^T(Y)\end{equation}
must have an $s \in [0,1]$ such that $Y_{f(s)}$ has an orbit of period $T$ or less through the union of supports of $f$ (cf. Definition~\ref{def:closing_width}). Our main result on the closing property is an estimate relating the spectral gap of $Y$, given by
\begin{equation} \label{eqn:intro_gap_Y}
\g_{\sigma,T}(Y) := \lim_{\epsilon \to 0} \g_{\sigma,T}([-\epsilon,0] \times Y)
\end{equation}
to the closing width. In this result, we must assume that the curve class $\sigma$ is \emph{pointed}, i.e. that the point number $m$ of $\sigma$ is at least one. The result can now be stated as follows.

\begin{thm*}[Theorem \ref{thm:quant_closing}] \label{thm:quant_closing_intro} Let $Y$ be a conformal Hamiltonian manifold with $[\omega] \in H^2(Y;\Q)$. Then
\[\g_{\sigma,T}(Y) \ge \close^T(Y) \qquad\text{for any pointed curve class }\sigma\]
\end{thm*}

\noindent The proof uses the axioms in Theorem \ref{thm:axioms_intro} in a simple way. We now provide a sketch.

\vspace{3pt}

\emph{Proof Sketch.} Fix a family of functions $f$ as in (\ref{eqn:intro_f_family}) and consider the family of cobordisms $W_s$ given by the region in $\R \times Y$ between the graph $Y_{f(s)}$ of $f(s)$ and the graph $Y_{\smallneg\epsilon}$ of the constant function $-\epsilon$ for small $\epsilon$. By (\ref{eqn:intro_gap_Y}), we have
\[\g_{\sigma,T}(W_0) \sim \g_{\sigma,T}(Y) \qquad\text{if $\epsilon$ is very small.}\]
 Moreover, if an orbit of period $\le T$ does not appear in $Y_{f(s)}$ for any $s$, then the area spectrum $\ASpec(Y_{f(s)})$ is constant in $s$. Thus the Hofer continuity and area spectrality axioms imply that $\g_{\sigma,T}(W_s)$ is continuous in $s$ and valued in a measure $0$ set. It is thus constant in $s$.
\begin{figure}[h]
\centering
\includegraphics[width=\textwidth]{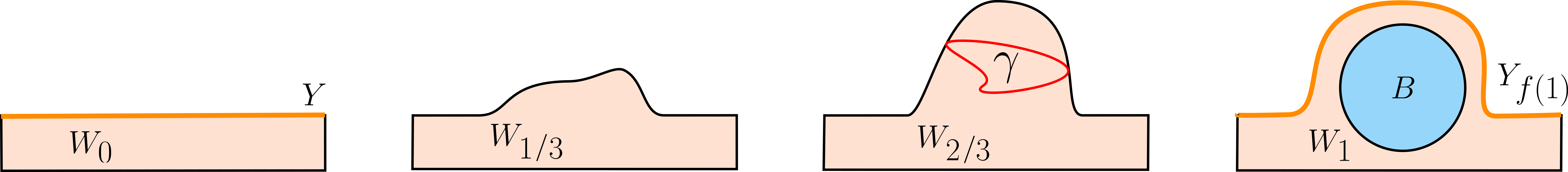}
\caption{A depiction of the famuly of cobordisms in the proof sketch.}
\label{fig:cobordism_family}
\end{figure}

On the other hand, the Liouville embedding theorem, and a simple computation of $\c_\sigma$ for the ball $B$ and a pointed curve class $\sigma$, implies that
\[\g_{\sigma,T}(Y) \sim \g_{\sigma,T}(W_0) = \g_{\sigma,T}(W_1) \ge \cGr(W_1) \ge \wY(Y,f(1))\]
This implies that if $\wY(Y,f(1)) > \g_{\sigma,T}(Y)$, then $Y_{sf}$ must develop an orbit for some $s$. \qed

\vspace{3pt}

Theorem \ref{thm:quant_closing_intro} can be utilized to provide a criterion for a strong version of the smooth closing property. This was originally formulated by Irie \cite[Def. 1.2]{irie2022strong} in the contact setting.

\begin{definition*}[Definition \ref{def:strong_closing_property}] \label{def:strong_closing_intro} A conformal Hamiltonian manifold $Y$ has the \emph{strong closing property} if for every non-zero function $f:Y \to [0,\infty)$, there exists $s \in[0,1]$ such that 
\[
Y_{sf} \qquad \text{has a periodic orbit passing through}\quad \on{supp}(f).
\]
    
\end{definition*}

\begin{thm*}[Corollary \ref{cor:gap_0_and_strong_closing}] \label{thm:stong_closing_intro} Let $Y$ be a conformal Hamiltonian manifold such that
\[[\omega] \in H^2(Y,\Q) \qquad\text{and}\qquad \inf_{\sigma,T} \g_{\sigma,T}(Y) = 0 \quad\text{taken over all points curve classes }\sigma
\]
Then $Y$ satisfies the strong closing property.
\end{thm*}

\begin{remark}[Rationality] The assumption that $[\omega] \in H^2(Y,\Q)$ appears in many results throughout this paper. It holds trivially for contact manifolds. For symplectomorphisms, it translates to the assumption that the symplectic form represents a rational cohomology class in the ambient symplectic, and that the symplectomorphism itself is rational, as in Definition~\ref{def:rational_symplectomorphism}.

\vspace{3pt}

This hypothesis also appears in the works on the closing properties on surfaces \cite{eh2021,cpz2021}, and fails for  Herman's counter example for the $C^\infty$ closing lemma \cite{herman1991exemples,herman1991differentiabilite}. It appears to be necessary for many of our dynamical results.
\end{remark}

\subsection{Hofer Near-Periodicity} We next discuss a large class of conformal Hamiltonian manifolds to which the criterion in Theorem \ref{thm:stong_closing_intro} can be applied.

\begin{definition*}[Definition \ref{def:Hofer_near_periodic}] \label{def:near_periodic_intro} A conformal Hamiltonian manifold $Y$ is \emph{Hofer near-periodic} if there is a sequence of conformal Hamiltonian manifolds $Y_i$ such that
\[\text{$Y_i$ has $T_i$-periodic Reeb flow}\qquad\text{and}\qquad T_i \cdot d_H(Y_i,Y) \to 0\]
\end{definition*}

In Section \ref{sec:apps_and_examples}, we will prove a number of results on the spectral gaps of periodic and near-periodic conformal Hamiltonian manifolds. Some of these results wil use Gromov-Witten invariants for orbifolds as described in \cite{chen2002orbifold}, and neck-stretching in the orbifold setting. To the best of our knowledge the foundations for these constructions have not yet appeared in the literature. Thus we will include a reference for the following assumption, in any result that relies on it.
\begin{assumption}\label{ass:GW_Orbifolds}  We assume the existence of a rigorous construction of the following results.
\begin{itemize}
\item Gromov-Witten invariants on symplectic orbifolds as described by Chen-Ruan \cite{chen2002orbifold}, cf. Proposition~\ref{prop:GW_axioms}.
\vspace{3pt}
\item An SFT neck stretching theorem (cf. \cite[Theorem 10.3]{sftCompactness}) for a closed symplectic orbifold, split along a (possibly disconnected) stable Hamiltonian hypersurface that is disjoint from the orbifold points.
\end{itemize}
\end{assumption}

\begin{thm*}[Theorem \ref{thm:closing_for_Hofer_nearly_periodic}] \label{thm:closing_near_periodic_intro} Let $(Y,\omega,\theta)$ be a Hofer near-periodic, conformal Hamiltonian manifold with $[\omega]\in H^2(Y,\Q)$. Either assume Assumption \ref{ass:GW_Orbifolds} or that $Y$ has conformal constant zero. Then
\[\g(Y) = \inf_{\sigma,T} \mathfrak{g}_{\sigma,T}(Y) =0\]
\end{thm*}

\noindent Theorem \ref{thm:closing_near_periodic_intro} is a relatively simple consequence of a Hofer Lipschitz property that refines the continuity axiom, coupled with the following computation.

\begin{thm*}[Theorem \ref{thm:periodic_spectral_gap}] \label{thm:periodic_gap_intro} Let $Y$ be a conformal Hamiltonian manifold with  with $[\omega]\in H^2(Y,\Q)$ and $T$-periodic Reeb flow. Either assume Assumption \ref{ass:GW_Orbifolds} or that $Y$ has conformal constant zero. Then there is a curve class $\sigma$ with
\[
\g_{\sigma,T}(Y) = 0
\]
\end{thm*}

The basic idea of Theorem \ref{thm:periodic_gap_intro} is very simple. The symplectization of a $T$-periodic Hamiltonian manifold $Y$ with periodic Reeb flow is foliated by cylinders over the orbits of period $T$. These are small-area, $J$-holomorphic with respect to certain translation invariant $J$. Moreover, they survive under variation of the complex structure. Thus, the curve class $\sigma$ with genus $0$ and $1$ point constraint in the homology class of these cylinders have zero gap in period $T$. Unfortunately, carrying out this sketch rigorously is cumbersome and encounters serious technical difficulties. Nevertheless, we will carefully carry out a detailed version in Section \ref{sec:gaps_of_periodic_SHS} to prove a limited version of Theorem \ref{thm:periodic_gap_intro}. 

\vspace{3pt}

To prove Theorem \ref{thm:periodic_gap_intro} in full generality, we instead follow a different strategy. When $Y$ is periodic, it admits an embedding into a certain sphere bundle $M$ with the structure of a symplectic orbifold. We are able to prove Theorem \ref{thm:periodic_gap_intro} by neck stretching certain small area pseudo-holomorphic spheres obtained from Gromov-Witten theory on $M$. When $Y$ has conformal constant zero, one can essentially assume that $M$ is a manifold. However, in the general case, $M$ will be an orbifold and we must use Assumption \ref{ass:GW_Orbifolds}.

\begin{remark}[Fish-Hofer] \label{rmk:hofer_fish} In \cite{fish2022almost}, Fish and Hofer articulated a general philosophy that the $C^\infty$- closing property should hold for closed symplectic hypersurfaces in symplectic manifolds with ``sufficiently rich Gromov-Witten theory.'' Our proof of Theorem \ref{thm:periodic_gap_intro} is very much in the spirit of this philosophy. More generally, our spectral invariants are closely related to Gromov-Witten theory (cf. Proposition \ref{prop:axioms_of_spectral_gaps}\ref{itm:axiom_Gromov-Witten}) and we expect that they will provide a useful tool in reifying Hofer-Fish's philosophy in more general settings.
\end{remark}

\subsection{Examples} There are several nice examples of near-periodic conformal Hamiltonian manifolds in the contact setting and the symplectomorphism setting. 

\begin{example}[Ellipsoids] Let $Y = \partial E$ be the contact boundary of an ellipsoid $E \subset \C^n$. Then $Y$ is near-periodic (see Lemma \ref{lem:ellipsoid_near_periodic}). This recovers the strong closing property for these spaces, a main results of \cite{chaidez2022contact} and \cite{x2022}.
\end{example}

\begin{example}[Contact $\T^2$-Actions] Let $(Y,\xi)$ be a contact manifold with a contact $T^2$-action $T^2 \times Y \to Y$. Suppose that the vector-fields generating the two $S^1$-actions are Reeb vector-fields
\[R_1 \text{ and }R_2 \qquad\text{with contact forms}\qquad \alpha_1\text{ and }\alpha_2\]
Then for any non-zero vector $a \in \R^2$ with non-negative entries, we get a Reeb vector-fied
\[
R_a := a_1\cdot R_1 + a_2 \cdot R_2 \qquad\text{with contact form}\qquad \alpha_a :=  \frac{\alpha_1}{\alpha_1(R_a)}
\]
We show in Lemma \ref{lem:contact_T2_near_periodic} that $\alpha_a$ is always near-periodic. This is non-trivial when $a_1/a_2$ is irrational, and there are examples with finitely many simple closed orbits \cite{agz2018}. \end{example}

\begin{example}[Hamiltonian $\T^k$-Actions] Let $X$ be a closed symplectic manifold of dimension $2n$ with a Hamiltonian $\T^k$-action
\[\phi:\T^k \times X \to X \qquad\text{given by}\qquad (a,x) \mapsto \phi_a(x)\]
Then the symplectomorphism $\phi_a$ is near-periodic for any $a \in \T^k$ (see Proposition \ref{prop:Ham_torus_action_nearly_per}). \end{example}

\begin{remark} At this time, we do not have examples of Hofer near-periodic, conformal Hamiltonian manifolds that are not near-periodic in a stronger, $C^\infty$-topology. It is interesting to ask if there is a rigidity property that prevents such examples.
\end{remark}

Theorem \ref{thm:stong_closing_intro} guarantees that the strong closing property holds for Hamiltonian $\T^k$ actions (Example~\ref{ex:ham_torus_action}). Moreover, under Assumption~\ref{ass:GW_Orbifolds} it guarantees that strong closing property for the above contact examples as well.
Note that the strong closing property can be naturally rephrased in terms of perturbations of the contact form (in the contact case) and Hamiltonian perturbation  (in the symplectomorphism case). We provide a detailed discussion in Section \ref{sec:apps_and_examples}.

\subsection{Towards General Closing Properties} A primary motivation of this work is the development of a general framework for proving smooth closing properties in higher dimensions. Thus, a natural first question is the following. 

\begin{question*} \label{qu:closing_for_spheres} Let $Y$ be the mapping torus of a Hamiltonian diffeomorphism of $\C P^n$. Is it true that
\[
\inf_{\sigma,T} \g_{\sigma,T}(Y) = 0 \qquad\text{with the infimum taken over all curve classes and }T>0?
\]
\end{question*}

The methods of Hutchings \cite{hutchings2022elementary} and Edtmair \cite{edtmair2022elementary} can be used to prove this in dimension three, using neck stretching applied to curves in $\C P^1 \times \C P^1$. The key to this argument is the existence of a sequence of homology classes $A_k$ and genera $g_k$ with the following properties.
\begin{itemize}
\item The area $\Omega \cdot A_k$ with respect to the standard symplectic form on $\C P^1 \times \C P^1$ grows sub-linearly. In other words, the ratio $\Omega \cdot A_k/k$ limits to zero as $k \to \infty$.
\vspace{3pt}
\item For every compatible almost complex structure $J$ on $\C P^1 \times \C P^1$ and every set of $k$ points $P \subset \C P^1 \times \C P^1$, there is a $J$-curve $C$ in $\C P^1 \times \C P^1$ passing through $P$ of genus $g_k$.
\end{itemize}

\noindent These curves can be stretched to produce small area curves in the symplectization of $Y$ with a point constraint, addressing Conjecture \ref{qu:closing_for_spheres}. The existence of these curves relies on certain results in Gromov-Taubes theory (cf. \cite[Section B.2]{edtmair2022elementary}). It is very much along the lines of Fish-Hofer's principle that rich Gromov-Witten theory implies smooth closing properties (see Remark \ref{rmk:hofer_fish}).

\vspace{3pt}

In higher dimensions, there are heuristic calculations with the Fredholm index that suggest that a single symplectic manifold $X$ with a sequence of homology classes and genera as above will not always exist. As a substitute, we pose the following question. 

 \begin{question*}[GW Closing Criterion]\label{qu:GW_closing} Is there a sequence of symplectic manifolds $(X_k,\Omega_k)$ of dimension $2n$, homology classes $A_k$ and genera $g_k$ such that
 \begin{itemize}
     \item The Gromov-Witten invariant counting $J$-holomorphic curves in class $A_k$ with genus $g_k$ passing through a generic $k$ points is non-zero.
     \vspace{2pt}
     \item The area $\Omega_k \cdot A_k$ of $A_k$ grows sub-linearly in $k$, so that $\lim_{k \to \infty} \Omega_k \cdot A_k /k \to 0$.
     \vspace{2pt}
     \item The product $\R/\Z \times \C P^{n-1}$ embeds as a Hamiltonian hypersurface in each $X_k$.
 \end{itemize}
 \end{question*}

\noindent A positive answer to Quesion \ref{qu:GW_closing} would likely provide enough curves to answer Question \ref{qu:closing_for_spheres}, and thus establish the closing property for higher dimensional Reeb flows via Theorem \ref{thm:stong_closing_intro}.

\begin{remark}[Local Arguments] In this paper, we make use of low-regulariy periodic approximations to prove spectral gap vanishing, and thus strong closing properties via Theorem \ref{thm:stong_closing_intro}.

\vspace{3pt}

These types of periodic approximations always exist locally, due to the $C^0$-closing property. Thus, it is natural to ask if our arguments can be carried out locally, providing an alternate route to the general strong closing property for conformal Hamiltonian manifolds. In future work, we plan to explore this possibility more thoroughly.
\end{remark}

\vspace{3pt}

\noindent {\bf Outline.} This concludes the introduction (Section \ref{sec:intro}). The rest of the paper is organized as so. 
\begin{itemize}
    \item In {\bf Section \ref{sec:conformal_manifolds}}, we develop the language of conformal Hamiltonian manifolds and conformal symplectic cobordisms. 
    \item In {\bf Section \ref{sec:SFT}}, we provide a review of symplectic field theory (SFT), with small deviations from the usual setup to fit our setting.
    \item In {\bf Section \ref{sec:spectral_gaps}}, we define the ESFT spectral gaps and develop their many formal properties.
    \item In {\bf Section \ref{sec:apps_and_examples}}, we discuss the relationship between our spectral gaps and various versions of the smooth closing property. We also discuss near-periodicity and provide examples.
    \item In {\bf Section \ref{sec:gaps_of_periodic_SHS}}, we give a proof of a limited version of the vanishing result for spectral gaps of periodic conformal Hamiltonian manifolds (Theorem \ref{thm:periodic_spectral_gap}) that is independent of Assumption \ref{ass:GW_Orbifolds}. 
\end{itemize}

\vspace{3pt}

\noindent {\bf Acknowledgements.}  We would like to thank Helmut Hofer for discussion leading to this project. We would also like to thank Alberto Abbondandolo, Oliver Edtmair, Viktor Ginzburg and Michael Hutchings for helpful conversations. JC was supported by the National Science Foundation under Award No. 2103165. ST was supported by a grant from the Institute for Advanced Study, as well as by the Schmidt futures program.

\section{Conformal Hamiltonian Topology} \label{sec:conformal_manifolds} In this section, we introduce the class of conformal Hamiltonian manifolds and conformal symplectic cobordisms, and discuss their important properties.

\subsection{Symplectic Cobordisms} \label{subsec:symplectic_cobordisms} We begin by discussing general (stable) Hamiltonian manifolds and symplectic cobordisms (cf. Wendl \cite{wendlsymplectic} or Celiebak-Volkov \cite{cv2014}).

\begin{definition} \cite{wendlsymplectic} A \emph{Hamiltonian manifold} $(Y,\omega)$ is a closed odd-dimensional manifold with a maximally non-degenerate, closed 2-form $\omega$. A \emph{stabilizing $1$-form} $\theta$ is a $1$-form satisfying
\[\ker \omega \subset \ker(d\theta)  \quad\text{ and }\quad  \theta|_{\ker(\omega)}>0.\]
A Hamiltonian manifold $(Y,\omega)$ equipped with a stabilizing $1$-form $\theta$ is called \emph{stable}. Finally, the \emph{symplectization} of $Y$ is simply $\R \times Y$.
\end{definition}

\begin{remark} Note that the conventional terminology \emph{symplectization} is slightly inappropriate, since $\R\times Y$ has no canonical symplectic structure in general. 
\end{remark}

\begin{definition}\label{def:symp_cob} \cite{wendlsymplectic} A \emph{symplectic cobordism} $X$ of stable Hamiltonian manifolds, denoted
\[(X,\Omega,\Theta):(Y_+,\omega _+,\theta_+) \to (Y_-,\omega_-,\theta_-) \qquad\text{or more simply}\qquad X:Y_+ \to Y_-\]
consists of the \emph{ends} $Y_\pm$ and the following data.
\begin{itemize}
    \item a compact symplectic manifold $(X,\Omega)$ with boundary $\partial X$ decomposed as $\partial_+X \sqcup -\partial_-X$.
    \vspace{3pt}
    \item a pair of stabilizing $1$-forms $\Theta_\pm$ along $\partial_\pm X$.
    \vspace{3pt}
    \item a pair of isomorphisms of stable Hamiltonian manifolds of the form
    \[\iota_\pm:(Y_\pm,\omega_\pm,\theta_\pm) \simeq (\partial_\pm X,\Omega|_{\partial_\pm X},\Theta_\pm).\]
\end{itemize}
The \emph{completion} $\hat{X}$ of a symplectic cobordism $(X,\Omega,\Theta)$ is the cylindrical manifold acquired by attaching half-cylinders over the ends to $X$ via the isomorphisms $\iota$.
\[
\hat{X} := (-\infty,0] \times Y_- \; \cup_\iota \; X \; \cup_\iota \; [0,\infty) \times Y_+
\]
\noindent In general, there is no natural symplectic form on the completion. 
\end{definition}

\begin{example}[Trivial Cobordisms] \label{ex:trival_cob} Let $(Y,\omega,\theta)$ be a stable Hamiltonian manifold. Fix a small increasing function
\[\phi:[\delta,\epsilon] \to \R \qquad\text{on a small interval}\qquad [\delta,\epsilon]\]
The \emph{trivial symplectic cobordism} $W = [\delta,\epsilon]_r \times Y$ associated to $Y$ and $\phi$ is given by
\[
W:(Y,\omega + \phi(\epsilon) \cdot d\theta, \phi'(\epsilon) \cdot \theta) \to (Y,\omega + \phi(\delta) \cdot d\theta, \phi'(\delta) \cdot \theta)\]
There is a symplectic form $\Omega$ on this space given by
\[\Omega = d(\phi(r)\theta) + \omega.\]
\end{example}

There are two fundamental operations on symplectic cobordisms that we will make repeated use of throughout this paper: disjoint union and trace.

\begin{definition} The \emph{disjoint union} $X \sqcup X'$ of a pair of symplectic cobordisms $X$ and $X'$, denoted
\[X \sqcup X':Y_+ \sqcup Y'_- \to Y'_+ \sqcup Y'_-\]
is given by the usual disjoint union, equipped with the symplectic form and stabilizing $1$-forms that restrict to the corresponding structures on $X$ and $X'$. \end{definition}

\begin{definition} The \emph{trace} $(X)_M$ of a symplectic cobordism $X:Y_+ \to Y_-$ along a closed stable Hamiltonian manifold $M = (M,\omega,\theta)$ equipped with isomorphisms
\[
\phi_\pm:M \simeq M_\pm \quad\text{for a union of components }M_\pm \subset Y_\pm \simeq \partial_\pm X
\]
is smooth quotient of $X$ by the relation $\phi_+(x) \sim \phi_-(x)$ for all $x \in M$. This yields a cobordism
\[(X)_M:Y_+ \setminus M_+ \to Y_- \setminus M_-\]
Any Darboux chart near $M$ and $\phi(M)$ induced by their stable Hamiltonian structures \cite{wendlsymplectic} yields a smooth structure on $(X)_M$, and there is a smooth quotient map
\[
X \to (X)_M
\]The symplectic form on $(X)_M$ and stabilizing $1$-forms on $\partial(X)_M$ descend from this quotient map. \end{definition}

As a special case of trace, we have the composition of two symplectic cobordisms.

\begin{definition} The \emph{composition} $X \circ X':Y_+ \to Y_-$ of a pair of symplectic cobordisms
\[X:Y_+ \to M \qquad\text{and}\qquad X':M \to Y_-\]
is simply the trace $(X \sqcup X')_M$ along the common boundary $M_+ = M_- = M$.\end{definition}

\subsection{Reeb Flows And Non-Degeneracy} A stable Hamiltonian manifold has a natural dynamical system called the Reeb flow, generated by the Reeb vector-field.

\begin{definition} The \emph{Reeb vector-field} $R$ of a stable Hamiltonian manifold $(Y,\omega,\theta)$ is defined by
\[\iota_R\omega = 0 \quad \text{and}\quad \iota_R\theta = 1.\]
A Reeb \emph{orbit set} is an unordered list of closed Reeb orbits (with multiple covers and repetitions).
\[\Gamma = (\gamma_1,\dots,\gamma_m)\]
\end{definition}

\begin{remark} Note that this differs slightly from the terminology in other settings (cf. \cite{hutchings2022elementary}). \end{remark}

Reeb orbit sets have  several numerical invariants that appear frequently in SFT. The first is simply the total period of all the constituent orbits. 

\begin{definition} The \emph{length} $\mathcal{L}(\Gamma)$ of an orbit set $\Gamma$ in a stable Hamiltonian manifold $(Y,\omega,\theta)$ is
\[
\mathcal{L}(\Gamma) = \sum_i\big( \int_{\gamma_i} \theta \big) \qquad\text{if}\qquad \Gamma = (\gamma_1,\dots,\gamma_n)
\]
The \emph{length spectrum} $\LSpec(Y) \subset (0,\infty)$ of $Y$ is the set of lengths of all Reeb orbit sets $\Gamma$.
\end{definition}

\begin{lemma}[Closed Length Spectrum] \label{lem:spec_closed_under_limits} Let $Y$ be a closed manifold and consider a sequences
\[\text{stable Hamiltonian structures }(\omega_i,\theta_i) \text{ on }Y \qquad\text{and}\qquad T_i \in \LSpec(Y,\omega_i,\theta_i)\]
Suppose that $(\omega_i,\theta_i) \to (\omega,\theta)$ in the $C^\infty$-topology and $T_i \to T$. Then
\[T\in \LSpec(Y,\omega,\theta).\]
In particular, the length spectrum $\LSpec(Y,\omega,\theta)$ of $(Y,\omega,\theta)$ is a closed subset of $\R$.
\end{lemma}
\begin{proof}
    Let $\gamma_i$ be  $T_i$-periodic orbits of the Reeb vector field $R_i$ corresponding to $(\omega_i,\theta_i)$, and reparametrize them as 
    \[
    \gamma_i:\R/\Z\rightarrow Y, \quad \dot\gamma_i = T_i\cdot R_i.
    \]
    Note that our assumption that $(\omega_i,\theta_i)\xrightarrow{C^\infty} (\omega,\theta)$ implies that $R_i\xrightarrow{C^\infty}R$. By Arzel\`a-Ascoli theorem there exists a subsequence of $\gamma_i$ converging uniformly to $\gamma:\R/\Z\rightarrow Y$.  We claim that $\gamma$ is a (reparametrization of a) $T$-periodic orbit of $R$. Indeed, let $t\in \R/\Z$ then for all $\varepsilon>0$,
    \[
    \gamma(t+\varepsilon)-\gamma(t) = \lim_i \gamma_i(t+\varepsilon)-\gamma_i(t)  = \lim_i \int_0^{\varepsilon} T_i\cdot R_i\circ \gamma_i(t+\tau) d\tau = \int_0^\varepsilon T\cdot R\circ \gamma(t+\tau)d\tau.
    \]
    Therefore $\dot\gamma = T\cdot R$ and, in particular, $T\in \LSpec(Y,\omega,\theta)$. The last assertion of the lemma follows from taking $(\omega_i,\theta_i)=(\omega,\theta)$. 
\end{proof}

\noindent There is also an important integer index invariant associated to any closed orbit that derives from the linearized Poincar\'{e} return map.

\begin{definition}[Index] The \emph{Robbin-Salamon index} $\indRS^\tau(\Gamma)$ of Reeb orbit set $\Gamma$ and a symplectic bundle trivialization $\tau$ of $\on{ker}(\theta)$ along $\Gamma$ is the sum
\[
\indRS^\tau(\Gamma) = \sum_i \indRS(\Psi_i)
\]
Here $\indRS(\Psi_i)$ is the Robbin-Salamon index \cite[Def. 46]{gutt2014generalized} of the path of symplectic matrices
\[
\Psi_i:[0,T] \to \Sp(2n) \qquad\text{where}\qquad \Psi_i(t) = \tau(\Phi(t,x)) \circ T\Phi(t,x) \circ \tau^{-1}(x) \;\;\text{and}\;\; x = \gamma_i(0)
\]
\end{definition}

There are a few important notions of non-degeneracy for Reeb flows that are important in the context of symplectic field theory. 

\begin{definition}[Morse-Bott] \cite{b2002} \label{def:Morse_Bott_SHS} A stable Hamiltonian structure $(\omega,\theta)$ on a manifold $Y$ is \emph{Morse-Bott} if the length spectrum $\LSpec(Y,\omega,\theta)$ is discrete and, for each period $T > 0$, we have
\begin{itemize}
    \item The set $N_T \subset Y$ of points lying on all closed Reeb orbits of period $T$ is a closed sub-manifold. 
    \item The tangent space of $N_T$ is given by the $1$-eigenspace of the time $T$ linearized flow.
    \[TN_T = \on{ker}(T\Phi_T(p) - \on{Id})\]
    \item The rank of the restricted $2$-form $\omega|_{TN_T}$ is locally constant on $N_T$.
\end{itemize}
A symplectic cobordism $X:Y_+ \to Y_-$ is \emph{Morse-Bott} (or has \emph{Morse-Bott ends}) if the stable Hamiltonian manifolds $Y_+$ and $Y_-$ are Morse-Bott. 

\vspace{3pt}

Each of the manifolds $N_T$ admit an $S^1$-action by the Reeb flow. We adopt the notation
\[
S_T = N_T/S^1
\]
A \emph{Morse-Bott family} $S$ of orbit sets is a connected component of the product manifold
\[
S_{T_1} \times \dots \times S_{T_m} \qquad\text{for some set of periods}\qquad T_1,\dots,T_m
\]
\end{definition}

\begin{definition}[Non-Degeneracy] A stable Hamiltonian structure $(\omega,\theta)$ on a manifold $Y$ is \emph{non-degenerate} if any point $p$ on a closed orbit of period $T$ satisfies
\[
\on{ker}(T\Phi_T(p) - \on{Id}) = \on{ker}(\omega)
\]
Equivalently, $(\omega,\theta)$ is Morse-Bott and every Morse-Bott family is a single orbit. \end{definition}

In this paper, we will extensively study stable Hamiltonian manifolds with periodic Reeb flow, which are natural examples that satisfy the Morse-Bott condition.

\begin{example} A stable Hamiltonian manifold $(Y,\omega,\theta)$ is \emph{periodic} with period $T$ if  the Reeb flow $\Phi:\R \times Y \to Y$ satisfies
\begin{equation} \label{eqn:periodic_SHS} \Phi_T = \on{Id}:Y \to Y \qquad\text{and}\qquad \Phi_S \neq \on{Id} \text{ if }0 < S < T\end{equation} 
\end{example}

\subsection{Symplectic Area And Spectra} \label{subsec:area} There is a relative version of symplectic area of homology classes in stable Hamiltonian manifolds and symplectic cobordisms. We now discuss the theory and properties of this relative area in detail.

\subsubsection{Surface Classes} We start by discussing the types of homology classes that possess a well-defined symplectic area. Let $M$ be an oriented manifold with boundary and fix two immersions $\Gamma_+ \to M$ and $\Gamma_- \to M$ of closed, oriented $1$-manifolds $\Gamma_+$ and $\Gamma_-$. 

\begin{definition}[Surface Classes] A \emph{surface class} $A:\Gamma_+\to \Gamma_-$ in $M$ is a relative homology class
\[A \in H_2(M,\Gamma_+ \cup \Gamma_-) \qquad\text{such that}\qquad \partial A = [\Gamma_+] - [\Gamma_-] \in H_1(\Gamma_+) \oplus H_1(\Gamma_-)\]
Here $[\Gamma_\pm] \in H_1(M)$ are the fundamental classes of the immersed sub-manifolds $\Gamma_\pm$. The set of surface classes from $\Gamma_+$ to $\Gamma_-$ is denoted by
\[S(M;\Gamma_+,\Gamma_-) \subset H_2(M,\Gamma_+ \cup \Gamma_-)\]
A smooth map $\iota:M \to M'$ such that there is a pair of oriented diffeomorphisms $\Gamma_\pm \simeq \Gamma_\pm'$ that intertwine the immersions $\Gamma_\pm \to M$ and $\Gamma'_\pm \to M'$ induces a pushforward map
\[\iota_*:S(M;\Gamma_+,\Gamma_-) \to S(M';\Gamma_+',\Gamma_-')\] \end{definition}

There are a number of important operations on surface classes. Given a surface class $A:\Gamma_+ \to \Gamma_- \sqcup \Xi$ and a surface class $B:\Xi_+ \sqcup \Xi \to \Xi_-$ that share a common end $\Xi$, the \emph{gluing}
\[A \cup_\Xi B:\Gamma_+ \sqcup \Xi_+ \to \Gamma_- \sqcup \Xi_-\]
is acquired by gluing the homology classes along $\Xi$ using the exact sequence in relative homology. In particular, one can compose surface classes $A:\Gamma_+ \to \Xi$ and $B:\Xi \to \Gamma_-$ to acquire a \emph{composition}
\[B \circ A:\Gamma_+ \to \Gamma_-\]
Moreover, given surface classes $A:\Gamma_+ \to \Gamma_-$ and $B:\Xi_+ \to \Xi_-$, the \emph{union}
\[A + B:\Gamma_+ \cup \Xi_+ \to \Gamma_- \cup \Xi_-\]
is inherited from the sum in $H_2(X,\Gamma_+ \cup \Xi_+ \cup \Gamma_- \cup \Xi_-)$. Note that $H_2(M) = S(M;\emptyset,\emptyset)$ and union yields a free, transitive $H_2(M)$-action on $S(M;\Gamma_+,\Gamma_-)$. 

\vspace{3pt}

There is also a well-defined pairing operation extending the pairing with cohomology. 

\begin{definition} \label{def:integral_pairing} The \emph{integral pairing} of a closed $2$-form $\omega$ with a class $A:\Gamma_+ \to \Gamma_-$ is given by
\begin{equation} \label{eqn:area_map_general}
\omega \cdot A := \int_S \jmath^*\omega
\end{equation}
Here $\jmath:S \to M$ is an smooth map from a surface with $\iota(\partial S) = \Gamma_+ \cup -\Gamma_-$ and $\iota_*[S] = A$. The pairing is compatible with the $H_2$-action, pushforward, composition and union as follows.
\[\omega \cdot \iota_*A =  \iota^*\omega \cdot A \]
\[\omega \cdot (B \cup_\Xi A) = \omega\cdot A + \omega \cdot B\qquad\text{and}\qquad \omega \cdot (A + B) = \omega \cdot A + \omega \cdot B \] \end{definition}

\subsubsection{Area In Hamiltonian Manifolds} We next discuss area and the area spectrum for surface classes in stable Hamiltonian manifolds.

\begin{definition}[Area, Hamiltonian Manifolds] \label{def:area_SHS_surface_class} The \emph{area} of a surface class $A:\Gamma_+ \to \Gamma_-$ in a stable Hamiltonian manifold $(Y,\omega,\theta)$ is the integral pairing
\[\omega \cdot A\]
The \emph{area spectrum} $\ASpec(Y;\Gamma_+,\Gamma_-)$ relative orbit sets $\Gamma_+$ and $\Gamma_-$ is the set
\[\ASpec(Y;\Gamma_+,\Gamma_-) := \{\omega \cdot A \; : \; A \in S(Y;\Gamma_+,\Gamma_-)\}\] 
The \emph{area spectrum} $\ASpec_T(Y)$ of $Y$ below length $T$ is the set
\[\ASpec_T(Y) :=  \{\omega \cdot A \; : \; A \in S(Y;\Gamma_+,\Gamma_-) \;,\; \mathcal{L}(\Gamma_-) \le T \; \text{and}\; \mathcal{L}(\Gamma_+) \le T\}\] \end{definition}

We will require many important properties of the area spectra as a subset of $\R$ (e.g. meagerness, closedness, etc.). We will now carefully establish these properties.

\begin{lemma}[Meager] \label{lem:area_meager_manifold} Let $(Y,\omega,\theta)$  be a closed stable Hamiltonian manifold. Then
\[\ASpec(Y;\Gamma_+,\Gamma_-) \text{ is countable for any $\Gamma_\pm$}\qquad\text{and}\qquad\ASpec_T(Y) \subset \R\text{ is meager}\]
\end{lemma}

\begin{proof} For the first claim, the set of surface classes $S(Y;\Gamma_+,\Gamma_-)$ admits a free transitive $H_2(Y;\Z)$ action. Thus it is countable and so is the area spectrum $\ASpec(Y;\Gamma_+,\Gamma_-)$. For the second claim, we use a reduction to Sard's theorem, adapting a proof of Edtmair \cite{edtmair2022elementary} and Oh.

\vspace{3pt}

We start by working locally near an orbit. Fix a closed orbit $\gamma$ of covering multiplicity $m$, and choose an embedding $\iota:B \to Y$ of the open ball $B$ of dimension $\on{dim}(Y) - 1$ into $Y$ such that
\begin{equation} \label{eq:meager_ASpec_1} \iota(B) \cap \gamma = \iota(0) \qquad\text{and}\qquad \iota \text{ is transverse to $R$}\end{equation}
Let $\Sigma \subset B$ be the open set of points $z \in B$ with the following property: there is a time $T_z > 0$ such that the Reeb arc
\[\eta_z:[0,T_z] \to Y \qquad\text{with}\qquad \eta_z(0) = \iota(z)  \]
intersects $\iota(B)$ exactly $m+1$ times and  $\eta_z(T_z)\in\iota(B)$. Let $U \subset Y$ be the open set of points whose Reeb flowline intersects $\iota(\Sigma)$. By choosing $\iota$ so that the image $\iota(B)$ is small enough, we may assume that
\begin{equation} \label{eq:meager_ASpec_2} U \text{ is contained in a tubular neighborhood $N$ of }\gamma\end{equation}
Let $\lambda$ be a primitive for $\omega$ on $N$. Then we may define a function
\begin{equation} \label{eq:meager_ASpec_3}
f:\Sigma \to \R \qquad\text{given by}\qquad f(z) = \int_{\iota \circ \kappa_z} \lambda + \int_{\eta_z} \lambda - \int_\gamma \lambda
\end{equation}
Here $\kappa_z$ is the straight line segment in $B$ connecting end point of $\eta_z$ in $B$ to $z$. A straightforward calculation shows that
\begin{equation} \label{eq:meager_ASpec_4}
df_z = 0 \qquad\text{if {and only if}}\qquad \eta_z \text{ is a closed Reeb orbit in $U$ homotopic to $\gamma$ within $N$}
\end{equation}
For a $z$ satisfying (\ref{eq:meager_ASpec_4}), there is a unique surface class $Z$ in $S(N;\eta_z,\gamma)$ with area
\[
\omega \cdot Z = f(z)
\]
Thus, locally near any orbit, the set of areas is a set of critical values and is therefore meager.

\vspace{3pt}

Next, note that the space of closed orbits of period $T$ or less is compact in the $C^0$-topology. Therefore, we may choose a finite set of orbits
\[\Gamma = (\gamma_1,\dots,\gamma_l)\]
such that every closed orbit $\gamma$ of period $T$ or less arises from a critical point of the function $f_i$ as in (\ref{eq:meager_ASpec_3}). We let $V_i$ denote the set of critical values of $f_i$.

\vspace{3pt}

Finally, we consider the area of an arbitrary surface class $A:\Gamma_+ \to \Gamma_-$ beween orbit sets
\[\Gamma_+ = (\gamma^+_1,\dots,\gamma^+_m)\qquad\text{and}\qquad \Gamma_- = (\gamma^-_1,\dots,\gamma^-_n)\]
with length bounded by $T$. Each orbit $\gamma^\pm_i$ in $\Gamma_\pm$ can be connected to an orbit $\gamma_j$ in $\Gamma$ by a small cylinder $Z^\pm_i$ such that $\omega \cdot Z^\pm_i$ is a critical value of $f_j$. Their union yields surface classes
\[
Z_+:\Gamma_+ \to \Xi_+ \qquad\text{and}\qquad Z_-:\Xi_- \to \Gamma_-
\]
where $\Xi_+$ and $\Xi_-$ are orbit sets consisting only of orbits in $\Gamma$. Note that the number of orbits in $\Gamma_\pm$ (and thus in $\Xi_\pm$) is bounded by
\[N = \lfloor T/\on{min}(\LSpec(Y))\rfloor\]
In particular, the number of possible orbit sets $\Xi_+$ and $\Xi_-$ is finite. We let $U$ denote the union of $\ASpec(Y;\Xi_+,\Xi_-)$ over all of these orbit sets. Note that this set is countable. We now write
\[
A = Z_- \circ B \circ Z_+ \qquad\text{for a surface class }B:\Xi_+ \to \Xi_-
\]
Thus the area of $A$ is an element of the Minkowski sum $U + V$ where $V$ is the set
\[
V = \bigcup_k \Big(\sum_j k_j \cdot V_j\Big)
\]
where the union is over all tuples $k = (k_1,\dots,k_l)$ of non-negative numbers with $k_1 + \dots + k_l \le N$. It is an elementary fact that if $A$ and $B$ are critical value sets, then so are $A + B, A \cup B$ and $a \cdot B$ for any $a \in \R$. Thus, $V$ is a set of critical values and is meager. The Minkwoski sum $U+V$ of the countable set $U$ and the meager set $V$ is also meager. Thus, $\omega \cdot A \in U + V$ implies that
\[\ASpec_T(Y) \subset U + V \quad\text{ and $\ASpec_T(Y)$ is meager}\qedhere\]
\end{proof}

\begin{lemma}[Discrete] \label{lem:area_discrete_manifold} Let $(Y,\omega,\theta)$ be a stable Hamiltonian manifold with $[\omega] \in H^2(Y;\Q)$. Then
\[\ASpec(Y;\Gamma_+,\Gamma_-) \text{ is discrete for any $\Gamma_\pm$}\]
\end{lemma}

\begin{proof}  Fix a surface class $A_0:\Gamma_+ \to \Gamma_-$ of area $c$ and an integer $m$ such that $\omega \cdot B \in \frac{1}{m}\Z$ for any $B \in H_2(Y)$. Any surface class $A:\Gamma_+ \to \Gamma_-$ is given by $A = A_0 + B$ for some $B \in H_2(Y)$. Thus
\[
\omega \cdot A = \omega \cdot A_0 + \omega \cdot B \qquad\text{and therefore}\qquad \ASpec(Y;\Gamma_+,\Gamma_-) \subset c + \frac{1}{m}\Z \qedhere
\]
\end{proof}

\begin{lemma}[$C^\infty$-Stable] \label{lem:area_stable_manifold} Let $Y$ be a closed manifold and fix the following sequences.
\begin{itemize}
    \item A sequence of stable Hamiltonian structures $(\omega_i,\theta_i)$ converging to $(\omega,\theta)$ in the $C^\infty$-topology.
    \vspace{2pt}
    \item A sequence of lengths $T_i$ converging to $T \in \R$.
    \vspace{2pt}
    \item A sequence of areas $C_i$ in $\ASpec_{T_i}(Y,\omega_i,\theta_i)$ converging to $C \in \R$.
\end{itemize}
Finally, assume that $[\omega] \in H_2(Y;\Q)$. Then $C$ is in the area spectrum $\ASpec_T(Y)$. \end{lemma}

\begin{proof} We consider sequences of Reeb orbit sets $\Gamma^+_i$ and $\Gamma^-_i$ in $(Y,\omega_i,\theta_i)$ with length bounded by $T_i$ and a sequence of surface classes
\[
A_i:\Gamma^+_i \to \Gamma^-_i \qquad\text{with}\qquad C_i = \omega \cdot A_i
\]
The Reeb vector fields $R_i$ of $(\omega_i,\theta_i)$ converge smoothly to the Reeb vector field of $(\omega,\theta)$. Thus, after passing to a subsequence and applying Arzel\'{a}–Ascoli, we may assume that
\[\Gamma^+_i = (\gamma^+_{1,i} \dots \gamma^+_{m,i}) \xrightarrow{C^\infty} (\gamma^+_1 \dots \gamma^+_m) = \Gamma_+ \qquad\text{and}\qquad \Gamma^-_i = (\gamma^-_{1,i} \dots \gamma^-_{n,i}) \xrightarrow{C^\infty} (\gamma^-_1 \dots \gamma^-_n) = \Gamma_-\]
For large $i$, we have $C^\infty$-small union of cylinders connecting the orbits of $\Gamma^\pm_i$ to the orbits of $\Gamma_\pm$. This may be viewed as a pair of surface classes
\[Z^+_i:\Gamma_+ \to \Gamma^+_i \qquad\text{and}\qquad Z^-_i:\Gamma^-_i \to \Gamma_i \qquad \text{with}\qquad \omega_i \cdot Z^\pm_i \to 0\]
Let $B_i:\Gamma_+ \to \Gamma_-$ be the composition $Z^-_i \circ A_i \circ Z^+_i$. Then we have
\[
\lim_{i \to \infty} \omega \cdot B_i = \lim_{i \to \infty} \omega_i \cdot B_i = \lim_{i \to \infty} \omega_i \cdot Z^-_i +  \lim_{i \to \infty} \omega_i \cdot A_i +  \lim_{i \to \infty} \omega_i \cdot Z^+_i = C
\]
Thus $C \in \ASpec(Y;\Gamma_+,\Gamma_-) \subset \ASpec_T(Y)$ by Lemma \ref{lem:area_discrete_manifold}.\end{proof}

\begin{lemma}[Morse-Bott] \label{lem:ASpec_MB_manifold} Let $(Y,\omega,\theta)$ be a Morse-Bott stable Hamiltonian manifold. Then
\begin{itemize}
\item For any Morse-Bott families $S_+$ and $S_-$ of orbit sets and any orbit sets $\Gamma_\pm \in S_\pm$, we have
\[\ASpec(Y;S_+,S_-) := \ASpec(Y;\Gamma_+,\Gamma_-) \qquad\text{depends only on $S_+$ and $S_-$}\]
\item The area spectrum $\ASpec_T(Y)$ is discrete.
\end{itemize}\end{lemma}

\begin{proof} For the first claim, let $\Gamma_\pm$ and $\Gamma'_\pm$ be different choices of orbit sets on the ends. By choosing paths in $S_\pm$ connecting $\Gamma_\pm$ and $\Gamma'_\pm$ and lifting them to $Y$, we acquire surface classes
\[Z_+:\Gamma_+ \to \Gamma'_+ \quad\text{and}\quad Z_-:\Gamma'_- \to \Gamma_- \qquad\text{with}\qquad \omega \cdot Z_\pm = 0\]
Thus if $A:\Gamma'_+ \to \Gamma'_-$ is a surface class then $Z_- \circ A \circ Z_+:\Gamma_+ \to \Gamma_-$ has the same area. Thus $\ASpec(Y;\Gamma_+,\Gamma_-) \subset \ASpec(Y;\Gamma_+',\Gamma_-')$, and a symmetric argument proves the opposite inclusion.

\vspace{3pt}

For the second claim, $\ASpec_T(Y)$ is a union of $\ASpec(Y;S_+,S_-)$ over all Morse-Bott families $S_+$ and $S_-$ of period less than $T$. Since there are only finitely many such families,  $\ASpec_T(Y)$ is discrete. \end{proof}

\subsubsection{Area In Symplectic Cobordisms} We finally discuss area and the area spectrum for surface classes in symplectic cobordisms.

\begin{definition}[Area, Symplectic Cobordisms] \label{def:area_cobordism_surface_class} The \emph{area} of a surface class $A:\Gamma_+ \to \Gamma_-$ in a symplectic cobordism $(X,\ \Omega,\Theta)$ is the integral pairing
\[\Omega \cdot A\]
The \emph{area spectrum} $\ASpec(Y;\Gamma_+,\Gamma_-)$ relative to orbit sets $\Gamma_+$ and $\Gamma_-$ is the set
\[\ASpec(X;\Gamma_+,\Gamma_-) := \{\Omega \cdot A \; : \; A \in S(X;\Gamma_+,\Gamma_-)\}\] 
The \emph{area spectrum} $\ASpec_T(X)$ of $X$ below length $T$ is the set
\[\ASpec_T(X) :=  \{\Omega \cdot A \; : \; A \in S(X;\Gamma_+,\Gamma_-) \;,\; \mathcal{L}(\Gamma_-) \le T \; \text{and}\; \mathcal{L}(\Gamma_+) \le T\}\] \end{definition}

The area spectra of a symplectic cobordisms satisfy analogues of the properties in the Hamiltonian case. We now state and prove these properties.

\begin{lemma}[Meager] \label{lem:area_meager_cobordism} Let $(X,\Omega,\Theta)$  be a symplectic cobordism. Then
\[\ASpec(X;\Gamma_+,\Gamma_-) \text{ is countable for any $\Gamma_\pm$}\qquad\text{and}\qquad\ASpec_T(X) \subset \R\text{ is meager}\]
\end{lemma}

\begin{proof} For the first claim, we simply note that the countability of $\ASpec(X;\Gamma_+,\Gamma_-)$ follows from the countability of $S(X;\Gamma_+,\Gamma_-)$, which is in bijection with $H_2(X)$. 

\vspace{3pt}

For the second claim, we use a reduction to Lemma \ref{lem:area_meager_manifold}. First, note that the space of periodic orbits of period bounded by $T$ is compact (see Lemma \ref{lem:spec_closed_under_limits}). It follows that the set of homology classes in $H_1(\partial X)$ represented by such orbits is finite. Thus we may choose a finite set of orbits
\[\Gamma = (\gamma_1,\dots,\gamma_l) \qquad\text{with}\qquad \mathcal{L}(\gamma_i) \le T\]
such that every closed orbit $\gamma$ of period $T$ or less in $\partial X$ is homologous to an orbit in $\Gamma$. Next, consider the area of an arbitrary surface class $A:\Gamma_+ \to \Gamma_-$ in $X$ between orbit sets
\[\Gamma_+ = (\gamma^+_1,\dots,\gamma^+_m)\text{ in }Y_+\qquad\text{and}\qquad \Gamma_- = (\gamma^-_1,\dots,\gamma^-_n)\text{ in }Y_-\]
with length bounded by $T$. By construction of $\Gamma$, there are surface classes
\[
Z_+:\Gamma_+ \to \Xi_+ \qquad\text{and}\qquad Z_-:\Xi_- \to \Gamma_-
\]
where $\Xi_+$ and $\Xi_-$ are orbit sets consisting only of orbits in $\Gamma$. Note that the number of orbits in $\Gamma_\pm$ (and thus in $\Xi_\pm$) is bounded by
\[N = \lfloor T/\on{min}({\LSpec(\partial X)})\rfloor\]
In particular, the number of possible orbit sets $\Xi_+$ and $\Xi_-$ consisting of $N$ or fewer orbits in $\Gamma$ is finite. We denote the union of the area spectra of $X$ relative to $\Gamma_+$ and $\Gamma_-$ by
\[U = \bigcup_{\Xi_\pm} \ASpec(X;\Xi_+,\Xi_-)  \]
This set is a finite union of countable sets, and is thus countable. Moreover, the total number of orbits in $\Xi_+$ and $\Xi_-$ is bounded by $2N$, and thus
\[
\omega_+ \cdot Z_+ + \omega_- \cdot Z_- \in \ASpec_{2NT}(\partial X)
\]
If $B:\Xi_+ \to \Xi_-$ be the curve class $(-Z_-) \circ A \circ (-Z_+)$, then we have $A = Z_- \circ B \circ Z_+$. Therefore
\[
\Omega \cdot A =  \Omega \cdot B + \omega_- \cdot Z_- + \omega_+ \cdot Z_+ \quad\text{and thus}\quad \ASpec_T(X) \subset U + \ASpec_{2NT}(\partial X) 
\]
Since $U$ is countable and $\ASpec_{2NT}(\partial X)$ is meager by Lemma \ref{lem:area_meager_manifold}, the Minkowski sum is meager. Thus $\ASpec_T(X)$ is meager. \end{proof}

The proofs of the remaining properties are identical to their stable Hamiltonian counterparts (Lemmas \ref{lem:area_discrete_manifold}-\ref{lem:ASpec_MB_manifold}) and therefore we omit them.

\begin{lemma}[Discrete] \label{lem:area_discrete_cobordism} Let $(X,\Omega,\Theta)$ be a symplectic cobordism with $[\Omega] \in H^2(X;\Q)$. Then
\[\ASpec(X;\Gamma_+,\Gamma_-) \text{ is discrete for any $\Gamma_\pm$}\]
\end{lemma}

\begin{lemma}[$C^\infty$-Stable] \label{lem:area_stable_cobordism}  Let $X$ be a smooth cobordism and fix the following sequences.
\begin{itemize}
    \item A sequence of symplectic cobordism structures $(\Omega_i,\Theta_i)$ converging to $(\Omega,\Theta)$ in the $C^\infty$-topology.
    \vspace{2pt}
    \item A sequence of lengths $T_i$ converging to $T \in \R$.
    \vspace{2pt}
    \item A sequence of areas $C_i$ in $\ASpec_{T_i}(X,\Omega_i,\Theta_i)$ converging to $C \in \R$.
\end{itemize}
Finally, assume that $[\Omega] \in H_2(X;\Q)$. Then $C$ is in the area spectrum $\ASpec_T(X)$. \end{lemma}

\begin{lemma}[Morse-Bott] \label{lem:ASpec_MB_cobordism} Let $(X,\Omega,\Theta)$ be a cobordism with Morse-Bott ends. Then
\begin{itemize}
\item For any Morse-Bott families $S_+$ in $Y_+$ and $S_-$ in $Y_-$ and any orbit sets $\Gamma_\pm \in S_\pm$, we have
\[\ASpec(X;S_+,S_-) := \ASpec(X;\Gamma_+,\Gamma_-) \qquad\text{depends only on $S_+$ and $S_-$}\]
\item The area spectrum $\ASpec_T(X)$ is discrete.
\end{itemize}\end{lemma}

\subsection{Conformal Symplectic Cobordisms} \label{subsec:conformal_hamiltonian_manifolds} We now introduce the special classes of stable Hamiltonian manifold and symplectic cobordisms that will be the focus of this paper.

\begin{definition} \label{def:pos_stab_form} A \emph{conformal} Hamiltonian manifold $(Y,\omega,\theta)$ is a Hamiltonian manifold with \emph{conformal} stabilizing $1$-form satisfying
\[d\theta = c_Y \cdot \omega  \qquad\text{for a conformal constant }c_Y\ge 0\] \end{definition}
Note that when $c_Y=1$, $(Y,\omega,\theta)$ is a contact manifold, and when $c_Y=0$ it is a mapping torus of a symplectomorphism. See examples~\ref{exa:mapping_torus}-\ref{exa:contact} below for more details.

\begin{definition} \label{def:pos_cob} A \emph{conformal} symplectic cobordism $X:Y_+ \to Y_-$ between conformal Hamiltonian manifolds is a symplectic cobordism $(X,\Omega,\Theta)$ equipped with a \emph{conformal} vector-field $Z$ satisfying
\[\mathcal{L}_Z\Omega = c_X \cdot \Omega \qquad\text{for a conformal constant }c_X\ge 0 \qquad\text{and}\qquad \iota_Z\Omega = \Theta_\pm\text{ on }Y_\pm\]
\end{definition} 
These classes are closed under the basic operations. Indeed, the disjoint union of two conformal cobordisms $(X,\Omega,Z)$ and $(X',\Omega',Z')$ is conformal, equipped with the conformal vector-field
\[Z \sqcup Z'\]
Similarly, the trace $(X)_M$ of a conformal cobordism $X$ along a Hamiltonian boundary region $M$ is conformal since the vector-field $Z$ on $X$ descends to the quotient.

\vspace{3pt}

In the remainder of this section, we will outline the properties of these spaces and establish important notation for use later in the paper.

\subsubsection{Conformal Completion} There is a canonical way of equipping symplectizations and completed cobordisms of conformal Hamiltonian manifolds with a symplectic structure.

\begin{definition}[Conformal Symplectization] \label{def:conformal_completion} The \emph{conformal  symplectic form} $\hat{\omega}$ on the symplectization $\R \times Y$ of a conformal Hamiltonian manifold $(Y,\omega,\theta)$ is uniquely defined by
\[\hat{\omega} = dr \wedge \theta + \omega \;\; \text{along} \;\; 0 \times Y\qquad\text{and}\qquad \mathcal{L}_{\partial_r}\hat{\omega} = c_Y \cdot \hat{\omega}\] 
There is also an extension of of $\theta$ to a conformal $1$-form on $\R \times Y$, given by
\[\hat{\theta} = \iota_{\partial_r}\hat{\omega} \qquad\text{so that}\qquad d\hat{\theta} = c_Y \cdot \hat{\omega}\] \end{definition}

\begin{definition}[Conformal Completion] The \emph{canonical symplectic form} $\hat{\Omega}$ on the completion $\hat{X}$ of a  symplectic cobordism
\[(X,\Omega,\Theta):(Y_+,\omega_+,\theta_+) \to (Y_-,\omega_-,\theta_-)\]
between conformal Hamiltonian manifolds $Y_+$ and $Y_-$, is defined by
\[
\hat{\Omega} = \Omega \text{ on }X \qquad \hat{\Omega} = \hat{\omega}_+ \text{ on }[0,\infty) \times Y_+\qquad\text{and}\qquad \hat{\Omega} = \hat{\omega}_- \text{ on }(-\infty,0] \times Y_-
\]
In the special case where $X$ is a conformal cobordism with conformal vector-field $Z$, there are natural extensions of the vector-field $Z$ and $1$-forms $\Theta$ to the completion of $X$ given by
\[
\hat{Z} = \partial_r \quad\text{and}\qquad \hat{\Theta} = \hat{\Omega}(\hat{Z},-)\qquad \text{on} \qquad (-\infty,0]_r \times \partial_-X \quad\text{and}\quad [0,\infty)_r \times \partial_+X
\]
The symplectic form $\hat{\Omega}$ on the completion is then the natural symplectic form satisfying
\[
\hat{\Omega}|_X = \Omega \qquad\text{and}\qquad \mathcal{L}_{\hat{Z}}\hat{\Omega} = c_X \cdot \hat{\Omega}
\]\end{definition}

The conformal symplectic form on the symplectization can be expressed more explicitly in the same form as the symplectic form on the trivial cobordism (Example \ref{ex:trival_cob}).

\begin{lemma}\label{lem:symplectic_form_on_completion} Fix a conformal Hamiltonian manifold $(Y,\omega,\theta)$  and consider the unique function
\begin{equation} \label{eqn:conformal_function}
\f:\R_r \to \R \qquad\text{satisfying}\qquad \f(0) = 0 \quad\text{and}\quad \f' = c_Y \cdot \f + 1
\end{equation}
Then the symplectic form $\hat{\omega}$ on the symplectization $\R_r \times Y$ of $Y$ is given by
\[
\hat{\omega} = d(\f(r)\theta) + \omega = \f'(r) \cdot dr \wedge d\theta + \f(r) \cdot d\theta + \omega
\]
\end{lemma}

\begin{proof} Let $Z$ denote the vector-field $\partial_r$ on $\R_r \times Y$. Since $\f(0) = 0$ and $\f'(0) = 1$, we have
\[
\f'(r) \cdot dr \wedge d\theta + \f(r) \cdot d\theta + \omega =  dr \wedge d\theta + \omega \quad\text{if}\quad r = 0
\]
Likewise, we can compute the Lie derivative along $Z = \partial_r$ to see that
\[
\mathcal{L}_Z\big(d(\f\theta) + \omega\big) = d(\iota_Z(\f'dr \wedge \theta)) = d(\f' \theta) = d((c_Y \cdot \f + 1)\theta) = c_Y(\f'dr \wedge \theta + \f d\theta + \omega)
\]
Thus $\hat{\omega}$ and $d(\f(r)\theta) + \omega$ satisfy the same ODE with the same initial conditions, and are equal. \end{proof}

\begin{remark} Even more explicitly, the function $\f:\R_r \to \R$ can be defined as follows.
\begin{equation}\label{eq:conformal_function}
    \f(r) = r \; \text{ if }c_Y = 0 \qquad\text{and}\qquad \f(r) = c_Y^{-1} \cdot (\exp(c_Y \cdot r) - 1) \;\text{ if }c_Y > 0.
\end{equation} \end{remark}

\subsubsection{Graphical Deformations}  There is a natural class of deformations on the space of conformal Hamiltonian structures and conformal cobordism structures on a given manifold. 

\begin{definition}[Graphical Deformation, Manifolds] \label{def:graphical_deformation_SHS} Let $(Y,\omega,\theta)$ be a conformal Hamiltonian manifold and fix a smooth function
\[
u:Y \to \R
\]
The \emph{graphical deformation} $Y_u$ is the conformal Hamiltonian manifold
\[
(Y_u,\omega_u,\theta_u) = (Y,\sigma^*\hat{\omega},\sigma^*\hat{\theta}) \qquad\text{where}\qquad\sigma:Y \to \R \times Y \text{ is the graph of $u$}
\] \end{definition}

The graphical deformations of a conformal Hamiltonian manifold satisfy many identities. First, we have the following additivity properties.
\[(Y_u)_v = Y_{u+v} \qquad\text{and}\qquad Y = Y_0\]
Moreover, there is a canonical symplectomorphism
\[\R \times Y_u = \R \times Y \qquad\text{ for any }u:Y \to \R\]

\begin{definition}[Graphical Deformation, Cobordisms] \label{def:graphical_deformation_cob} Let $(X,\Omega,Z):W \to Y$ be a conformal symplectic cobordism and fix a pair of smooth functions
\[
u:W \to \R \qquad\text{and}\qquad v:Y \to \R
\]
Consider the completion $(\hat{X},\hat{\Omega},\hat{Z})$ and note that flow by $\hat{Z}$ defines a local embedding
\[
\Phi:\R \times (W \sqcup Y) \to \hat{X} \qquad\text{with}\qquad \Phi^*\hat{\Omega} = \hat{\omega} \quad \text{and}\quad \Phi^*\hat{Z} = \partial_r
\]
Let $U \subset \R \times W$ and $V \subset \R \times Y$ denote the subsets of the symplectizations of $W$ and $Y$ given by
\[
U := \big\{(r,x) \; : \; r \ge u(x)\big\} \qquad\text{and}\qquad V := \big\{(r,x) \; : \; r \le v(x)\big\} 
\]
The pair $(u,v)$ is a \emph{deformation pair} for $X$ if $\Phi$ restricts to an embedding on the open subset $U \sqcup V$. 

\vspace{3pt}

The \emph{graphical deformation} of $X$ by a deformation pair $u$ and $v$ is the conformal cobordism
\[
X^u_v:W_u \to Y_v
\]
given by the complement of the image of $U$ and $V$ under $\Phi$ in the symplectization
\[X^u_v := \hat{X} \; \setminus \; \Phi(U \sqcup V) \qquad\text{with}\qquad \Omega^u_v := \hat{\Omega}|X^u_v \quad\text{and}\quad Z^u_v := \hat{Z}|X^u_v\]
This is a conformal cobordism from the graphical deformation $W_u$ of $W$ by $u$, to the graphical deformation of $Y$ by $v$, via Definition \ref{def:graphical_deformation_SHS}. \end{definition}

\begin{definition}[Intergraph]\label{def:intergraph_cobordism} Let $(Y,\omega,\theta)$ be a conformal Hamiltonian manifold and fix two smooth functions
\[
u:Y \to \R \quad\text{and}\quad v:Y \to \R \qquad\text{with}\qquad u > v
\]
The \emph{intergraph} cobordism associated to $Y$ and the pair $u,v$ is the conformal symplectic cobordism
\[\Gamma Y^u_v:Y_u \to Y_v\]
acquired by taking the subset of the symplectization
\[
\Gamma Y^u_v := \big\{(r,y)\; : \; v(y) \le r \le u(y)\} \subset \R_r \times Y
\]
equipped with the restriction of the symplectic form $\hat{\omega}$ and the conformal vector-field $\hat{Z} = \partial_r$. \end{definition}

Intergraphs and graphical deformations satisfy a number of useful formal properties that will be used repeatedly in this paper. First, we have the following additivity property
\begin{equation} \label{eqn:deformation_additive}
(\Gamma Y^u_v)^t_w = \Gamma Y^{u+t}_{v+w} \qquad \text{and}\qquad (X^u_v)^t_w = X^{u+t}_{v+w}
\end{equation}
We also have the following composition properties. 
\begin{equation} \label{eqn:intergraph_composition}
\Gamma Y^v_w \circ \Gamma Y^u_v = \Gamma Y^u_w \qquad \text{if}\qquad u > v > w
\end{equation}
\begin{equation} \label{eqn:deformation_composition}
\Gamma Y^v_w \circ X^u_v = X^u_w \quad\text{and}\quad X^u_v \circ \Gamma Y_u^t = X^t_v \qquad  \text{if}\qquad t > u \quad\text{and}\quad v > w
\end{equation}
These composition laws imply the following inclusion properties.
\begin{equation} \label{eqn:deformation_inclusion} \Gamma Y^u_v \subset \Gamma Y^t_w \qquad\text{and}\qquad X^u_v \subset X^t_w \qquad\text{if}\qquad  t > u \quad\text{and}\quad v > w\end{equation}
Finally, we have the following identifications of completions. 
\begin{equation} \label{eqn:deformation_completion} \hat{\Gamma Y}^u_v = \hat{Y} \qquad\text{and}\qquad \hat{X}^u_v = \hat{X} \end{equation}
These identifications respect the symplectic form and the conformal vector-field.

\vspace{3pt}

The following property of the area spectrum under graphical deformation will be useful later.

\begin{lemma}[Constant $\ASpec$] \label{lem:area_spec_same} Let $X:M \to N$ be a conformal symplectic cobordism and fix a graphical deformation $X^u_v:M_u \to N_v$. Suppose that
\begin{itemize}
    \item Any closed orbit of $M$ or $M_u$ of length $T$ or less is contained in the zero set of $u$.
    \vspace{3pt}
    \item Any closed orbit of $N$ or $N_v$ of length $T$ or less is contained in the zero set of $v$.
\end{itemize}
Then the area spectra of $X$ and $X^u_v$ below length $T$ are equal. That is
\[\ASpec_T(X^u_v) = \ASpec_T(X)\]
\end{lemma}

\begin{proof} There are inclusions of $M,M_u,N$ and $N_v$ into $\hat{X} = \hat{X}^u_v$ as the boundaries of $X$ and $X^u_v$. Moreover, $M_u \cap M$ and $N_v \cap N$ in $\hat{X}$ are the zero sets of $u$ and $v$. Our assumptions imply that orbit sets $\Gamma$ of $M$ and $M_u$ of period $T$ or less agree and are contained in $M \cap M_u$, and likewise for $N$ and $N_v$. Thus, fix orbit sets $\Gamma_+$ in $M$ and $\Gamma_-$ in $N$ of period $T$ or less. The inclusions $X \to \hat{X}$ and $X^u_v \to \hat{X}$ induce an identification
\[
S(X;\Gamma_+,\Gamma_-) \simeq S(\hat{X};\Gamma_+,\Gamma_-) \simeq S(X^u_v;\Gamma_+,\Gamma_-)
\]
that respects the area, i.e. $\Omega \cdot A = \hat{\Omega} \cdot A = \Omega^u_v \cdot A$. The result now follows. \end{proof}

\subsubsection{Thickening And Thinning} A particularly important special case of graphical deformation is the simple case of deformation by a constant. That is, we consider
\[(Y_\delta,\omega_\delta,\theta_\delta) \qquad\text{for a conformal Hamiltonian manifold $(Y,\omega,\theta)$ and a constant }\delta\]
\[(X^\delta_\epsilon,\Omega^\delta_\epsilon,Z^\delta_\epsilon):Y_\delta \to Y_\epsilon \qquad\text{for a conformal cobordism $(X,\Omega,Z)$ and constants }\delta,\epsilon\]
We will typically assume either that $\delta > 0 > \epsilon$, where $X^\delta_\epsilon$ is referred to as a \emph{thickening} of $X$, or that $\epsilon > 0 > \delta$, where $X^\delta_\epsilon$ is referred to as a \emph{thinning} of $X$. 

\vspace{3pt}

We will require many elementary properties of this simple class of deformations, that we now discuss in detail. The stable Hamiltonian structure on $Y_\delta$ is explicitly given  by
\[\omega_\delta = \omega + \f(\delta)d\theta = (1 + c_Y \cdot \f(\delta)) \cdot \omega \qquad\text{and}\qquad \theta_\delta = \f'(\delta) \cdot \theta\]
Therefore, the Reeb vector-field $R$ of $Y$ and $R_\delta$ of $Y_\delta$ are related as follows.
\[\f'(\delta) \cdot R_\delta = R\]
In particular, the closed Reeb orbits (and orbit sets) of $Y$ are the same as those in $Y_\delta$ up to reparametrization. We denote the orbit set in $Y_\delta$ corresponding to an orbit set $\Gamma$ in $Y$ by
\[\Gamma^\delta\]
Given $\delta > 0$ (respectively, $\delta<0$), there is an obvious surface class in $\hat{X}$ connecting $\Gamma^\delta$ to $\Gamma$ (respectively, $\Gamma$ to $\Gamma^\delta$), acquired by taking the product of $\Gamma$ with the interval.
\[
[0,\delta] \times \Gamma:\Gamma^\delta \to \Gamma \qquad\Big(\text{respectively, }[\delta,0]\times\Gamma:\Gamma\to \Gamma^\delta\Big)
\]
Composition on either side with these surface classes in $\hat{X}$ yields canonical identification
\[
S(\hat{X};\Gamma_+,\Gamma_-)  \simeq S(\hat{X};\Gamma_+^\delta,\Gamma_-^\epsilon)  
\]
Since the inclusions $X \to \hat{X}$ and $X^\delta_\epsilon \to \hat{X}$ are homotopy equivalences (and thus induce isomorphisms on the 2nd homology), this yields an identification
\[S(X;\Gamma_+,\Gamma_-) \simeq S(X^\delta_\epsilon;\Gamma_+^\delta,\Gamma^\epsilon_-) \qquad\text{denoted by}\qquad A \mapsto A^\delta_\epsilon\]

It will be  especially important to understand how the area of a surface class changes under thickening of a cobordism. We record the required rule below.

\begin{lemma}[Thickening Area] \label{lem:area_change_under_thickening} Fix a surface class $A:\Gamma_+ \to \Gamma_-$ and fix constants $\delta,\epsilon$. Then
\[\Omega^\delta_\epsilon \cdot A^\delta_\epsilon - \Omega \cdot A = \f(\delta) \cdot \mathcal{L}(\Gamma_+) - \f(\epsilon) \cdot \mathcal{L}(\Gamma_-)\]
\end{lemma}

\begin{proof} Since the inclusions $X \to \hat{X}$ and $X^\delta_\epsilon \to \hat{X}^\delta_\epsilon = \hat{X}$ are symplectic, we have
\[\Omega^\delta_\epsilon \cdot A^\delta_\epsilon = \hat{\Omega} \cdot A^\delta_\epsilon \qquad\text{and}\qquad \Omega \cdot A = \hat{\Omega} \cdot A\]
See Definition \ref{def:integral_pairing}. By construction, we can decompose the surface class $A^\delta_\epsilon$ as a composition
\[A^\delta_\epsilon = S_- \circ A \circ S_+ \quad\text{where}\quad S_+ = [0,\delta] \times \Gamma_+\quad\text{and}\quad  S_- = [\epsilon,0] \times \Gamma_-\]
Thus $\Omega^\delta_\epsilon \cdot A^\delta_\epsilon - \Omega \cdot A$ is equal to $\hat{\Omega} \cdot S_+ + \hat{\Omega} \cdot S_-$. We compute that $\hat{\Omega} \cdot S_+$ is given by
\[
\hat{\Omega} \cdot S_+ = \int_{[0,\delta] \times \Gamma_+} \hat{\Omega} = \int_{[0,\delta] \times \Gamma_+}  d(\f(r)\theta_+) + \omega_+ = \int_{\Gamma_+} \f(\delta)\theta_+ - \f(0)\theta_+ = \f(\delta) \cdot \mathcal{L}(\Gamma_+)
\]
An analogous computation for $S_-$ yields the desired result.\end{proof}

\subsubsection{Space Of Conformal Hamiltonian Structures} The notion of graphical deformation yields a natural division of conformal Hamiltonian structures into equivalence classes.

\begin{definition} A conformal Hamiltonian structure $(\omega',\theta')$ is \emph{graphically equivalent} to $(\omega,\theta)$ if it is a graphical deformation of $(\omega,\theta)$ by some function $\rho$. \end{definition}

The space of conformal Hamiltonian structures in a graphical equivalence class $\eta$ is denoted
\[\PHam(Y,\eta)\]
The space $\PHam(Y,\eta)$ has a natural $C^k$-topology for any $k$, via any of the bijections with $C^\infty(Y,\R)$ determined by a choice of representative $(\omega,\theta)$. 

\begin{lemma}[Generic Non-Degeneracy] \label{lem:generic_nondegeneracy} Let $Y$ be a closed manifold with a graphical equivalence class of conformal Hamiltonian structures $\eta$. Then the subset
\[\PHam_{\on{ND}}(Y,\eta) \subset \PHam(Y,\eta)\]
of non-degenerate, conformal Hamiltonian structures in the class $\eta$ is comeager in the $C^\infty$-topology.
\end{lemma}

\begin{proof} Fix a stable Hamiltonian structure $(\omega,\theta)$ in $\PHam(Y,\eta)$ and consider the symplectization
\[
(\R \times Y,\hat{\omega})
\]
By Robinson \cite[Thm 1]{r1970}, we can find an embedded hypersurface $\Sigma \subset \R \times Y$ that is diffeomorphic to $Y$, arbitrarily close to $0 \times Y$ in the $C^\infty$-topology and that has non-degenerate characteristic foliation. On the other hand, if $\Sigma$ is sufficiently $C^\infty$-close, then we may write
\[
\Sigma = \sigma(Y) \qquad\text{for a $C^\infty$-small section}\qquad \sigma:Y \to \R \times Y
\]
This determines a conformal Hamiltonian structure $(\omega_\sigma,\theta_\sigma)$ in $\PHam(Y,\eta)$ with the same characteristic foliation. We can choose $\sigma$ to be arbitarily small in $C^\infty$, so $\PHam_{\on{ND}}(Y,\eta)$ is dense.

\vspace{3pt}

Now let $\mathcal{U}_T \subset \PHam(Y,\eta)$ denote the set of conformal Hamiltonian structures in the class $\eta$ such that every closed orbit of period less than or equal to $T$ is non-degenerate. This set is open in the $C^\infty$-topology and we have
\[\PHam_{\on{ND}}(Y,\eta) = \bigcap_{T \in \N} \mathcal{U}_T\]
Therefore, $\PHam_{\on{ND}}(Y,\eta)$ is a countable intersection of dense open sets, and is comeager. \end{proof}

\subsubsection{Hofer Distance} There is a natural distance on the space of conformal Hamiltonian manifolds and conformal symplectic cobordisms, induced by the notion of graphical deformation.

\begin{definition}[Hofer Distance] \label{def:hofer_distance_manifolds} The \emph{Hofer distance} $d_H(M,N)$ between two conformal Hamiltonian manifolds $M$ and $N$ is given by
\[
d_H(M,N) := \inf\big\{\|u\|_{C^0} \; : \; M \overset{\Psi}{\simeq} N_u \text{ for a smooth function }u:M \to \R\big\}
\]
The \emph{Hofer distance} $d_H(W,X)$ between two conformal symplectic cobordisms $W$ and $X$ is given by
\[
d_H(W,X) := \inf\big\{\on{max}(\|u\|_{C^0},\|v\|_{C^0}) \; : \; W \overset{\Psi}{\simeq} X^u_v \text{ for a deformation pair }u,v\big\}
\]\end{definition}

These distances are morally similar to the Gromov-Hausdorff distance on metric spaces, and are equivalent to the usual Hofer metric in the case of mapping tori of symplectomorphisms. Precisely, we can show that these functions are pseudo-metrics in the following sense.

\begin{definition}[Pseudo-Metric] A function $m(A,B)$ on pairs of objects in a category $\mathcal{C}$ is a \emph{pseudo-metric} if it satisfies the following axioms.
\begin{itemize}
    \item (Positivity) For any pair $A$ and $B$, $d_H(A,B) \ge 0$. 
    \vspace{3pt}
    \item (Identity) If $A \simeq B$ (and in particular if $A = B$) then $m(A,B) = 0$.
    \vspace{3pt}
    \item (Symmetry) For any pair $A$ and $B$, $m(A,B) = m(B,A)$.
    \vspace{3pt}
    \item (Triangle) For any three objects $A,B$ and $C$, we have
    \[m(A,B) + m(B,C) \ge m(A,C)\]
\end{itemize}
\end{definition}

\begin{lemma} The Hofer distance is a pseudo-metric on conformal Hamiltonian manifolds and on conformal symplectic cobordisms.
\end{lemma}

\begin{proof} In both the manifold and cobordism cases, only the symmetry and triangle property are non-trivial. For the Hofer distance on conformal Hamiltonian manifolds, we note that
\[M \simeq N_u \qquad\text{implies that}\qquad N \simeq (N_u)_{\smallneg u} \simeq M_{\smallneg u}\]
\[M \simeq Y_u \quad\text{and}\quad Y \simeq N_v \qquad\text{implies that}\qquad M \simeq (N_v)_u = N_{u+v} x\]
These formulas imply symmetry and the triangle property, respectively. Similarly, for the Hofer distance on conformal symplectic cobordisms, we note that by the additiviy property (\ref{eqn:deformation_additive}) of deformations, we have
\[W \simeq X^u_v \qquad\text{implies that}\qquad W^{\smallneg u}_{\smallneg v} \simeq (X^u_v)^{\smallneg  u}_{\smallneg  v} = X\]
\[V \simeq W^w_x \quad\text{and}\quad W \simeq X^u_v \qquad\text{implies that}\qquad V \simeq (X^w_x)^u_v = X^{w+u}_{x+v} \qedhere\]
\end{proof}

\subsubsection{Examples} The most important types of conformal Hamiltonian manifolds are mapping tori and contact manifolds. 

\begin{example}[Mapping Torus]\label{exa:mapping_torus}
Let $(P,\omega)$ be a symplectic manifold. The \emph{mapping torus} $M_\Phi$ of a symplectomorphism $\Phi:P \to P$ is the Hamiltonian manifold given by
\[
M_\Phi := [0,1]_t \times P/\sim \quad\text{where}\quad (1,x) \sim (0,\Phi(x))
\]
The Hamiltonian $2$-form $\omega_\Phi$ and conformal stabilizing $1$-form $\theta$ are the pushforwards of $\omega$ and $dt$ through the quotient $[0,1] \times P \to M_\Phi$. Any graphically equivalent pair is of the form
\[(\omega_{\Phi,H},dt) = (\omega_\Phi + dH \wedge dt,dt) \qquad\text{for a $C^\infty$-function}\qquad H:M_\Phi \to \R\]
There is an isomorphism $(M_\Phi,\omega_{\Phi,H},dt) \simeq (M_\Psi,\omega_\Psi,dt)$ where $\Psi = \Phi^H \circ \Phi$ is the composition of $\Phi$ with the Hamiltonian flow of $H$. Thus the graphical equivalence class
\[
\eta = [\omega_\Phi,dt]
\]
can be viewed as the Hamiltonian isotopy class of the symplectomorphism $\Phi$. In this case, Lemma \ref{lem:generic_nondegeneracy} states that the set of Hamiltonian perturbations $\Psi$ of $\Phi$ that have non-degenerate periodic points is comeager.\end{example}

\begin{example}[Contact Manifolds] \label{exa:contact} A \emph{contact manifold} $(Y,\xi)$ is an odd-dimensional manifold $Y$ with a hyperplane field $\xi \subset TY$ called a \emph{contact structure} that admits a $1$-form $\alpha$ satisfying
\[\xi = \on{ker}(\alpha) \qquad\text{and}\qquad \alpha \wedge d\alpha^{n-1} > 0\]
The $1$-form is called a \emph{contact form}. This is an example of a conformal Hamiltonian manifold with Hamiltonian 2-form $d\alpha$ and conformal stabilizing $1$-form $\alpha$.

\vspace{3pt}

The contact structure $\xi$ is itself a graphical equivalence class and the space $\PHam(Y,\xi)$ is simply the space
\[
\on{Form}(Y,\xi) = \big\{\text{ contact forms $\alpha$ for $\xi$}\big\}
\]
In this case, Lemma \ref{lem:generic_nondegeneracy} is the standard fact that non-degenerate contact forms are comeager.\end{example}

\section{Symplectic Field Theory} \label{sec:SFT} In this section, we review the basic geometric formalism of symplectic field theory (SFT) as introduced by Eliashberg-Givental-Hofer \cite{egh2000}. The setup that we will use is tailored to the setting of conformal symplectic cobordisms, and differs slightly from the standard one.

\subsection{Almost Complex Structures} \label{subsec:almost_complex_str} Symplectic field theory requires the use of a specific class of cylindrical almost complex structures. We now introduce our versions of these structures.

\begin{definition} \label{def:acs_symplectization} An almost complex structure $J$ on the symplectization $\R \times Y$ of a stable Hamiltonian manifold $(Y,\omega,\theta)$ is \emph{compatible} if
    \begin{itemize}
        \item $J$ is invariant under the natural $\R$-action on $\R \times Y$.
        \vspace{3pt}
        \item $J$ maps the $\R$-direction to the Reeb direction, so $J(\partial_r) = cR$ for some $c > 0$.
        \vspace{3pt}
        \item $J$ preserves $\ker(\theta)$ and the restriction of $\omega(-,J-)$ is a metric on $\on{ker}(\theta)$. 
    \end{itemize}
We denote the space of compatible almost complex structures on $\R \times Y$ as follows.
\[
\mathcal{J}(Y,\omega,\theta) \qquad\text{or more simply} \qquad \mathcal{J}(Y)
\]
\end{definition} 

\noindent Standard arguments (cf. \cite[\S 2.2]{wendlholomorphic}) imply that the space of these structures is non-empty.

\begin{lemma} \label{lem:space_of_J_symp} The space $\mathcal{J}(Y)$ for a stable Hamiltonian manifold $Y$ is contractible and non-empty,
\end{lemma}

We will use the following class of cylindrical complex structures on cobordisms. Note that it differs slightly from the standard version in SFT  (cf. \cite[\S 1.4]{wendlsymplectic}). 

\begin{definition}\label{def:acs_cob} An almost complex structure $J$ on the completion $\hat{X}$ of a symplectic cobordism
\[(X,\Omega,\Theta):Y_+ \to Y_- \qquad\text{between conformal Hamiltonian }Y_\pm\]
is called \emph{compatible} if there is a compact set $K \subset \hat{X}$ containing $X$ such that:
\begin{itemize}
    \item $J$ is compatible with the symplectic form $\hat{\Omega}$ on $\hat{X}$. That is
    \[\hat{\Omega}(-,J-) \qquad\text{is a $J$-invariant Riemannian metric}\]
    \item There is a compatible almost complex structure $J_{\partial X}$ on $\R \times \partial X$ such that
    \[J = J_{\partial X} \qquad\text{on}\qquad \hat{X} \setminus K \subset \R \times \partial X\]
    We say that $J$ is \emph{cylindrical} on the subset $\hat{X} \setminus K$, or away from the subset $K$.
\end{itemize}
If $M \subset \partial X$ is a component, the \emph{asymptotic complex structure $J_M$} is the compatible almost complex structure on $\R \times M$ agreeing with $J$ on the cylindrical end corresponding to $M$. 

\vspace{3pt}

The space of compatible almost complex structures on a cobordism $(X,\Omega,Z)$ will be denoted
\[\mathcal{J}(X,\Omega,Z) \qquad\text{or more simply}\qquad \mathcal{J}(X) \]
We also denote the subset of almost complex structures cylindrical away from a subset $K \subset \hat{X}$ by
\[\mathcal{J}_K(X,\Omega,Z) \qquad\text{or more simply}\qquad \mathcal{J}_K(X) \]
We will topologize this space with a slightly stronger topology than the $C^\infty$ topology. We say that $J_k$ converges to $J$ \emph{$C^\infty$-cylindrically} if there is a fixed compact set $K \subset \hat{X}$ such that
\[
J_k \to J \text{ uniformly in $C^\infty$} \qquad\text{and}\qquad J_k \text{ is cylindrical on $X \setminus K$}
\]
\end{definition}

\begin{lemma} \label{lem:space_of_J_cob} The spaces of compatible almost complex structures on $(X,\Omega,Z)$
\[\mathcal{J}(X,\Omega,Z) \qquad\text{and}\qquad \mathcal{J}_X(X,\Omega,Z) \subset \mathcal{J}(X,\Omega,Z)
\]
are contractible and non-empty. Moreover, the following restriction map is surjective.
\[
\mathcal{J}_X(X) \to \mathcal{J}(\partial X) \qquad\text{given by}\qquad J \mapsto J_{\partial X}
\]\end{lemma}
The assertion of this lemma is stated in \cite[Section 1]{wendlsymplectic} for $\cJ_X(X)$ and it follows for $\cJ(X)$ as it is exhausted by the previous class when we extend the cobordism $X$. 

Given a conformal symplectic cobordism $(X,\Omega,Z)$, there is a natural $\R$-action on $\mathcal{J}(X)$ induced by the flow $\Phi:\R \times \hat{X} \to \hat{X}$ of $\hat{Z}$.
\[\Phi_*:\R \times \mathcal{J}(X) \to \mathcal{J}(X) \qquad\text{given by}\qquad (t,J) \mapsto \Phi^t_*J\]
 Moreover, the identification of completions $\hat{X}^\delta_\epsilon = \hat{X}$ for constants $\delta,\epsilon$ yields a natural identification of spaces of almost complex structures for thickenings.
\[\mathcal{J}(X^\delta_\epsilon) = \mathcal{J}(X)\]

\subsection{Holomorphic Curves} \label{subsec:holomorphic_curves} We now introduce the types of pseudo-holomorphic curves that we will use in this paper. Recall that a \emph{$J$-holomorphic map} $u:(\Sigma,j) \to (M,J)$ from a Riemann surface $(\Sigma,j)$ to a manifold $M$ with an almost complex structure $J$ is a map such that
\[du \circ j = J \circ du\]
A \emph{$J$-holomorphic curve} $u$ is a $J$-holomorphic map modulo the equivalence relation of reparametrization of the domain surface.

\vspace{3pt}

\subsubsection{Asymptotics} We will be exclusively interested in pseudo-holomorphic curves that are asymptotic to Reeb orbits in an appropriate sense. 

\begin{definition}[Asymptotics, Local] \label{def:local_asymptotics} Let $Y$ be a conformal Hamiltonian manifold and fix a compatible almost complex structure $J$ on $\R \times Y$. A $J$-holomorphic map 
\[
w:([0,\infty) \times \R/\Z,i) \to (\R \times Y,J) \qquad \text{denoted by}\qquad w(s,t) = (a(s,t),v(s,t)) \in \R \times Y
\]
is \emph{asymptotic} to a closed orbit $\gamma:\R/T\Z \to Y$ at $\pm \infty$  if $a$ and $v$ satisfy the following limiting conditions.
\begin{equation}
v(s,-) \xrightarrow{s \to \infty} \gamma(T-) \quad\text{and}\quad a(s,-) \xrightarrow{s \to \infty} \pm\infty \qquad\text{ in }C^\infty(\R/\Z)
\end{equation}
\end{definition} 

This asymptotic property generalizes to a condition on any $J$-holomorphic curve as follows.

\begin{definition}[Asymptotics] Let $X:Y_+ \to Y_-$ be a symplectic cobordism with conformal ends and let $J$ be a compatible almost complex structure. Fix orbit sets
\[\Gamma_+ = (\gamma^+_1,\dots,\gamma^+_m) \text{ in }Y_+\qquad\text{and}\qquad \Gamma_- = (\gamma^-_1,\dots,\gamma^-_\ell) \text{ in }Y_-
\]
A $(j,J)$-holomorphic map (or curve) $u:(\Sigma,j) \to (\hat{X},J)$ is \emph{asymptotic} to $\Gamma_\pm$ at $\pm \infty$ if there are holomorphic embeddings from the half cylinder 
\[
\iota_k^\pm:([0,\infty) \times \R/\Z,i) \to (\Sigma,j) \qquad\text{with composition}\qquad w^\pm_k = u \circ \iota^\pm_k
\]
that satisfy the following properties for each $k$.
\begin{itemize}
    \item The complement of the images of $\iota^\pm_k$ has compact closure in $\Sigma$.\vspace{1pt}
    \item The map $w^+_k$ has image in $[0,\infty) \times Y_+ \subset \R \times Y$ and $w^-_k$ has image in $(-\infty,0] \times Y_- \subset \R \times Y_-$.\vspace{1pt}
    \item The map $w^\pm_k$ is asymptotic to $\gamma^\pm_k$ at $\pm \infty$, as in Definition \ref{def:local_asymptotics}.
\end{itemize}
Any such pseudo-holomorphic curve represents a well-defined surface class
\[
[u]:\Gamma_+ \to \Gamma_- \quad\text{in}\quad S(X;\Gamma_+,\Gamma_-)
\]
This is the fundamental class of the map $\bar{u}:\bar{\Sigma} \to X$ from the compactification of $\Sigma$ along the punctures, acquired by composing $u$ with the map
\[
\pi:\hat{X} \to X
\]
given by the identity on $X \subset \hat{X}$ and by the projection $\pi(s,y) = y$ on the subsets $[0,\infty) \times \partial_+X$ and $(-\infty,0] \times \partial_-X$. 

\vspace{3pt}

We have an analogous notion for a $(j,J)$-holomorphic map $u:(\Sigma,j) \to (\R \times Y,J)$ to the symplectization, equipped with a compatible almost complex structure $J$ (which may be viewed as the special case of $\hat X$ for $X = [0,1] \times Y$). In this case, the surface class
\[
[u]:\Gamma_+ \to \Gamma_- \quad\text{in}\quad S(Y;\Gamma_+,\Gamma_-)
\]
is represented by the map $\bar{u}:\bar{\Sigma} \to Y$ from the compactification of $\Sigma$ acquired by composing $u$ with the projection $\R \times Y \to Y$.\end{definition}

\subsubsection{Energy} We will primarily consider $J$-holomorphic curves that have finite energy. Let $\mathcal{S}$ denote the space of compactly suppoted, smooth functions
\[\phi:\R \to [0,\infty) \qquad\text{such that}\qquad \int_{-\infty}^{\infty} \phi(r)dr = 1\]

\begin{definition}[Energy, Symplectization] \label{def:energy_symplectization} The \emph{$\omega$-energy} and the \emph{$\theta$-energy} of a map $u:\Sigma \to \R_r \times Y$  from a surface $\Sigma$ to the symplectization of a stable Hamiltonian manifold $(Y,\omega,\theta)$ are given by
\[E_\omega(u) = \int_\Sigma u^*\omega  \qquad\text{and}\qquad E_\theta(u) = \sup_{\phi \in \mathcal{S}}\Big(\int_\Sigma u^*(\phi(r) \cdot dr \wedge \theta)\Big)\]\end{definition}

\begin{definition}[Energy, Cobordisms] \label{def:energy_cobordism} The \emph{$\Omega$-energy} and the \emph{$\Theta$-energy} of a map $u:\Sigma \to \hat{X}$ from a surface $\Sigma$ to the completion of a symplectic cobordism $(X,\Omega,\Theta)$ are given by
\[E_\Omega(u) = \int_{\Sigma_X} u^*\Omega + \int_{\Sigma_+} u^*\omega_+ + \int_{\Sigma_-} u^*\omega_- \]
\[E_\Theta(u) = \sup_{\phi \in \mathcal{S}}\Big(\int_{\Sigma_-} u^*(\phi(r) \cdot dr \wedge \theta_-) + \int_{\Sigma_+} u^*(\phi(r) \cdot dr \wedge \theta_+) \Big)\]
Here we have decomposed $\Sigma$ into a union of regions $\Sigma_+ \cup \Sigma_X \cup \Sigma_-$ where
\[\Sigma_X := u^{-1}(X) \qquad \Sigma_+ = u^{-1}([0,\infty) \times Y_+) \quad\text{and}\quad \Sigma_- = u^{-1}((-\infty,0] \times Y_-)\]
We sometimes denote by $E(u):=E_\Omega(u)+E_\Theta(u)$ the total energy.
\end{definition}

\begin{remark} \label{rmk:different_energies_notation} If we need to emphasize the dependence of the energies and lengths on the specific ambient cobordism, we will use the notation
\[
E_\Omega(u;X) \qquad E_\theta(u;X) \qquad 
E(u;X) \qquad \mathcal{L}(\Gamma^+;X) 
\]
We will apply similar notation in the symplectization case. \end{remark}

\begin{remark}[Finite energy and asymptotics]
    Let $X$ be a conformal cobordism and assume that the Hamiltonian structures on $\partial_\pm X$ are Morse-Bott. Then, any $J$-holomorphic curve $u$ in $\hat X$ of finite total energy is asymptotic to periodic Reeb orbits at the ends (See for example \cite[Proposition 5.6]{sftCompactness}).
\end{remark}
In the setting of conformal Hamiltonian manifolds and symplectic cobordisms, there are several inequalities relating energies and lengths.

\begin{lemma} \label{lem:energy_bounds_symplectization} Let $(Y,\omega,\theta)$ be a conformal Hamiltonian manifold of conformal factor $c_Y$. Let $J$ be a compatible almost complex structure and fix a a proper, finite energy $J$-holomorphic map
\[
u:(\Sigma,j) \to (\R \times Y,J) \qquad\text{asymptotic to}\qquad \Gamma_\pm \text{ at }\pm \infty
\]
Then the following estimates hold.
\begin{enumerate}[label=(\alph*)]
\item (Area) \label{itm:energy_bound_symplectization:area} The $\omega$-energy is non-negative and is the area of the surface class of $u$ (see Definition \ref{def:area_SHS_surface_class}).
\[0 \le E_\omega(u) \qquad\text{and}\qquad E_\omega(u) = \omega \cdot [u]\]
\item (Length) \label{itm:energy_bound_symplectization:length} The length of $\Gamma_+$ and $\Gamma_-$ satisfy the following bound.
\[\
\mathcal{L}(\Gamma_+) - \mathcal{L}(\Gamma_-) = c_Y \cdot E_\omega(u) \qquad\text{and}\qquad \mathcal{L}(\Gamma_+) \ge \mathcal{L}(\Gamma_-) 
\]
\item (Theta) \label{itm:energy_bound_symplectization:theta} 
 The $\theta$-energy satisfies the following bound.
\[E_\theta(u) \le \mathcal{L}(\Gamma_+)\]
\end{enumerate}
\end{lemma}

\begin{proof} Starting with (a), we note let $\bar{u}:\bar{\Sigma} \to Y$ be the map aquired by composing $u$ with the projection $\R \times Y \to Y$, then compactifying along the punctured of $\Sigma$. Then
\[\omega \cdot [u] = \int_{\bar{\Sigma}} \bar{u}^*\omega = \int_\Sigma u^*\omega = E_\omega(u)\]
On the other hand, $h = \omega(-,J-)$ is positive semi-definite on $T(\R \times Y)$ and thus so is $u^*h$. Since $u$ is $J$-holomorphic, we have
\[u^*h = u^*(\omega(-,J-)) = (u^*\omega)(-,j-)\]
It follows that $u^*\omega = f \cdot \mu_\Sigma$ where $f$ is a non-negative 
smooth function and $\mu_\Sigma$ is a volume form on $\Sigma$. Therefore, $E_\omega(u)$ is non-negative. To show (b), note that $(Y,\omega,\theta)$ is conformal. Thus by Stokes theorem
\[
 \mathcal{L}(\Gamma_+) - \mathcal{L}(\Gamma_-) = \int_{\Gamma_+} \theta - \int_{\Gamma_-} \theta = \int_{\Sigma} u^*d\theta = \int_{\Sigma} c_Y \cdot u^*\omega = c_Y \cdot E_\omega(u)
\]
Finally, to prove (c) choose $\phi \in \mathcal{S}$ and let $\psi:\R \to [0,1]$ be the unique 
 increasing smooth function with
\[\frac{d\psi}{dr} = \phi \qquad \lim_{r \to - \infty} \psi(r) = 0 \quad\text{and}\quad \lim_{r \to \infty} \psi(r) = 1\]
We then acquire the following integral bound.
\[\int_\Sigma u^*(\phi(r) \cdot dr \wedge \theta) = \int_\Sigma u^*\big(d(\psi(r) \cdot \theta) - \psi(r) d\theta\big) \le \int_{\Gamma_+} \theta - \int_{\Sigma} u^*\psi \cdot c_Y\cdot f \cdot \mu_\Sigma\]
Since $\psi, f$ an $c_Y$ are non-negative and $\mu_\Sigma$ is an area form on $\Sigma$, the integral on the right above is non-negative. Thus 
\[\int_\Sigma u^*(\phi(r) \cdot dr \wedge \theta) \le \int_{\Gamma_+} \theta  = \mathcal{L}(\Gamma_+) \qedhere\]
\end{proof}

\begin{lemma} \label{lem:energy_bounds_cobordism} Let $(X,\Omega,Z)$ be a conformal symplectic cobordism with conformal factor $c_X$. Let $J$ be a compatible almost complex structure and fix a a proper, finite energy $J$-holomorphic map
\[
u:(\Sigma,j) \to (\widehat{X},J) \qquad\text{asymptotic to}\qquad \Gamma_\pm \text{ at }\partial_\pm X
\]
Assume that $J$ is cylindrical outside of the thickening $X^a_{\smallneg a}$. Then the following estimates hold.
\begin{enumerate}[label=(\alph*)]
\item (Area) \label{itm:energy_bound_cobordism:area} The $\Omega$-energy satisfies the following identities.
\[E_\Omega(u) = \Omega \cdot [u] \qquad\text{and}\qquad E_\Omega(u) \ge -(\f(a) + \frac{\f'(a)\f({-}a)}{\f'(-a)}) \cdot \mathcal{L}(\Gamma_+)\]
\item (Length) \label{itm:energy_bound_cobordism:length} The lengths of $\Gamma_+$ and $\Gamma_-$ satisfy the following identities.
\[\mathcal{L}(\Gamma_+) - \mathcal{L}(\Gamma_-) = c_X\cdot E_\Omega(u) \qquad\text{and}\qquad \mathcal{L}(\Gamma_-) \le \frac{\f'(a)}{\f'(-a)} \cdot \mathcal{L}(\Gamma_+) \]
\item (Theta)  \label{itm:energy_bound_cobordism:theta} The $\Theta$-energy satisfies the following estimate if $J$ is cylindrical outside of $X$.
\[E_\Theta(u) \le c_X\cdot E_\Omega(u) + \mathcal{L}(\Gamma_+) + \mathcal{L}(\Gamma_-) \le c_X\cdot E_\Omega(u) + 2\mathcal{L}(\Gamma_+)\]

\end{enumerate}
\end{lemma}

\begin{proof} We start by noting that if $J$ is cylindrical outside of $X$, then $E_\Omega(u) \ge 0$. Indeed, in that case we have
\[E_\Omega(u) = \int_{\Sigma_X} u^*\Omega + \int_{\Sigma_+} u^*\omega_+ + \int_{\Sigma_-} u^*\omega_- \]
The first integral is non-negative since $J$ is compatible with $\Omega$, and the last two are non-negative if $J$ is cylindrical on $[0,\infty) \times \partial_+X$ and $(-\infty,0] \times \partial_-X$ by Lemma \ref{lem:energy_bounds_symplectization}(a). Thus $E_\Omega(u) \ge 0$. We stress that generally for $J\in \cJ(X)\setminus \cJ_X(X)$ the energy $E_\Omega(u)$ could be negative.

\vspace{3pt}

Now we show \ref{itm:energy_bound_cobordism:area} and \ref{itm:energy_bound_cobordism:length}. The proof that
\[E_\Omega(u) = \Omega \cdot [u] \qquad\text{and}\qquad \mathcal{L}(\Gamma_+) - \mathcal{L}(\Gamma_-) = c_X \cdot E_\Omega(u)\]
is an applications of Stokes theorem, as in Lemma \ref{lem:energy_bounds_symplectization}\ref{itm:energy_bound_symplectization:area} and \ref{itm:energy_bound_cobordism:length}. For the second part of \ref{itm:energy_bound_cobordism:length}, we compute that
\[
\f'(-a) \cdot \mathcal{L}(\Gamma_-;X) = \mathcal{L}(\Gamma_-;X^a_{\smallneg a}) \le \mathcal{L}(\Gamma_-;X^a_{\smallneg a}) + E_\Omega(u;X^a_{\smallneg a}) = \mathcal{L}(\Gamma_+;X^a_{\smallneg a}) \le \f'(a) \cdot \mathcal{L}(\Gamma_+;X)
\]
Finally, for the second part of \ref{itm:energy_bound_symplectization:area}, note that by Lemma \ref{lem:area_change_under_thickening} we have
\[
E_\Omega(u;X) + \f(a) \cdot \mathcal{L}(\Gamma_+;X)  - \f(-a) \cdot \mathcal{L}(\Gamma_-;X) = E_\Omega(u;X^a_{\smallneg a})  \ge 0
\]
Putting this together with the inequality in \ref{itm:energy_bound_cobordism:length}, we get the desired inequality for \ref{itm:energy_bound_symplectization:area}.
\[
E_\Omega(u;X) \ge - \f(a) \cdot \mathcal{L}(\Gamma_+;X) + \f(-a) \cdot \mathcal{L}(\Gamma_-;X) \ge -(\f(a) + \frac{\f'(a)\f({-}a)}{\f'(-a)}) \cdot \mathcal{L}(\Gamma_+;X) 
\] 
Finally, to prove \ref{itm:energy_bound_cobordism:theta}, choose $\phi \in \mathcal{S}$. Let $\psi$ be defined as the unique function with
\[
\frac{d\psi}{dr} = \phi \qquad\text{and}\qquad \psi(0) = 0
\]
Note that $\psi$ is monotonically increasing with limits $c_+$ and $c_-$ at $+\infty$ and $-\infty$ with $|c_\pm| \le 1$. Now we compute that
\[
\int_{\Sigma_+} u^*(\phi(r) \cdot dr \wedge \theta_+) = \int_{\Sigma_+} u^*(d(\psi(r)\cdot \theta_+) - \psi(r) \cdot d\theta_+) = c_+ \cdot \mathcal{L}(\Gamma_+) - \int_{\Gamma_+} u^*\psi \cdot c_X\cdot f \cdot \mu_\Sigma
\]
Note that $0 < c_+ \le 1$ and the latter integral is non-negative (as in Lemma \ref{lem:energy_bounds_symplectization}) since $J$ is cylindrical on $[0,\infty) \times \partial_+X$. Thus we have
\[
\int_{\Sigma_+} u^*(\phi(r) \cdot dr \wedge \theta_+) \le \mathcal{L}(\Gamma_+)  
\]
By a similar computation, we can bound the analogous integral over $\Sigma_-$ as follows.
\[
\int_{\Sigma_-} u^*(\phi(r) \cdot dr \wedge \theta_-) \le -c_- \cdot \mathcal{L}(\Gamma_-) - \int_{\Gamma_-} u^*\psi \cdot d\theta_- \le \mathcal{L}(\Gamma_-) + \int_{\Sigma_-} u^*d\theta_-
\]
Here we use the fact that $-1 \le c_- < 0$ and $-1 \le \psi \le 0$ on $(-\infty,0]$. The second term can be bounded as follows.
\[
\int_{\Sigma_-} u^*d\theta_- = c_X \cdot \int_{\Sigma_-} u^*\omega_- \le c_X \cdot \Big(\int_{\Sigma_+} u^*\omega_+ + \int_{\Sigma_-} u^*\omega_- + \int_{\Sigma_X} u^*\Omega\Big) = c_X \cdot E_\Omega(u) 
\]
Note that this inequality holds only for $J$ that are cylindrical outside of $X$, since then the three terms in the definition of $E_\Omega(u)$ are non-negative. We finally conclude that
\[
E_\Theta(u) = \sup_{\phi \in \mathcal{S}}\Big(\int_{\Sigma_+} u^*(\phi(r) \cdot dr \wedge \theta_+) + \int_{\Sigma_-} u^*(\phi(r) \cdot dr \wedge \theta_-)\Big) \le \mathcal{L}(\Gamma_+) + \mathcal{L}(\Gamma_-) + c_X \cdot E_\Omega(u)
\]
This concludes the proof of \ref{itm:energy_bound_cobordism:theta} and the desired result.\end{proof}

\begin{warning} We warn the reader that the $\Omega$-energy $E_\Omega(u)$ of a $J$-holomorphic curve $u:\Sigma \to \hat{X}$ can be negative for a general $J$ in the class $\mathcal{J}(X)$ introduced in Definition \ref{def:acs_cob}. This is due to the fact that our almost complex structures need only be cylindrical near infinity.
\end{warning}

Using the above, we can bound the number of connected components of holomorphic curves with bounded energy and period.

\begin{lemma} (Component Bound) \label{lem:component_bound} Let $(X,\Omega,Z)$ be a conformal symplectic cobordism with $[\Omega] \in H^2(X;\Q)$ and fix constants $E,T,a > 0$. Then there exists a constant
\[
N(\Omega,E,T,a) \qquad\text{depending only on}\qquad \Omega,E,T,a
\]
with the following property. Fix a pair $(J,u)$ of
\begin{itemize}
    \item a compatible almost complex structure $J \in \mathcal{J}(X)$ cylindrical outside of $X^a_{\smallneg a}$
    \item a $J$-holomorphic map $u:(\Sigma,j) \to (\hat{X},J)$ is  asymptotic to $\Gamma_\pm$ at $\pm \infty$ satisfying
    \[E_\Omega(u) \le E \qquad\text{and}\qquad \mathcal{L}(\Gamma_+) \le T\]
\end{itemize}
Then the number of components $n(\Sigma)$ of $\Sigma$ on which $u$ is non-constant is bounded by $N(
\Omega,E,T,a)$.
\end{lemma}

\begin{proof} We can assume that $u$ is non-constant on every component. Let $S \subset \Sigma$ be a connected component with ends $\Xi_\pm$. Then by Lemma \ref{lem:energy_bounds_cobordism}\ref{itm:energy_bound_cobordism:length}, we have
\[
\mathcal{L}(\Xi_-) \le \frac{\f'(a)}{\f'(-a)} \cdot \mathcal{L}(\Xi_+)  \qquad\text{for some }a > 0 
\]
It follows that $S$ cannot have only negative punctures (otherwise $\mathcal{L}(\Xi_-) \le0$) - it must either have a positive puncture or be closed. Thus we may divide $\Sigma$ as
\[
\Sigma = A \sqcup B
\]
where $A$ consists of closed components and $B$ consists of punctured components, each of which has at least one positive puncture. The number of components $n(B)$ in $B$ is bounded by
\begin{equation} \label{eqn:component_bound_1}
n(B) \le \frac{T}{T_{\on{min}}}
\end{equation}
where $T_{\on{min}}$ is the minimum period of an orbit in $\partial_+X$. Similarly, the area of any closed curve is bounded below by
\[
\hbar(\Omega) := \on{min}\big\{{\Omega}\cdot A  \; : \;A \in H_2(X;\Z)\text{ such that } A \cdot {\Omega} > 0\big\}
\]
Since $[\Omega] \in H_2(X;\Q)$, this quantity is positive. By Lemma \ref{lem:energy_bounds_cobordism}\ref{itm:energy_bound_cobordism:area}, we have
\[
E_\Omega(u|_B) \ge -(\f(a) + \frac{\f'(a)\f(a)}{\f'(-a)}) \cdot \mathcal{L}(\Gamma_+)
\]
Therefore the number of components $n(A)$ is bounded as follows.
\begin{equation} \label{eqn:component_bound_2}
\hbar(\Omega) \cdot n(A) - (\f(a) + \frac{\f'(a)\f(a)}{\f'(-a)}) \cdot \mathcal{L}(\Gamma_+) \le E_\Omega(u|_A) + E_\Omega(u|_B) = E_\Omega(u) \le E
\end{equation}
It follows that $n(A)$ is bounded by a constant depending only on $\Omega,a,T$ and $E$. Combining (\ref{eqn:component_bound_1}) and (\ref{eqn:component_bound_2}) yields the desired result. \end{proof}

\subsection{Holomorphic Buildings And SFT Compactness} We are now prepared to formulate the core geometric result of symplectic field theory: SFT compactness, due to Bourgeois-Eliashberg-Hofer-Wysocki-Zehnder \cite{sftCompactness}.

\vspace{3pt}

We first require the notion of broken symplectic cobordism and almost complex structures.

\begin{definition} \label{def:broken_cobordism} A \emph{broken symplectic cobordism} $X$ is a symplectic cobordism with an isomorphism
\[X \simeq (W)_M\]
identifying $X$ with the trace of a cobordism $W$ along components $M \simeq M_\pm \subset \partial_\pm X$. 
\end{definition}

\begin{definition} \label{def:broken_J} Let $X = (W)_M$ be a broken symplectic cobordism where the boundary $\partial W$ has conformal Hamiltonian components. 

\vspace{3pt}

A \emph{broken} almost complex structure on $X = (W)_M$ is a compatible almopst complex structure $J$ on the completion of $W$ such that
\begin{itemize}
\item The asymptotic complex structures $J_{M_+}$ and $J_{M_-}$ are equal under the map $M_+ \simeq M_-$.
\vspace{3pt}
\item $J$ is cylindrical on an open neighborhood in $\hat W$ of the cylindrical ends $[0,\infty) \times M_+$ and $(-\infty,0] \times M_-$.
\end{itemize}
\end{definition}

\begin{definition} \label{def:symplectic_building_symplectization} A \emph{$J$-holomorphic building} ${\bf u}$ in a symplectization $\R \times Y$ with a compatible almost complex structure $J$ is a sequence of proper, finite energy $J$-holomorphic curves
\[{\bf u} = (u_1,\dots,u_m) \qquad\text{satisfying}\qquad u_i \to \Gamma^\pm_i \text{ at }\pm \infty\]
with the property that the negative end $\Gamma^-_i$ of $u_i$ is equal to the positive end $\Gamma^+_{i+1}$ of $u_{i+1}$. We say that $\Gamma^+ = \Gamma^+_1$ and $\Gamma^- = \Gamma^-_m$ are the positive and negative ends of ${\bf u}$, respectively.
\end{definition}

\begin{definition} \label{def:symplectic_building_cobordism} A \emph{$J$-holomorphic building} ${\bf u}$ in a broken symplectic cobordism $X \simeq (W)_M$ with respect to a broken almost complex structure $J$ on $X$ consists of
\begin{itemize}
    \item A (proper, finite energy) $J$-holomorphic curve $u_X$ from $\Gamma^+_X$ to $\Gamma^-_X$. The restriction of $u_X$ to a component $V \subset X$ is called the \emph{cobordism level} $u_V:\Sigma_V \to X$.
    \vspace{3pt}
    \item A $J_N$-holomorphic building ${\bf u}_N$ in the symplectization $\R \times N$ for each component $N \subset \partial W$. A constituent map $u_{N,i}:\Sigma_{N,i} \to \R \times N$ of ${\bf u}_N$ is called a \emph{symplectization level} of ${\bf u}$ in $N$.
\end{itemize}
that satisfy the following compatibility assumptions.
\begin{itemize}
    \item If $N_+$ and $N_-$ are identified with the same component of $N \subset M$ under the isomorphisms $M \simeq M_+ \simeq M_-$, then ${\bf u}_{N_+}$ and ${\bf u}_{N_-}$ are identified with a single building ${\bf u}_N$ in $\R \times N$.
    \vspace{3pt}
    \item The ends of $u_X$ and ${\bf u}_N$ match up in the following sense. If $\Gamma^\pm_N$ denote the ends of ${\bf u}_N$ and $\Gamma^\pm_{X,N} \subset \Gamma^\pm_X$ denote the ends of $u_X$ in $N$ for $N \subset M$, then
    \[\Gamma^+_N = \Gamma^-_{X,N} \qquad\text{and}\qquad \Gamma^-_N = \Gamma^+_{X,N}\]
\end{itemize}
The \emph{domain} ${\bf \Sigma}$ of ${\bf u}$ is the smooth surface acquired by compactifying punctures of the domains of $u_X$ and $u_{N,i}$ to boundary components and gluing together pairs of such circles corresponding to matched ends of ${\bf u}$. The \emph{special locus} ${\bf \Gamma} \subset {\bf \Sigma}$ is the union of these circles. There is a natural map
\[\pi:{\bf \Sigma} \to ({\bf \Sigma}_\star,{\bf j})\]
to a nodal Riemann surface $({\bf \Sigma}_\star,{\bf j})$ that is a biholomorphism on the domain $\Sigma \subset {\bf \Sigma}$ of any level.

\vspace{3pt}

The \emph{$\Omega$-energy} of a building ${\bf u}$ is simply the sum of the areas of the cobordism and symplectization levels.
\[
E_\Omega({\bf u}) = E_\Omega(u_X) + \sum_N \big(\sum_i E_\omega(u_{N,i})\big)
\]
Note that the levels $u_{N,i}$ of the building ${\bf u}_N$ associated to a component of the trace locus $M$ are only counted once in this sum. \end{definition}

The following lemma is an immediate consequence of Lemmas \ref{lem:energy_bounds_symplectization} and \ref{lem:energy_bounds_cobordism}. that we will use repeatedly in later sections.

\begin{lemma} \label{lem:length_bound_building} Let $X$ be a conformal symplectic cobordism. Let ${\bf u}$ be a $J$-holomorphic building in $X$ asymptotic to $\Gamma_\pm$ at $\partial_\pm X$, and let $\Xi_\pm$ denote the ends of the cobordism level $u_X$. Assume that every level of ${\bf u}$ has positive $\Omega$-energy. Then
\[\mathcal{L}(\Xi_-) \le \mathcal{L}(\Xi_+) \le \mathcal{L}(\Gamma_+)\]
\end{lemma}

\begin{definition}[BEHWZ Convergence] \label{def:BEHWZ_convergence} Let $X$ be a conformal symplectic cobordism and fix a sequence of almost complex structures
\[J_k \in \mathcal{J}(X) \qquad\text{with}\qquad J_k \to J \quad \text{ $C^\infty$-uniformly}\]A sequence of of proper, finite energy $J_k$-holomorphic maps $u_k:(\Sigma_k,j_k) \to (\hat{X},J_i)$ \emph{converges in the BEHWZ sense} to a $J$-holomorphic building ${\bf u}$ in the broken cobordism $X$ if there exists
\begin{itemize}
    \item A sequence of smooth maps $\varphi_k:{\bf \Sigma} \to \Sigma_k$ from the domain ${\bf \Sigma}$  of ${\bf u}$ for $k \in \N$
\end{itemize}
such that the following conditions hold.
\begin{itemize}
    \item $({\bf \Sigma},\varphi^*_k j_k)$ converges to $({\bf \Sigma}_\star,{\bf j})$ in Deligne-Mumford space after stabilization.
    \vspace{3pt}
    \item The sequence $u_k \circ \varphi_k:\Sigma_X \to X$ converges in $C^\infty_{\on{loc}}$  to $u_X$ on the domain $\Sigma_X \subset {\bf \Sigma}$ of $u_X$.
    \vspace{3pt}
    \item The sequence $u_k \circ \varphi_k:\Sigma_{N,i} \to X$ is contained in the cylindrical end corresponding to $N \subset \partial X$ for each component $N$ of $\partial X$ and for large $k$.
    \vspace{3pt}
    \item There is a divergent sequence $c_k \in \R$ such that the sequence of maps
    \[v_k:\Sigma_{N,i} \xrightarrow{u_k \circ \varphi_k} \R \times N \xrightarrow{ \cdot + c_k} \R \times N\]
    converges in $C^\infty_{\on{loc}}$ to $u_{N,i}$, for each component $N \subset \partial X$ and each level $u_{N,i}$.
    \vspace{3pt}
    \item The area (i.e. the $\Omega$-energy) of $u_k$ converges to the area of ${\bf u}$.
\end{itemize}
\end{definition}

We now state the formulation of SFT compactness that we will use in the rest of the paper. The main difference between our statement and the standard one in \cite{sftCompactness} is our use of a broader class of almost complex structures.

\begin{thm}[SFT Compactness] \label{thm:SFT_compactness} Let $X$ be a conformal symplectic cobordism and fix a sequence $J_k$ converging $C^\infty$-cylindrically to $J$. That is, there is an $a > 0$ such that
\[J_k \to J \quad \text{ $C^\infty$-uniformly} \qquad\text{and}\qquad J_k \text{ cylindrical outside of }X^a_{\smallneg a} \subset \hat{X}\]
Let $u_k:(\Sigma_k,j_k) \to (\hat{X},J_k)$ be a sequence of proper, finite energy $J_k$-holomorphic curves with ends $\Gamma_\pm^k$ with
\begin{itemize}
\item The total genus $g(\Sigma_k)$ and number of components $n(\Sigma_k)$ are bounded uniformly  in $k$.
\vspace{3pt}
\item The $\Omega$ energy $E_\Omega(u_k)$ and length $\mathcal{L}(\Gamma_+^k)$ are bounded uniformly  in $k$.
\vspace{3pt}
\end{itemize}
Then there is a subsequence of $u_k$ that is BEHWZ convergent to a $J$-holomorphic building ${\bf u}$.
\end{thm} 

\begin{proof} From the identification of completions $\hat{X} = \hat{X}^a_{\smallneg a}$, we may view the complex structures $J_k$ as a sequence of almost complex structures
\[J_k \in \mathcal{J}(X^a_{\smallneg a}) \quad\text{cylindrical on } \quad \hat{X}^a_{\smallneg a} \setminus X^a_{\smallneg a} = (-\infty,0] \times \partial_-X^a_{\smallneg a} \cup [0,\infty) \times \partial_+X^a_{\smallneg a}\]
 By standard SFT compactness \cite[Thm 10.2]{sftCompactness}, the sequence $u_k$ BEHWZ converges if the number of components $n(\Sigma_k)$, the genus $g(\Sigma_k)$, the $\Omega$-energy $E_\Omega(u_k,X^a_{\smallneg a})$ and the $\Theta$-energy $E_\Theta(u_k,X^a_{\smallneg a})$ are all uniformly bounded in $k$. Note that we must use the energy with respect to the thickening $X^a_{\smallneg a}$ to directly apply \cite[Thm 10.2]{sftCompactness}. We will carefully distinguish between the different energies below (see Remark \ref{rmk:different_energies_notation}). 

 \vspace{3pt}

 The genera and number of components are assumed to be bounded. To prove that the thickened energies are bounded, choose $E > 0$ so that
 \[
 E_\Omega(u_k;X) \le E \qquad\text{and}\qquad \mathcal{L}(\Gamma_+^k;X) \le E  \qquad\text{for all $k$}
 \]
First, note that the lengths before and after thickening are related by
 \[\mathcal{L}(\Gamma_+^k;X^a_{\smallneg a}) = \f'(a) \cdot \mathcal{L}(\Gamma_+^k;X) \qquad \mathcal{L}(\Gamma_-^k;X^a_{\smallneg a}) = \f'(-a) \cdot \mathcal{L}(\Gamma_-^k;X)\]
In particular, by Lemma \ref{lem:energy_bounds_cobordism}(a) and (b), we have the following inequality
 \[
\f'(a) \cdot \mathcal{L}(\Gamma_+^k;X) - \f'(-a) \cdot \mathcal{L}(\Gamma_-^k;X)  = c_X \cdot E_\Omega(u_k;X^a_{\smallneg a})\qquad\text{and thus}\qquad \frac{\f'(a)}{\f'(-a)} \cdot \mathcal{L}(\Gamma_+^k;X) \ge \mathcal{L}(\Gamma_-^k;X)
 \]
We can then apply Lemma \ref{lem:area_change_under_thickening} to see that the thickened $\Omega$-energy is uniformly bounded.
\[
E_\Omega(u_k;X^a_{\smallneg a}) \le E_\Omega(u_k;X) + \f(a) \cdot \mathcal{L}(\Gamma_+^k) - \f(- a) \cdot \mathcal{L}(\Gamma_-^k;X)\]
\[ \le E_\Omega(u_k;X) + (1 + \frac{-\f(-a) \cdot \f'(a)}{\f'(-a)}) \cdot \mathcal{L}(\Gamma_+^k;X)  \le C(E,a)
\]
Here $C(E,a)$ is a constant depending only on $E$ and $a$. Similarly, we apply Lemma \ref{lem:energy_bounds_cobordism}(c) to bound the $\Theta$-energy as follows.
\[
E_\Theta(u_k;X^a_{\smallneg a}) \le \mathcal{L}(\Gamma_+^k;X^a_{\smallneg a}) + \mathcal{L}(\Gamma_-^k;X^a_{\smallneg a}) + c_X \cdot E_\Omega(u_k;X^a_{\smallneg a})\]
\[\le 2\f'(a) \cdot \mathcal{L}(\Gamma_+^k;X) + c_X \cdot C(E,a) \le D(E,a)
\]
Here $D(E,a)$ is another constant depending only on $E$ and $a$. This proves the required energy bound and concludes the proof. \end{proof}

\subsection{SFT Neck Stretching} We next introduce the notion of a neck-stretching family of almost complex structures and the corresponding SFT compactness statement. 

\begin{definition}[Neck-Stretching Family] \label{def:neck_stretching_family} Let $X \simeq (W)_M$ be a broken symplectic cobordism with conformal boundary components. Fix a broken almost complex structure $J$ and a tubular neighborhood
\[\jmath:(-\epsilon,\epsilon) \times M \hookrightarrow \hat X \quad\text{with}\quad \jmath^*\Omega = \hat{\omega}_M\]
A \emph{stretching profile} is a smooth map $\mathfrak{st}:[0,\infty)_s \times \R_r \to \R$ that is increasing in $r$ and that satisfies
\[\mathfrak{st}_s(r) = \log(|r|) + \on{sgn}(r) \cdot (s + 1) \qquad\text{if}\qquad |\log(|r|)| \le s\]
The \emph{neck-stretching family} associated to $J$ and a choice of stretching profile $\mathfrak{st}$, denoted by
\[J_s \in \mathcal{J}(X) \qquad\text{for $s$ sufficiently large}\]
is the smooth family of compatible almost complex structures on $X$ defined as follows. Let $U:=\on{im}(\jmath)$ denote the tubular neighborhood of $M$ determined above. We have a map
\[\Psi_s:U \xrightarrow{\jmath^{-1}} (-\epsilon,\epsilon) \times M \xrightarrow{\mathfrak{st}_s \times \on{Id}} \R \times M \qquad\text{for each }s \in [0,\infty) \]
Let $J_M$ be the restriction of $J$ to the cylindrical ends over $M$. We now define
\[
J_s = \Psi_s^*J_M \quad\text{on }U \qquad \text{and}\qquad J_s = J \quad\text{on }X \setminus U
\]
This family of almost complex structures is well-defined and smooth for large $s$.\end{definition}

\begin{definition} \label{def:neck_stretching_convergence} Let $X = (W)_M$ be a broken symplectic cobordism.  Fix a broken almost complex structure $J$ and a stretching profile $\mathfrak{st}$. Let $I_s$ the corresponding neck-stretching family and let
\[J_k := I_{s_k} \in \text{$\mathcal{J}(X)$} \qquad\text{for a sequence}\qquad s_k \in \R \quad \text{with} \quad s_k \to \infty \]
A sequence of proper, finite energy $J_k$-holomorphic maps $u_k:(\Sigma_k,j_k) \to (\hat{X},J_k)$ \emph{converges in the BEHWZ sense} to a $J$-holomorphic building ${\bf u}$ in the broken cobordism $X$ if
\begin{itemize}
    \item A sequence of smooth maps $\varphi_k:{\bf \Sigma} \to \Sigma_k$ to the domain of ${\bf u}$ of ${\bf \Sigma}$ for $k \in \N$
\end{itemize}
such that the conditions of Definition \ref{def:BEHWZ_convergence} hold and additionally
\begin{itemize}
    \item The sequence $u_k \circ \varphi_k:\Sigma_{N,i} \to \hat X$ is contained in  the tubular neighborhood $U=\on{im}(\jmath)$ of $N \subset M$ for large $k$ and there is a 
divergent sequence $c_k \in \R$ such that the maps
    \[v_k:\Sigma_{N,i} \xrightarrow{u_k \circ \varphi_k} U \xrightarrow{\Psi_{s_k}} \R \times N \xrightarrow{ \cdot + c_k} \R \times N\]
    converge in $C^\infty_{\on{loc}}$ to the level $u_{N,i}$, for each component $N \subset M$ and each level $u_{N,i}$.
\end{itemize}
\end{definition}

\begin{thm}[Neck-Stretching Compactness] \label{thm:neck_stretching_compactness} Let $X = (W)_M$ be a broken symplectic cobordism such that either
\begin{itemize}
    \item the cobordism $X$ is conformal (i.e. admits a conformal vector-field $Z$) or
    \item the cobordism $X$ is closed (i.e. has no boundary).
\end{itemize}
Fix a broken almost complex structure $J$ and let $I_s$ be the corresponding neck-stretching family for some stretching profile. Let
\[u_k:(\Sigma_k,j_k) \to (\hat{X},J_k)\] be a proper, finite energy  $J_k$-holomorphic map for a subsequence $J_k = I_{s_k}$ of the family $I_s$ with $s_k \to \infty$. Assume that 
\vspace{3pt}
\begin{itemize}
\item The total genus $g(\Sigma_k)$ and number of components $n(\Sigma_k)$ is uniformly bounded in $k$.
\vspace{3pt}
\item The $\Omega$-energy $E_\Omega(u_k)$ and length $\mathcal{L}(\Gamma^+_k)$ are uniformly bounded in $s$.
\vspace{3pt}
\end{itemize}
Then there is a subsequence of $u_k$ that is BEHWZ convergent to a $J$-holomorphic building ${\bf u}$ in the broken cobordism $X=(W)_M$ (see Definition \ref{def:broken_cobordism}).
\end{thm}

\begin{proof} In the case where $X$ is conformal, we note that the neck stretching family $I_s$ is cylindrical on $\hat{X} \setminus X^a_{\smallneg a}$ for any $s$ and some $a > 0$. Thus as in Theorem \ref{thm:SFT_compactness}, we have uniform estimates
\[
E_\Omega(u_k;X^a_{\smallneg a}) \le C(E,a) \qquad\text{and}\qquad E_\Theta(u_k;X^a_{\smallneg a}) \le D(E,a)
\]
for constants $C$ and $D$ depending only on the bound $E$ on $E_\Omega(u_k)$ and $\mathcal{L}(\Gamma^+_k)$, and on the constant $a$. By the ordinary SFT neck stretching result \cite[Thm 10.3]{sftCompactness}, the curves $u_k$ BEHWZ converge to ${\bf u}$ in the sense of Definition \ref{def:neck_stretching_convergence} (or equivalently \cite[\S 9.1]{sftCompactness}). 

\vspace{3pt}

In the case where $X$ is closed, this is simply \cite[Thm 10.3]{sftCompactness} in the closed setting since the $\Omega$-energy is simply the area and there is no $\Theta$-energy.\end{proof}

\begin{remark} Note that the broken pieces of $X = (W)_M$, i.e. the connected components of the cobordism $W$, do not individually need to be conformal for us to apply Theorem \ref{thm:neck_stretching_compactness}. 
\end{remark}



\subsection{Gromov-Witten Invariants} We conclude this section by briefly reviewing some background from Gromov-Witten theory.

\vspace{3pt}

Let $(X,\Omega)$ be a closed symplectic manifold (or more generally, a symplectic orbifold). Fix a curve type $\sigma = (g,m,A)$ of integers $g,m \ge 0$ and a class $A \in H_2(X)$ (see Definition \ref{def:type}). Let
\[\bar{\mathcal{M}}_{g,m}\]
denote the compactified Deligne-Mumford space of genus $g$ curves with $m$ marked points. The \emph{Gromov-Witten class} is a homology class
\[\GW_\sigma(X,\Omega) \in H_{\on{vdim}(\sigma)}(X^m \times \bar{\mathcal{M}}_{g,m};\Q)\]
It is, roughly speaking, the pushforward of the fundamental class of the moduli space of $J$-holomorphic curves $C \subset X$ with respect to a compatible almost complex structure $J$, via natural evaluation maps. Given a homology class $D \in H_*(X^k)$ for $k \le m$, we let
\[
\GW_\sigma(X,\Omega) \cap D \in H_*(X^{m-k} \times \bar{\mathcal{M}}_{g,m};\Q)
\]
denote the cup pairing of the Gromov-Witten class and $D$ in the first $k$ factors of $H_*(X^m)$.

\vspace{3pt}

\begin{remark} Several rigorous constructions of the Gromov-Witten invariants of a general closed symplectic manifold have appeared, using both polyfold theory  \cite{hofer2017applications,schmaltz2018gromov}, virtual fundamental cycle methods \cite{pardon2016algebraic} and global Kuranishi charts \cite{hs2022,h2023}. The orbifold Gromov-Witten invariants were introduced by Chen-Ruan \cite{chen2002orbifold}. We mark any result that relies on orbifold Gromov-Witten theory by Assumption~\ref{ass:GW_Orbifolds}.
\end{remark}

\vspace{3pt} 

We will need several axiomatic properties of the invariants (see \cite{pardon2016algebraic,h2023,schmaltz2018gromov} for proofs of these axioms in the manifold case and \cite{chen2002orbifold} for a discussion of the orbifold case).

\begin{prop} \label{prop:GW_axioms} The Gromov-Witten invariants $\GW$ satisfy the following axioms.
\begin{itemize}
    \item \label{itm:GW_deformation} (Deformation) Let $\Omega_t$ for $t \in [0,1]$ be a family of symplectic forms on $X$. Then
    \[\GW_\sigma(X,\Omega_0) = \GW_\sigma(X,\Omega_1) \qquad\text{for any curve type }\sigma\]
    \item \label{itm:GW_curve} (Curve) Let $B_i \in H_*(X)$ for $i = 1,\dots,k$ be homology classes, and suppose that
    \[\GW_\sigma(X) \cap (B_1\otimes \dots \otimes B_k) \neq 0 \in H_*(X^{m - k} \times \bar{\mathcal{M}}_{g,m};\Q)\]
    Then for every compatible almost complex structure $J$ on $X$ and any set of sub-manifolds $S_i$ with $[S_i] = B_i$, there is a $J$-holomorphic curve
    \[C \subset X \quad\text{with}\quad g(C) = g \quad C \cap S_i \neq 0 \quad\text{and}\quad [C] = A\]
    \item \label{itm:GW_zero} (Zero) Let $\sigma = (0,m,0)$ be a curve type with genus zero and zero homology class $A = 0 \in H_2(X)$. Let $X_{\on{diag}} \subset X^m$ be the diagonal copy of $X$.
    \[\GW_\sigma(X,\Omega) =  [\bar{\mathcal{M}}_{0,m} \times X_{\on{diag}}] \in H_*(\bar{\mathcal{M}}_{g,m}  \times X^m;\Q)\]
    This formulation is given e.g. in \cite[Section 1.3]{schmaltz2018gromov}.
    \vspace{3pt}
    \item \label{itm:GW_product} (Product) Let $X_1$ and $X_2$ be two symplectic manifolds and fix two homology classes $A_i \in H_2(X_i)$. Fix curve types
    \[\sigma_1 = (g,m,A_1)\qquad \sigma_2 = (g,m,A_2) \quad\text{and}\quad \sigma = (g,m,A_1 \oplus A_2)\]
    Then the Gromov-Witten classes are related by
    \[\GW_\sigma(X_1 \times X_2) = [\Delta] \cap (\GW_{\sigma_1}(X_1) \otimes \GW_{\sigma_1}(X_2)) \in H_*(\bar{\mathcal{M}}_{g,m} \times X_1 \times X_2)\]
    Here $[\Delta]$ is the fundamental class of the diagonal in $\Delta \subset \bar{\mathcal{M}}_{g,m} \times \bar{\mathcal{M}}_{g,m}$. This property can be found in e.g. \cite[Section 6]{hs2022}.
\end{itemize}
\end{prop}

\noindent Many basic properties of Gromov-Witten invariants can be easily deduced from these axioms. 

\vspace{3pt}

One quick application is the following uniruledness property for products of the form $S^2 \times X$. We will use this property to estimate spectral gaps later in the paper (see Section \ref{sec:gaps_of_periodic_SHS}). See McDuff \cite{mcduff2009hamiltonian} for another proof using more standard enumeration methods.

\begin{lemma}[Sphere Product] \label{lem:GW_axiom_S2_times_X} Let $X$ be a closed symplectic $2n$-manifold and consider the curve class
\[
\sigma = (0,3,A) \qquad\text{where}\qquad A = [S^2] \times 0 \in S^2 \times X
\]
Then $\GW_\sigma(S^2 \times X) \cap [\on{pt}] \neq 0$. \end{lemma}

\begin{proof} Note that $\bar{\mathcal{M}}_{0,3}$ is a point, since any biholomorphism of $S^2$ is determined by where it sends three points. It suffices to show that the corresponding Gromov-Witten class for $S^2 \times X$ is given by
\[\GW_\sigma(S^2 \times X) = [S^2]^{\otimes 3} \otimes [X_{\text{diag}}] \in H_{2n}((S^2)^3 \times X^3) = H_{2n}(\bar{\mathcal{M}}_{0,3} \times X^3)\]
Here $[X_{\text{diag}}]$ is the fundamental class of the diagonal copy of $X$ in $X^3$. The lemma then follows. 

\vspace{3pt}

To prove this claim, let $\mu = (0,3,[S^2])$ and $\nu = (0,3,0)$ be curve classes in $S^2$ and $X$. By the product axiom
\[
\GW_\sigma(S^2 \times X) = [\Delta] \cap (\GW_{\mu}(S^2) \otimes \GW_{\nu}(X)) = \GW_\sigma(S^2 \times X) = \GW_\mu(S^2) \otimes  \GW_\nu(X)
\]
Here $[\Delta] = [\bar{\mathcal{M}}_{0,3}] = [\text{pt}]$ since $\bar{\mathcal{M}}_{0,3}$ is a point. Since the genus is zero, the zero axiom gives
\[ \GW_\nu(X) = [\bar{\mathcal{M}}_{0,3} \times X_{\on{diag}}] = [X_{\on{diag}}]\]
Finally, the Gromov-Witten invariant $\GW_\mu(S^2)$ is a count of biholomorphisms $S^2 \to S^2$ with three point constraints. This space is equivalent to $S^2 \times S^2 \times S^2$ and the Gromov-Witten invariant is precisely
\[
 \GW_\mu(S^2) = [S^2 \times S^2 \times S^2] \in H_6((S^2)^3 \times \bar{\mathcal{M}}_{0,3}) \qedhere
\]
\end{proof}

It will be useful to have a related result on the Gromov-Witten invariants of certain orbifold sphere bundles. A similar construction was carried out in \cite{abbondandolo2023symplectic} in the smooth manifold case.

\begin{definition} \label{def:associated_sphere_bundle} Let $(Y,\xi)$ be a closed contact manifold with a periodic contact form $\alpha$ and fix $\epsilon > 0$. The \emph{associated sphere bundle} is the symplectic orbifold
\[
\Sigma Y := [0,\epsilon]_s \times Y/\sim
\]
where $\sim$ is the equivalence relation that quotients $\{0\} \times Y$ and $\{\epsilon\} 
\times Y$ by the $S^1$-action induced by the periodic Reeb flow of $Y$. This space has a natural symplectic structure $\Omega_\epsilon$ (cf. \cite{lerman1995symplectic}) with
\[
\Omega|_{(0,\epsilon) \times Y} = e^s\alpha
\]
There is a natural symplectic submersion of orbifolds to the quotient of $Y$ by the periodic Reeb flow.
\[
\Sigma Y \to Y/S^1
\]
\end{definition}

\begin{lemma} \label{lem:GW_axiom_orbifold_sphere_bundle} Assume~\ref{ass:GW_Orbifolds}. Let $Y$ be a contact manifold with periodic Reeb flow and let $\pi:\Sigma Y \to Y/S^1$ be the associated sphere bundle. Let $A = [F]$ be the homology class of a generic fiber $F$ of $\pi$. Then
\[
\GW_\sigma(\Sigma Y) \cap [\on{pt}] \neq 0 \qquad\text{where}\qquad \sigma = (0,1,A)
\]
\end{lemma}

\begin{proof} Let $F$ be a fiber lying over a manifold point in $Y/S^1$. There is a neighborhood $U$ of $F$ with
\[
U \simeq F \times B^{2n-2}(r) \qquad\text{for small }r > 0
\]
where $\pi|_U$ is modelled on the projection $F \times B^{2n-2}(r) \to B^{2n-2}(r)$. Choose compatible almost complex structures $J$ on $\Sigma Y$ and $I$ on $Y/S^1$ so that the projection $\pi$ is $(J,I)$-holomorphic and such that
\[
J|_U = j_F \oplus J_{\on{std}}
\]
where $j_F$ is a complex structure on $F \simeq S^2$ and $J_{\on{std}}$ is the standard almost complex structure. By these choices, if $p \in F$ is any point, then $F$ is the unique $J$-holomorphic sphere in the class $[F]$ containing $p$. 

\vspace{3pt}

Now we claim that $F$ has a non-zero virtual count. Indeed, one may check that $F$ is transversely cut out by our choice of almost complex structure (cf. McDuff-Salamon \cite[Lemma 3.3.1]{mcduff2012j}). Alternatively, the contribution of the curve $F$ to the virtual count in $\GW_\sigma$ must be non-zero due to the computation in Lemma \ref{lem:GW_axiom_S2_times_X} and due to the fact that such counts are local in the space of maps. This shows that $\GW_\sigma(\Sigma Y) \neq 0$, concluding the proof. \end{proof}

\section{Elementary SFT Spectral Gaps}\label{sec:spectral_gaps}
In this section, we formulate the elementary symplectic field theory (ESFT) spectral gaps and prove their various formal properties.

\subsection{Moduli Spaces} We start by describing the moduli spaces that we will use in our construction. The moduli spaces (and corresponding spectral gaps) will be indexed by abstract curve types that track the topological data of the curves. 

\begin{definition}[Curve Types] \label{def:type} An abstract \emph{curve type} $\sigma$ in a symplectic cobordism $X$ consists of
\begin{itemize}
    \item the \emph{genus} $g$ of $\sigma$, a non-negative integer representing the genus of a surface.
    \vspace{2pt}
    \item the \emph{point number} $m$ of $\sigma$, a non-negative integer representing a set of points in a surface.
    \vspace{2pt}
    \item the \emph{homology class} $A \in H_2(X,\partial X)$ of $\sigma$, representing the fundamental class of a surface.
\end{itemize}
The set of all curve types in the symplectic cobordism $X$ is denoted as follows.
\[\mathcal{S}(X) := \big\{\text{curve types $\sigma = (g,m,A)$ on $X$}\big\}\]
This set has the structure of a partially ordered, abelian monoid with partial order
\[
(g,m,A) \preceq (h,n,B) \qquad\text{if}\qquad g \le h,\quad m \ge n  \quad\text{and}\quad A=B.
\]
The commutative, associative and order-preserving \emph{union} operation is given by
\[
\mathcal{S}(X) \times \mathcal{S}(X) \xrightarrow{+} \mathcal{S}(X) \qquad\text{given by}\qquad (g,m,A) + (h,n,B) = (g + h, m + n, A + B)
\]
The \emph{restriction} $\mathcal{S}(X) \to \mathcal{S}(U)$ of an inclusion $U \subseteq X$ of symplectic cobordisms is given by
\[
\sigma \mapsto \sigma|_U := (g,m,A|_U) 
\]
The \emph{pushforward} $\iota_U:\mathcal{S}(U) \to \mathcal{S}(X)$ by the inclusion $U \subset X$ of a component, given by
\[
\sigma \mapsto \iota_U(\sigma) = (g, m, \iota_*A)
\]
Here $A|_U = \iota^!A$ and $\iota_*A$ are the images under compact pullback and pushforward, i.e. the maps
\[
\iota^!:H_2(X,\partial X) \xrightarrow{\on{PD}} H^2(X) \xrightarrow{\iota^*} H^2(U) \xrightarrow{\on{PD}} H_2(U,\partial U) \qquad\text{and} \qquad \iota_*:H_2(U) \to H_2(X)
\]\end{definition}

It is often important to restrict our discussion to curve types that have at least one point. We fix the following terminology for later use.

\begin{definition} A curve type $\sigma$ is \emph{pointed} if the point number $m$ is positive. Otherwise, we will call $\sigma$ \emph{pointless}.
\end{definition}

\begin{definition} \label{def:moduli_spaces}  Let $X$ be a conformal symplectic cobordism and $\sigma = (g,m,A)$ be a curve type in $X$. Fix a set of $m$ points $P \subset \hat X$ and an almost-complex structure $J \in \mathcal{J}(X)$. We let
\[\mathcal{M}_\sigma(X;J,P)\]
be the moduli space of equivalence classes of tuples $(\Sigma,j,u,S)$ consisting of the following data.
\begin{itemize}
    \item a punctured, possibly disconnected or empty, Riemann surface $(\Sigma,j)$.
    \vspace{3pt}
    \item a set of $m$ distinct points $S \subset \Sigma$. 
    \vspace{3pt}
    \item a non-constant, proper, finite energy $(j,J)$-holomorphic map $u:\Sigma \to \hat{X}$ with $u(S) = P$.
    \vspace{3pt}
\end{itemize}
The tuple $(\Sigma,j,S)$ must also satisfy the following two properties.
\begin{itemize}
    \item The sum $g(C_1) + \dots + g(C_k)$ of the genera of the components is less than or equal to $g$.
    \vspace{3pt}
    \item The relative homology class $[u] = u_*[\Sigma]$ in $H_2(X,\partial X)$ must be equal to $A$. If $\Sigma$ is empty, we adopt the convention that $u_*[\Sigma] = 0$.
    \vspace{3pt}
\end{itemize}
We say that $(\Sigma,j,u,S)$ is equivalent to $(\Sigma',j',u',S')$ if there is a biholomorphism $\varphi:\Sigma \simeq \Sigma'$ such that $u' \circ \varphi = u$ and $\varphi(S) = S'$. We let
\[
\mathcal{M}^T_\sigma(X;J,P) \subset \mathcal{M}_\sigma(X;J,P) 
\]
denote the subset of curves with positive ends $\Gamma_+$ of total period $\mathcal{L}(\Gamma_+)$ bounded by $T$.
\end{definition}

 \begin{remark}\label{rem:empty_domain_for_pointless} We emphasize that the empty curve is a permissible element of the moduli spaces in Definition \ref{def:moduli_spaces} if the point number is zero and homology class vanishes. This is a key requirement for various later statements.
\end{remark}

\subsection{Gaps Parametrized By Complex Structures And Points} In order to formulate spectral gaps, we first consider a version depending on a choice of a complex structure and points. 

\begin{definition}\label{def:def_gap_JP} Fix a conformal symplectic cobordism $(X,\Omega,Z)$ with Morse-Bott ends, and choose a point set and a complex structure
\[P \in \hat{X}^m \qquad\text{and}\qquad J \in \mathcal{J}(X,\Omega,Z)\]
The \emph{elementary symplectic field theory spectral gap} of $(X,J,P)$ with curve type $\sigma$ and period $T$ is the infimum of areas among $J$-curves of type $\sigma$ with ends of period $T$ or less.
\[\mathfrak{g}_{\sigma,T}(X,\Omega,Z;J,P) := \on{inf} \big\{E_\Omega(u)\; : \; u \in \mathcal{M}_\sigma^T(X;J,P)\big\}\]\end{definition}

These spectral gaps enjoy a number of continuity properties with respect to the choice of almost complex structure and point set. We detail these in the following lemma.

\begin{lemma}\label{lem:minimizer_and_JP_semicont} Let $(X,\Omega,Z)$ be a conformal symplectic cobordism with Morse-Bott ends and $[\Omega] \in H^2(X;\Q)$. Fix a curve type $\sigma = (g,m,A)$ and a period $T$. Then
\vspace{3pt}
\begin{enumerate}[label=(\alph*)]

\item \label{itm:axiom_minimizer} (Minimizer) If $\mathcal{M}_{\sigma}^T(X;J,P)$ is non-empty, then there is a curve $u \in \mathcal{M}_{\sigma}^T(X;J,P)$
\[E_\Omega(u) = \g_{\sigma,T}(X,\Omega,Z;J,P)\]

\vspace{3pt}

\item \label{itm:axiom_semicontinuity} (Semi-Continuity) The spectral gap is a lower-semicontinuous map 
\[\mathfrak{g}_{\sigma,T}(X,\Omega,Z;-):\mathcal{J}(X,\Omega,Z) \times X^m \to (-\infty,\infty]\]
where $\mathcal{J}(X,\Omega,Z)$ is equipped with the $C^\infty$-cylindrical topology (see Definition \ref{def:acs_cob}).
\end{enumerate}
\end{lemma}

\begin{proof} Let $\sigma = (g,m,A)$ be the curve type with total genus $g$, number of points $m$ and relative homology class $A \in H_2(X,\partial X)$. 

\vspace{3pt}

\emph{\ref{itm:axiom_minimizer}} - Fix a pair $(J,P) \in \mathcal{J}(X) \times \hat{X}^m$ and let $u_k \in \mathcal{M}_\sigma^T(X;J,P)$ be a sequence of curves with ends $\Gamma^\pm_k$ that satisfy 
\[E_\Omega(u_k) \to \mathfrak{g}_{\sigma,T}(X,\Omega,Z;J,P) \qquad\text{as}\qquad k \to \infty\] 
The total genus of $u_k$ is bounded by $g$. The energies and lengths are bounded (for large $k$) by
\[E_\Omega(u_k) \le 2 \cdot \mathfrak{g}_{\sigma,T}(X,\Omega,Z;J,P) \qquad \text{and}\qquad \mathcal{L}(\Gamma^+_k) \le T \]
Finally, the number of components of $u_k$ is bounded by a constant $N$ by Lemma \ref{lem:component_bound}. 

\vspace{3pt}

Thus by SFT compactness (Theorem \ref{thm:SFT_compactness}), there is a $J$-holomorphic building ${\bf u}$ and a subsequence $u_i$ of $u_k$ with
\[u_i \to {\bf u} \qquad\text{in the BEHWZ sense}\]
Let ${\bf \Sigma}$ be the domain of ${\bf u}$ and let $v = [\Sigma,j,v,S]$ be the union of the curves in the cobordism level. Then since $u_i \to v$ in the $C^\infty_{\on{loc}}$ topology on maps to $\hat{X}$, we have
\[
P \subset v(\Sigma) \qquad\text{and}\qquad v_*[\Sigma] = A \in H_2(X,\partial X)
\]
Since $\Sigma$ is a sub-surface of ${\bf \Sigma}$, we have
\[
g(\Sigma) \le g({\bf \Sigma}) \le g
\]
Furthermore, let $\Gamma$ and $\Xi$ be the positive ends of $v$ and ${\bf u}$ respectively. Then $\Xi$ is connected to $\Gamma$ by a series of $J$-holomorphic curves in the symplectization of $\partial_+X$. Since $X$ is conformal, it follows from Lemma \ref{lem:energy_bounds_symplectization}\ref{itm:energy_bound_symplectization:length} that
\[\mathcal{L}(\Gamma) \le \mathcal{L}(\Xi)\le T\]
This implies that $v$ is an element of $\mathcal{M}_\sigma^T(X,\Omega,Z;J,P)$ and therefore
\[
\mathfrak{g}_{\sigma,T}(X,\Omega,Z;J,P) \le E_\Omega(v)\]
Finally, we note that ${\bf u}$ consists of the level $v$, plus a set of symplectization levels with positive $\Omega$-energy (by Lemma \ref{lem:energy_bounds_symplectization}\ref{itm:energy_bound_symplectization:area}). It follows that
\[
E_\Omega(v) \le E_\Omega({\bf u}) = \mathfrak{g}_{\sigma,T}(X,\Omega,Z;J,P)
\]
This proves that $E_\Omega(v) = \g_{\sigma,T}(X;J,P)$ and concludes the proof.

\vspace{3pt}

\emph{\ref{itm:axiom_semicontinuity}} 
- Choose a convergent sequence of almost complex structures and points as follows.
\[(J_k,P_k) \in \mathcal{J}(X,\Omega,Z) \times \widehat{X}^m \qquad\text{converging to}\qquad (J,P)\]
We may  assume that this sequence recovers the $\liminf$ in the sense that
\[\lim_{k \to \infty} \mathfrak{g}_{\sigma,T}(X,\Omega,Z;J_k,P_k) = \liminf_{(J_k,P_k) \to (J,P)} \mathfrak{g}_{\sigma,T}(X,\Omega,Z;-).\]
We may also assume that the spectral gap of $(X;J_k,P_k)$ if finite for each $k$. By (a), we can choose  a sequence of curves $u_k \in \mathcal{M}_\sigma^T(X;J_k,P_k)$ that satisfies
\[E_\Omega(u_k) = \mathfrak{g}_{\sigma,T}(X,\Omega,Z;J_k,P_k) \qquad\text{for each }k \in \N\]
The total genus of $u_k$ is bounded by $g$, and we have energy and length bounds
\[E_\Omega(u_k) \le 2 \cdot \mathfrak{g}_{\sigma,T}(X,\Omega,Z;J,P) \qquad \text{and}\qquad \mathcal{L}(\Gamma^+_k) \le T \]
Moreover, since $J_k \to J$ $C^\infty$-cylindrically, there is an $a > 0$ such that $J_k$ is cylindrical on the thickening $X^a_{\smallneg a}$ for sufficiently large $k$. Thus Lemma \ref{lem:component_bound} implies that the number of components of $u_k$ is uniformly bounded in $k$.

\vspace{3pt}

Thus, by applying SFT compactness (Theorem \ref{thm:SFT_compactness}) and passing to a subsequence of $u_k$, we acquire a $J$-holomorphic building ${\bf u}$ for $J$ that satisfies
\[
E_\Omega({\bf u}) = \lim_{k \to \infty} E_\Omega(u_k) = \liminf_{(J_k,P_k) \to (J,P)} \mathfrak{g}_{\sigma,T}(X,\Omega,Z;-)
\]
By an identical argument to (a), the cobordism level $v$ of ${\bf u}$ is an element of $\mathcal{M}_\sigma^T(X,\Omega,Z)$ of $\Omega$-energy less than that of ${\bf u}$. Therefore
\[\g_{\sigma,T}(X,\Omega,Z;J,P) \le E_\Omega(v) \le E_\Omega({\bf u}) = \liminf_{(J_k,P_k) \to (J,P)} \mathfrak{g}_{\sigma,T}(X,\Omega,Z;-) \qedhere\]
\end{proof}

The spectral gaps also satisfy an invariance property with respect to flow by the conformal vector-field of a given conformal cobordism. 

\begin{lemma}[Flow Invariance] \label{lem:gap_flow_invariance} Let $(X,\Omega,Z)$ be a conformal symplectic cobordism with Morse-Bott ends, and let
\[\Phi:\R \times \hat X \to \hat X \qquad\text{given by}\qquad (t,x) \mapsto \Phi^t(x)\]
be the flow by the conformal vector-field $\hat Z$. Then
\[
\mathfrak{g}_{\sigma,T}(X;\Phi^t_*J,\Phi^t(P)) = \mathfrak{g}_{\sigma,T}(X;J,P) \qquad\text{for any }t \in \R
\]
\end{lemma}

\begin{proof} The flow $\Phi^t$ defines a bijection of moduli spaces
\[
\Phi^t_*:\mathcal{M}^T_\sigma(X;J,P) \to \mathcal{M}^T_\sigma(X;\Phi^t_*J,\Phi^t(P)) \quad\text{given by}\qquad \Phi^t_*[\Sigma,j,u,S] := [\Sigma,j,\Phi^t \circ u,S]
\]
This translation action preserves the ends $\Gamma_\pm$ of $u$ since the projection to $Y_+$ and $Y_-$ of $u$ and $\Phi^t \circ u$ are the same near the cylindrical ends of $\hat{X}$. Moreover, the surface class $[u] \in S(X;\Gamma_+,\Gamma_-)$ is also preserved. The period of the positive end and the $\Omega$-energy only depend on $\Gamma_\pm$ and the surface class $[u]$ (see Section \ref{subsec:area}), so the translation action preserves these quantities. The lemma thus follows from Definition \ref{def:def_gap_JP}. \end{proof}

\subsection{Gaps Of Cobordisms With Morse-Bott Ends} \label{subsec:spectral_gaps_MB} We are now ready to introduce the main invariants of this paper for symplectic cobordisms with Morse-Bott ends.

\begin{definition} \label{def:spectral_gap_MB} The \emph{elementary symplectic field theory (ESFT) spectral gap} of a symplectic cobordism $(X,\Omega,Z)$ with Morse-Bott ends, with curve type $\sigma$ and a period $T$ is the supremum
\begin{equation} \label{eqn:def_gap_MB}
\g_{\sigma,T}(X,\Omega,Z) := \sup \big\{\mathfrak{g}_{\sigma,T}(X,\Omega,Z ;J,P) \; : \; P \in X^m \; \text{and}\; J \in \mathcal{J}(X)\big\} \in [0,\infty]
\end{equation}
\end{definition}

We will also fix a special terminology for the spectral gap of a Liouville domain. For this, we require the following lemma. 

\begin{lemma}[Period Independence] \label{lem:g_independent_of_T_for_Liouville} Let $(W,\lambda)$ be a Liouville domain with Morse-Bott contact boundary and let $\sigma$ be a curve type. Then there is an $S > 0$ so that
\[\g_{\sigma,T}(W)\quad \text{is independent of $T$ for all $T \ge S$}\]
\end{lemma}

\begin{proof} Suppose without loss of generality that $\g_{\sigma,T}(W)$ is finite for all sufficiently large $T$. Note that for any $J \in \mathcal{J}(W)$ and any finite energy, proper $J$-holomorphic curve
\[u:\Sigma \to \hat{W} \qquad\text{asymptotic to $\Gamma$ at $+\infty$}\]
We can apply Stokes theorem to find that
\[E_\Omega(u) = \int_\Sigma u^*d\lambda = \int_{\partial\Sigma} u^*\lambda = \mathcal{L}(\Gamma)\]
Thus, for sufficiently large period $T$, we have
\[
\g_{\sigma,T}(W;J,P) \le T  \qquad \text{and}\qquad \mathfrak{g}_{\sigma,T}(W;J,P) \in \LSpec(\partial W,\lambda)
\]
By supremizing over $J$ and $P$, we acquire the same properties for $\g_{\sigma,T}(W)$.  Now note that the length spectrum
\[\LSpec(\partial W,\lambda) \subset [0,\infty)\]
is discrete (and thus closed) since its Reeb flow on $(\partial W,\lambda)$ is Morse-Bott. Since $\g_{\sigma,T}(W)$ is non-increasing in $T$ (see Proposition \ref{prop:axioms_of_spectral_gaps}(e) below), it must be constant in $T$ for $T \ge S$ and some $S$. \end{proof}

\begin{definition} The \emph{elementary symplectic field theory (ESFT) capacity} of a Liouville domain $(W,\lambda)$ with Morse-Bott ends, with curve type $\sigma$, is given by
\[\mathfrak{c}_\sigma(W,\lambda) := \mathfrak{g}_{\sigma,T}(W) \qquad\text{for $T$ sufficiently large}\]
\end{definition}

The ESFT spectral gaps and capacities satisfy a number of basic properties that are essentially immediate from the definition. We describe these in the following proposition.

\begin{prop}[Basic Properties] \label{prop:basic_properties_of_spectral_gaps} The ESFT spectral gaps of a conformal symplectic cobordism
\[(X,\Omega,Z)\]
with Morse-Bott ends have the following basic properties.
\begin{enumerate}[label=(\alph*)]
        \item \label{itm:MB_positive} (Positive) $\g_{\sigma,T}(X) \ge 0$ with equality if and only if $\sigma$ has no points and zero homology class.
        \[\sigma = (g,0,0)\]
	\item \label{itm:MB_sub_linearity} (Sub-Linearity) If $\sigma,\tau$ are two curve types and $S,T > 0$, then
	\[\mathfrak{g}_{\sigma + \tau,S+T}(X) \le \mathfrak{g}_{\sigma,S}(X) + \mathfrak{g}_{\tau,T}(X)\]
	\item \label{itm:MB_disj_union} (Union) If $X = U \sqcup V$ is a disjoint union of conformal cobordisms $U$ and $V$ then 
	\[\mathfrak{g}_{\rho,T}(X) =\on{min} \big\{\mathfrak{g}_{\sigma,R}(U) + \mathfrak{g}_{\tau,S}(V) \; : \; \rho = \iota_U(\sigma) + \iota_V(\tau) \text{ and }T = R + S\big\}\]
	\item \label{itm:MB_conformality} (Conformality) If $a$ is a positive real number, then
	\[\mathfrak{g}_{\sigma,T}(X,a \cdot \Omega,Z) = a \cdot \mathfrak{g}_{\sigma,T/a}(X,\Omega,Z) \qquad \mathfrak{g}_{\sigma,T}(X,\Omega,a \cdot Z) = \mathfrak{g}_{\sigma,T/a}(X,\Omega,Z)\]
	\item \label{itm:MB_monotonicity} (Monotonicity) If $\sigma$ and $\tau$ are curve types, and $S,T > 0$ are periods then
	\[\mathfrak{g}_{\sigma,S}(X) \geq \mathfrak{g}_{\tau,T}(X) \qquad\text{if}\qquad \sigma \preceq \tau \quad\text{and}\quad S \le T\]
	\item \label{itm:MB_thickening} (Thickening) Fix constants $\delta > 0 > \epsilon$ and fix a period $T$. Let $T(\delta) = \exp(c_X \cdot \delta) \cdot T$. Then
	\[|\mathfrak{g}_{\sigma,T(\delta)}(X^\delta_\epsilon) -  \mathfrak{g}_{\sigma,T}(X)| \le T \cdot (\f(\delta) - \f(\epsilon))\]
	\item \label{itm:MB_spectrality} (Area) If the cohomology class of $\Omega$ is rational, i.e. $[\Omega] \in H^2(X,\Q)$, then
    \[\mathfrak{g}_{\sigma,T}(X) \in \ASpec_T(X) \qquad\text{or}\qquad \mathfrak{g}_{\sigma,T}(X) = \infty\]
       \item \label{itm:MB_length} (Length) If  $S < T$ and $[S,T] \cap \LSpec(\partial_+ X) = S$, then \[\mathfrak{g}_{\sigma,T}(X) = \mathfrak{g}_{\sigma,S}(X)\]
       \item \label{itm:MB_cover} (Cover) Let $\pi:U \to X$ be a covering map of conformal cobordisms and let $\sigma = (g,m,A)$ be a curve type in $U$. Then
       \[
       \g_{\sigma,T}(U) \ge \g_{\pi(\sigma),T}(X) \qquad\text{where}\qquad \pi(\sigma) = (g,m,\pi_*A)
       \]
\end{enumerate}
\end{prop}

\begin{proof} We prove each of these properties individually. Throughout the proof, fix curve types
\[\sigma = (g,m,A) \quad\text{and}\quad \tau = (h,n,B)\]

\emph{\ref{itm:MB_positive} - Positive.} We may assume that $\g_{\sigma,T}(X) < \infty$. Choose an almost complex structure and point set $(J,P) \in \mathcal{J}(X) \times \hat{X}^m$ such that $J$ is cylindrical outside of $\hat{X} \setminus X$. Applying Lemma~\ref{lem:energy_bounds_cobordism}\ref{itm:energy_bound_cobordism:area} with $a=0$ (and recalling that $\f(0)=0$), we see that for any non-constant, finite energy, proper $J$-holomorphic curve $u:\Sigma \to \hat{X}$, we have
\[
E_\Omega(u) \ge 0
\]
with equality if and only if $u$ is constant. It follows that
\[
\g_{\sigma,T}(X) \ge \g_{\sigma,T}(X;J,P) \ge 0
\]
with equality if and only if $\mathcal{M}_\sigma(X;J,P)$ contains the empty or constant curve (i.e. if and only if $m = 0$ and $A = 0$).

\vspace{3pt}

\emph{\ref{itm:MB_sub_linearity} - Sub-Linearity.} Fix a choice of compatible almost complex structure $J$ on $X$ and point sets $P \in \hat{X}^m$ and $Q \in \hat{X}^n$. Then there is a map
\[\mathcal{M}_\sigma^S(X;J,P) \times \mathcal{M}_\tau^T(X;J,Q) \to \mathcal{M}^{S+T}_{\sigma + \tau}(X;J,P \times Q)\]
given by taking the union of curves. The area of a union of curves is the sum of the areas, and the length of the positive ends are also additive. Therefore, by Definition \ref{def:def_gap_JP} we find that
\[
\g_{\sigma,S}(X;J,P) + \g_{\tau,T}(X;J,Q) \ge \g_{\sigma + \tau,S+T}(X;J,P \times Q)
\]
Taking the supremum over all $J$ and all pairs $P$ and $Q$ then yields the desired bound for \ref{itm:MB_sub_linearity}.

\vspace{3pt}

\emph{\ref{itm:MB_disj_union} - Disjoint Union.} Fix a choice of compatible almost complex structure $J$ on $X$ and points $P \in \hat{X}^m$. Then there is a map of the form
\[\mathcal{M}_\sigma(U;J|_U,P \cap \hat{U}) \times \mathcal{M}_\tau(V;J|_V,P \cap \hat{V}) \to \mathcal{M}_{\rho}(X;J,P)\]
that is given by taking the union of curves. Any curve in $X$ must decompose into a curve in $U$ and a curve in $V$. Thus these maps piece together to an area-preserving bijection
\[
\mathcal{M}_\rho(X;J,P) = \bigsqcup_{\rho = \sigma \sqcup \tau} \mathcal{M}_\sigma(U;J|_U,P \cap \hat{U}) \times \mathcal{M}_\tau(V;J|_V,P \cap \hat{V})
\]
Therefore by Definition \ref{def:def_gap_JP}, we can write the gap $\g_{\rho,T}(X;J,P)$ is as follows.
\[\on{min} \big\{\mathfrak{g}_{\sigma,R}(U;J|_U,P \cap \hat{U}) + \mathfrak{g}_{\tau,S}(V;J|_V,P \cap \hat{V}) \; : \; \rho = \iota_U(\sigma) + \iota_V(\tau) \text{ and }T = R + S\big\}
\]
Taking the supremum over all $J$ and all pairs $P$ then yields the desired formula in \ref{itm:MB_disj_union}.

\vspace{3pt}

\emph{\ref{itm:MB_conformality} - Conformality.} The following spaces of compatible almost complex structures are the same.
\[\mathcal{J}(X,a \cdot \Omega,Z) = \mathcal{J}(X,\Omega, Z) \qquad \text{and}\qquad \mathcal{J}(X,\Omega,a \cdot Z) = \mathcal{J}(X,\Omega, Z)\]
If $\Omega$ is scaled by $a > 0$, the area of any holomorphic maps $u$ scales by $a$ and the length of any closed orbit is scaled by $a$ . Therefore
\[\mathfrak{g}_{\sigma,T}(X,a \cdot \Omega, Z;J,P) = a \cdot \mathfrak{g}_{\sigma,T/a}(X,\Omega,Z;J,P) \]
If $Z$ is scaled by $a > 0$, the area of any holomorphic map $u$ is preserved and the length of any closed orbit is scaled by $a$. Therefore
\[
\mathfrak{g}_{\sigma,T}(X,\Omega,a \cdot Z;J,P) = \mathfrak{g}_{\sigma,T/a}(X,\Omega,Z;J,P)
\]
The desired formula in \ref{itm:MB_conformality} follows by maximizing these two formulas over $J$ and $P$. 
\vspace{3pt}

\emph{\ref{itm:MB_monotonicity} - Monotonicity.} Fix a choice of compatible almost complex structure $J$ on $X$ and point sets $P \in \hat{X}^m$. Let $Q \in \hat{X}^n$ denote the first $n$ points in $P$, which we may take since $n \le m$. There is an inclusion of moduli spaces of the form
\[\mathcal{M}_\sigma^S(X,\Omega;J,P) \subseteq \mathcal{M}_\tau^T(X,\Omega;J,Q) \qquad \text{if}\qquad \sigma \preceq \tau \qquad\text{and}\qquad S \le T\]
By applying Definition \ref{def:def_gap_JP} and minimizing the area over curves for a fixed $(J,P)$, we have
\[\mathfrak{g}_{\sigma,S}(X;J,P) \ge \mathfrak{g}_{\tau,T}(X;J,P)\]
Taking the supremum over all pairs $(J,P)$ then yields the desired formula in \ref{itm:MB_monotonicity}.

\vspace{3pt}

\emph{\ref{itm:MB_thickening} - Thickening.} As mentioned in (\ref{eqn:deformation_completion}), we have a canonical identification  $\hat{X}_\epsilon^\delta = \hat{X}$.  Moreover the following spaces of almost complex structures are the same,
\[\mathcal{J}(X_\epsilon^\delta,\Omega_\epsilon^\delta,Z_\epsilon^\delta) = \mathcal{J}(X,\Omega, Z).\]
Since the spectral gaps are non-negative due to \ref{itm:MB_positive}, it is sufficient to consider pairs $(J,P)$ of an almost complex structure $J \in \mathcal{J}(X,\Omega,Z)$ and a set of points $P \in \hat{X}^m$ for which 
\begin{equation}\label{eq:spec_gaps_assumed_positive}
    \g_{\sigma,T}(X;J,P)\geq 0 \qquad \text{and}\qquad 
\g_{\sigma,T}(X^\delta_\epsilon;J,P)\geq 0
\end{equation}
Fix such $(J,P)$ and consider a fixed finite energy $J$-holomorphic map
\[u \in \mathcal{M}^T_\sigma(X;J,P) \qquad\text{asymptotic to $\Gamma_\pm$ at $\pm \infty$}\]
Note that $E_\Omega(u)\geq 0$ due to (\ref{eq:spec_gaps_assumed_positive}), and by Lemma~\ref{lem:energy_bounds_cobordism}\ref{itm:energy_bound_cobordism:length}, this implies that $\cL(\Gamma_+)\geq \cL(\Gamma_-)$.
By Lemma \ref{lem:area_change_under_thickening}, the energy of $u$ in ${X}_\epsilon^\delta$ and in $X$ are related as follows.
\[
E_\Omega(u;X_\epsilon^\delta) - E_\Omega(u;X) = \mathcal{L}(\Gamma_+) \cdot \mathfrak{f}(\delta)- \mathcal{L}(\Gamma_-) \cdot \mathfrak{f}(\epsilon) \le \big(\mathfrak{f}(\delta) - \f(\epsilon)\big) \cdot T,
\]
where the last inequality relies on the fact that $\epsilon<0$ and that $\cL(\Gamma_-)\le\cL(\Gamma_+)\le T$.
Likewise, the length $\mathcal{L}(\Gamma_+,X_\epsilon^\delta)$ of $\Gamma_+$ with respect to the stabilizing form on $\partial_+X_\epsilon^\delta$ is given by
\[
\mathcal{L}(\Gamma_+,X_\epsilon^\delta) = \f'( \delta) \cdot \mathcal{L}(\Gamma_+,X)=\exp(c_X \cdot \delta) \cdot \mathcal{L}(\Gamma_+,X) \le \exp(c_X \cdot \delta) \cdot T = T(\delta)
\]
By applying Definition \ref{def:def_gap_JP} and minimizing the area over curves for a fixed $(J,P)$, we have
\[|\mathfrak{g}_{\sigma,T(\delta)}(X_\epsilon^\delta;J,P) -  \mathfrak{g}_{\sigma,T}(X;J,P)| \le \big(\mathfrak{f}(\delta) - \f(\epsilon)\big) \cdot T\]
Maximizing over pairs $(J,P)$, we acquire the desired inequality in \ref{itm:MB_thickening}.
\vspace{3pt}

\emph{\ref{itm:MB_spectrality} - Area Spectrality.} Assume $\g_{\sigma,T}(X)<\infty$.
For pair $(J,P)$, Lemma~\ref{lem:minimizer_and_JP_semicont} states that there exists $u\in \cM_\sigma^T(X;J,P)$ with $E_\Omega(u) = \g_{\sigma,T}(X,\Omega,Z;J,P)$. By Definition~\ref{def:area_cobordism_surface_class}, this implies that \[\g_{\sigma,T}(X,\Omega,Z;J,P)\in \ASpec_T(X)\]
Lemma~\ref{lem:ASpec_MB_cobordism} states the area spectrum is closed, and therefore $\g_{\sigma,T}(X)\in \ASpec_T(X)$.

\vspace{3pt}

\emph{\ref{itm:MB_length} - Length Spectum.} We simply note that if $[S,T] \cap \LSpec(\partial_+X)$ is the singleton $S$, then
\[\mathcal{M}_\sigma^T(X;J,P) = \mathcal{M}_\sigma^S(X;J,P) \quad\text{and}\quad  \mathfrak{g}_{\sigma,T}(X;J,P) = \mathfrak{g}_{\sigma,S}(X;J,P) \qquad\text{for all} \qquad J,P\]
The result follows by maximizing over all pairs $(J,P)$. 

\vspace{3pt}

\emph{\ref{itm:MB_cover} - Cover.} Choose a pair $(J,P) \in \mathcal{J}(X) \times \hat{X}^m$. Through the covering map, $J$ pulls back to an almost complex structure $I = \pi^*J \in \mathcal{J}(X)$ and we may choose a set of $m$ points $Q$ in $U$ with $\pi(Q) = P$. There is a natural map
\[
\mathcal{M}_\sigma^T(U;I,Q) \to \mathcal{M}_{\pi(\sigma)}^T(X;J,P) \quad\text{given by}\quad u \mapsto \pi \circ u
\]
This map preserves $\Omega$-energy. Therefore by Definition \ref{def:def_gap_JP} and Definition \ref{def:spectral_gap_MB}, we have
\[
\g_{\pi(\sigma),T}(X;J,P) \le \g_{\sigma,T}(U;I,Q) \le \g_{\sigma,T}(U)
\]
The desired property follows by supremizing over all choices of $J$ and $P$. \end{proof}

The ESFT spectral gaps satisfy two important monotonicity properties with respect to symplectic embeddings. We start with the case of embeddings of Liouville domains.

\begin{prop}[Liouville embeddings]\label{prop:monotonicity_embedding_Liouville} Let  $X$ be a conformal cobordism with $[\Omega] \in H^2(X;\Q)$ and let $(U,\lambda)$ be a Liouville domain with a symplectic embedding $U \to X$. Then
     \[\mathfrak{c}_{\sigma|_U}(U) \le \mathfrak{g}_{\sigma,T}(X) \qquad\text{for any curve class $\sigma = (g,m,A)$ on $X$ and period $T$}\]
\end{prop}

\begin{proof} By replacing $U$ with the thinning $U_\epsilon$ for a small $\epsilon < 0$, we may assume that $U \subset \on{int}(X)$. The general case then follows from the fact that
\[
(U_\epsilon,\lambda|_{U_\epsilon}) \simeq (U,e^{-\epsilon}\lambda)
\]
and the conformality property of the spectral gap (Proposition \ref{prop:basic_properties_of_spectral_gaps}\ref{itm:MB_conformality}). We may then treat $X$ as a broken cobordism in the sense of Definition \ref{def:broken_cobordism}, by writing $X$ as a composition
\[
X \simeq (X\setminus U) \circ U \quad\text{for the symplectic cobordism}\quad X\setminus U:\partial_+X \to \partial_-X \sqcup \partial U
\]
Note that $X \setminus U$ is \emph{not} a conformal symplectic cobordism, but the boundary componens are conformal.

\vspace{3pt}

Now choose an almost complex structure $J \in \mathcal{J}(U)$ and a set of $m$ points $P$ in $\hat{U}$. After flowing $J$ and $P$ by the Liouville vector-field $V$ on $W$ for some large negative time, we may assume that
\[
J \text{ is cylindrical on }[0,\infty) \times \partial U \subset \hat{U} \qquad\text{and}\qquad P \subset U \subset \hat{U}
\]
This does not change the spectral gap $\mathfrak{g}_{\sigma|_U,T}(U;J,P)$ by Lemma \ref{lem:gap_flow_invariance}. Also choose a compatible almost complex structure $I \in \mathcal{J}(X \setminus U)$ such that
\[
I \text{ is cylindrical outside of }X \setminus U \subset \widehat{X \setminus U} \qquad\text{and}\qquad I_{\partial U} = J_{\partial U}
\]
The pair $I$ and $J$ now define a broken complex structure (see Definition \ref{def:broken_J}). Let $J_s \in \mathcal{J}(X)$ be a neck-stretching family for $I$ and $J$ (see Definition \ref{def:neck_stretching_family}). By Definition \ref{def:spectral_gap_MB} and Lemma \ref{lem:minimizer_and_JP_semicont}\ref{itm:axiom_minimizer}, for each $s > 0$ large there is a curve
\[u_s \in \mathcal{M}_\sigma^T(X;J_s,P) \qquad \text{with}\qquad E_\Omega(u_s) = \mathfrak{g}_{\sigma,T}(X;J_s,P) \le \mathfrak{g}_{\sigma,T}(X)\]
The total genus of each curve is bounded by $g$, the total energy is bounded by $\mathfrak{g}_{\sigma,T}(X)$ and the total period of the positive ends are bounded by $T$. Moreover, the glued almost complex structures $J_s$ are cylindrical outside of $X$, so Lemma \ref{lem:component_bound} implies that the number of components is uniformly bounded in $s$.

\vspace{3pt}

Thus, by Theorem \ref{thm:neck_stretching_compactness}, we can pass to a subsequence $u_k$ that BEHWZ converges to a building
\[{\bf u} \quad\text{in the broken cobordism }X = (X \setminus U) \circ U \quad\text{with domain }({\bf \Sigma},{\bf j})\]

Let $v:(\Sigma,j) \to \hat{U}$ be the $U$ level of the building ${\bf u}$, and let $\Gamma_+$ denote the positive ends. Then $u_k$ converges to $v$ on $U$ in the $C^\infty$ topology on compact sets. This implies that
\[P \subset v(\Sigma) \qquad g(\Sigma) \le g({\bf \Sigma}) \le g \qquad [v] = A|_U \in H_2(U,\partial U)\]
Moreover, by Theorem \ref{thm:neck_stretching_compactness} and Stokes theorem, the $\Omega$-energy of $v$ and length of $\Gamma_+$ satisfy
\[E_\Omega(v) \le E_\Omega(\text{$\bf {u}$}) \le \mathfrak{g}_{\sigma,T}(X) \qquad\text{and}\qquad \mathcal{L}(\Gamma_+) = E_\Omega(v) \le \mathfrak{g}_{\sigma,T}(X)\]
We can thus choose a set of $m$ points $Q \subset \Sigma$ mapping to $P$ to acquire an element
\[v = [\Sigma,j,v,Q] \in \mathcal{M}_{\sigma|_U}^S(U;J,P) \qquad\text{where} \qquad S \ge \mathfrak{g}_{\sigma,T}(X)\]
By Theorem \ref{thm:neck_stretching_compactness} and Definition \ref{def:def_gap_JP}, the spectral gap of $(U;J,P)$ satisfies
\[\mathfrak{g}_{\sigma,S}(U;J,P) \le E_\Omega(v) \le E_\Omega({\bf u}) \le \mathfrak{g}_{\sigma,T}(X)\]
Finally, by taking the supremum over all choices of $(J,P)$, we find that
\[
\mathfrak{g}_{\sigma|_U,S}(U) \le \mathfrak{g}_{\sigma,T}(X)
\]
By choosing $R \ge S$ large enough and applying the monotonicity property of the spectral gap (Proposition \ref{prop:basic_properties_of_spectral_gaps}\ref{itm:MB_monotonicity}), we acquire the claimed inequality.
\[\c_{\sigma|_U}(U) = \g_{\sigma|_U, R}(U)\leq \g_{\sigma|_U, S}(U) \le \g_{\sigma,T}(X).   \qedhere \]
\end{proof}

There is also an important monotonicity property under trace of conformal cobordisms. We will start by stating a slightly complicated version of this property. 

\begin{lemma} \label{lem:general_trace} Fix three conformal symplectic cobordisms with Morse-Bott ends of the form
\[U:Y_0 \to M \sqcup Y_1 \qquad W:M \to N \qquad V:N \sqcup Y_2 \to Y_3\]
Let $X = (X,\Omega,Z)$ be the trace $(U \sqcup W \sqcup V)_{M \sqcup N}$ of $U \sqcup W \sqcup V$ along $M \sqcup N$ and assume that
\[
[\Omega] \in H^2(X;\Q)
\]
Finally, fix three curve types of the form
\[\sigma = (g,m + n,A) \in \mathcal{S}(X) \qquad \sigma_U = (g,m,A|_U) \in \mathcal{S}(U) \quad\text{and}\quad \sigma_V = (g,n,A|_V) \in \mathcal{S}(V) \]
Then the spectral gaps of $X,U$ and $V$ satisfy the following trace inequality.
\[\mathfrak{g}_{\sigma,T}(X) \ge \mathfrak{g}_{\sigma_U,T}(U) + \mathfrak{g}_{\sigma_V,T}(V)\]
\end{lemma}

To help orient the reader, we have included a depiction of the setup of this lemma. See Figure \ref{fig:general_trace}.

\begin{figure}
    \centering
    \includegraphics[width=.8\textwidth]{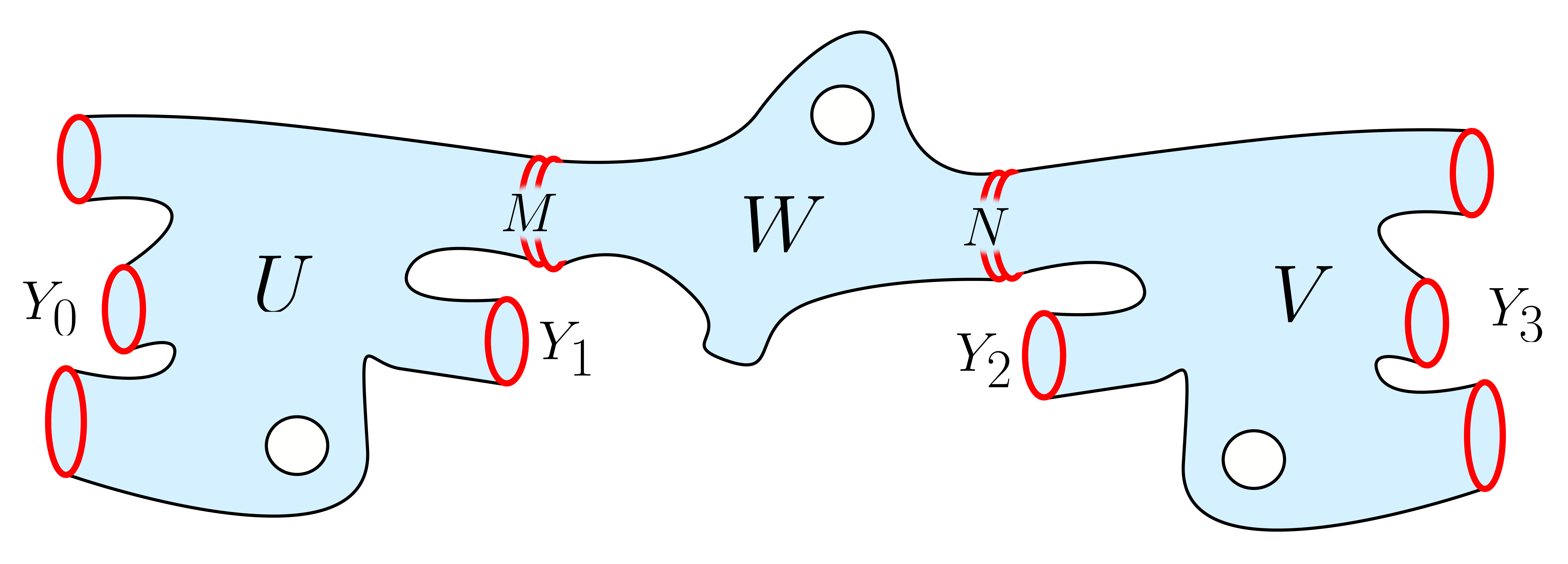}
    \caption{A picture of the setup in Lemma \ref{lem:general_trace}}
    \label{fig:general_trace}
\end{figure}

\begin{proof} Choose two pairs of an almost complex structure and set of points, as follows.
\[
(J_U,P_U) \in \mathcal{J}(U) \times \hat{U}^{m} \qquad\text{and}\qquad (J_V,P_V) \in \mathcal{J}(V) \times \hat{V}^{n}
\]
By flowing the pairs $(J_U,P_U)$ and $(J_V,P_V)$ by the conformal vector-fields on $U$ and $V$, respectively, we may assume that
\[J_U \text{ is cylindrical on }(-\infty,0] \times M \quad\text{and}\quad P_U \cap (-\infty,0] \times M = \emptyset\]
\[J_V \text{ is cylindrical on }[0,+\infty) \times N \quad\text{and}\quad P_V \cap [0,+\infty) \times N = \emptyset\]
This does not change the spectral gaps $\mathfrak{g}_{\sigma_U,T}(U;J_U,P_U)$ and $\mathfrak{g}_{\sigma_V,T}(V;J_V,P_V)$ by Lemma \ref{lem:gap_flow_invariance}. Since the spectral gaps of $U$ and $V$ are non-negative (see Proposition \ref{prop:basic_properties_of_spectral_gaps}), we may also assume that
\[
\g_{\sigma_U,T}(U;J_U,P_U) \ge 0\qquad\text{and}\qquad \g_{\sigma_V,T}(V;J_V,P_V) \ge 0
\]
In particular, every curve in the moduli space $\mathcal{M}_{\sigma_U}(U;J_U,P_U)$ and $\mathcal{M}_{\sigma_V}(V;J_V,P_V)$ has non-negative energy. Finally, choose a compatible almost complex structure $J_W\in\cJ(W)$ such that
\[
J_W \text{ is cylindrical on }\hat{W} \setminus W = (-\infty,0] \times N \sqcup [0,\infty) \times M
\]
and that has the same asymptotic complex structures at $M$ and $N$ as $J_U$ and $J_V$, respectively. The triple $J = (J_U,J_W,J_V)$ is broken almost complex structure on $X$ (see Definition \ref{def:broken_J}). 

\vspace{3pt}

Now let $J_s$ be the corresponding neck-stretching  family of almost complex structures on $X$ (see Definition \ref{def:neck_stretching_family}) and let $P = P_U \cup P_V$ be the subset of $\hat{X}$ determined by $P_U$ and $P_V$. By Definition \ref{def:spectral_gap_MB} and Lemma \ref{lem:minimizer_and_JP_semicont}\ref{itm:axiom_minimizer}, for each $s > 0$ large there is a curve
\[u_s \in \mathcal{M}_\sigma^T(X;J_s,P) \qquad \text{with}\qquad E_\Omega(u_s) = \mathfrak{g}_{\sigma,T}(X;J_s,P) \le \mathfrak{g}_{\sigma,T}(X)\]
The total genus of each curve is bounded by $g$ and the total energy is bounded by $\mathfrak{g}_{\sigma,T}(X)$. Moreover, the glued almost complex structures $J_s$ are cylindrical outside of the thickening $X^a_{\smallneg a}$ for some $a > 0$ independent of $s$. Therefore, Lemma \ref{lem:component_bound}, the number of components of $u_s$ is bounded independent of $s$.

\vspace{3pt}

Thus we can apply neck stretching compactness (Theorem \ref{thm:neck_stretching_compactness}) to pass to a subsequence $u_k$ that BEHWZ converges to a building
\[{\bf u} \quad\text{in the broken cobordism}\quad X = (U \sqcup W \sqcup V)_{M \sqcup N} \quad\text{with domain }({\bf \Sigma},{\bf j})\]
We consider the union $u_U$ of the $U$-levels and the union $u_V$ of the $V$-levels of the building ${\bf u}$.
\[u_U:(\Sigma_U,j_U) \to (\hat{U},J_U) \qquad\text{and}\qquad u_V:(\Sigma_V,j_V) \to (\hat{V},J_V)\]
We denote the positive ends by $\Gamma^+_U$ and $\Gamma^+_V$, respectively. Since $u_k$ converges to $u_U$ and $u_V$ on $\hat{U}$ and $\hat{V}$ in the $C^\infty$-topology on compact sets, we have
\[P_U \subset u_U(\Sigma_U) \qquad g(\Sigma_U) \le g({\bf \Sigma}) \le g \qquad [u_U] = A|_U \in H_2(U,\partial U)\]
\[P_V \subset u_V(\Sigma_V) \qquad g(\Sigma_V) \le g({\bf \Sigma}) \le g \qquad [u_V] = A|_V \in H_2(V,\partial V)\]
Moreover, every level of ${\bf u}$ has non-negative energy. Indeed, by assumption we have
\[
0 \le \mathfrak{g}_{\sigma_U,T}(U;J_U,P_U) \le E_\Omega(u_U) \qquad\text{and}\qquad 0 \le \mathfrak{g}_{\sigma_V,T}(V;J_V,P_V) \le E_\Omega(u_V)
\]
and the almost complex structure on $W$ is cylindrical on $\hat{W}\setminus W$ (and thus only has positive energy curves, see Lemma \ref{lem:energy_bounds_cobordism}\ref{itm:energy_bound_cobordism:area}). It then follows from Lemma \ref{lem:length_bound_building} that the lengths of the positive ends of $u_U$ and $u_V$ satisfy
\[
\mathcal{L}(\Gamma^+_U) \le T \qquad\text{and}\qquad \mathcal{L}(\Gamma^+_V) \le T
\]
Thus we may choose $m$ points $Q_U \subset \Sigma_U$ mapping to $P_U$ and $n$ points $Q_V \subset \Sigma_V$ mapping to $P_V$ to acquire elements
\[
[\Sigma_U,j_U,u_U,Q_U] \in \mathcal{M}^T_{\sigma_U}(U;J_U,P_U) \qquad\text{and}\qquad [\Sigma_V,j_V,u_V,Q_V] \in \mathcal{M}^T_{\sigma_V}(V;J_V,P_V)
\]

Since $u_U$ and $u_V$ are disjoint unions of levels in ${\bf u}$ and every level of ${\bf u}$ has non-negative energy, we see that
\[
\mathfrak{g}_{\sigma_U,T}(U;J_U,P_U) +\mathfrak{g}_{\sigma_V,T}(V;J_V,P_V) \le  E_\Omega(u_U) + E_\Omega(u_V) \le E_\Omega({\bf u}) \le \mathfrak{g}_{\sigma,T}(X)
\]
Finally, we acquire the desired inequality by taking the supremum over choices of $(J_U,P_U)$ and $(J_V,P_V)$ and applying Definition \ref{def:spectral_gap_MB}.
\[
\mathfrak{g}_{\sigma_U,T}(U) +\mathfrak{g}_{\sigma_V,T}(V) \le \mathfrak{g}_{\sigma,T}(X)
\qedhere \]\end{proof}

The intermediate cobordism $W$ in the proof of Lemma \ref{lem:general_trace} provides room for interpolation between the almost complex structures on $U$ and $V$. It can be removed by a thickening argument, as follows.

\begin{prop}[Trace] \label{prop:trace_property_MB} Fix two conformal symplectic cobordisms with Morse-Bott ends of the form
\[U:Y_0 \to M \sqcup Y_1 \qquad V:M \sqcup Y_2 \to Y_3\]
Let $X = (U \sqcup V)_{M}$ be the trace of $U \sqcup V$ along $M$. Fix three curve types of the form
\[\sigma = (g,m + n,A) \in \mathcal{S}(X) \qquad \sigma_U = (g,m,A|_U) \in \mathcal{S}(U) \quad\text{and}\quad \sigma_V = (g,n,A|_V) \in \mathcal{S}(V) \]
Then the spectral gaps of $X,U$ and $V$ satisfy the following trace inequality.
\[\mathfrak{g}_{\sigma,T}(X) \ge \mathfrak{g}_{\sigma_U,T}(U) + \mathfrak{g}_{\sigma_V,T}(V)\]
\end{prop}

\begin{proof} To simplicity the notation, assume that $X = U \circ V$, so that $\partial_-U = \partial_+V = M$. Moreover, assume that $T$ does not lie in the length spectrum of $Y_3$ (otherwise, approximate $T$ by $S>T$ that are not in the spectrum). Fix $\epsilon > 0$ and consider the thickening $X_{-\epsilon}^\epsilon$ of $X$, which decomposes as
\[
X_{-\epsilon}^\epsilon = \Big([0,\epsilon] \times \partial_+U\Big) \circ U \circ V \circ \Big([-\epsilon,0] \times \partial_-V\Big) 
\]
We can flow $U$ upward with respect to the conformal vector-field $Z$ on $X$ for time $\epsilon$ and, similarly, flow $V$ downward for time $\epsilon$. This gives a presentation of $X_\epsilon$ as the composition
\[
X_{-\epsilon}^\epsilon = U' \circ \Big([-\epsilon,\epsilon] \times M\Big) \circ V'
\]
where $U'$ and $V'$ are, respectively, the conformal symplectic cobordisms with symplectic forms multiplied by $\exp(c_X\cdot\epsilon)$  on $U$  and  by $\exp(-c_X\cdot \epsilon)$ on $V$. By  Lemma \ref{lem:general_trace}, we conclude that
\begin{equation} \label{eqn:trace_eqn_thickened}
\mathfrak{g}_{\sigma_U,T}(U') + \mathfrak{g}_{\sigma_V,T}(V') \le \mathfrak{g}_{\sigma,T}(X_{-\epsilon}^\epsilon)
\end{equation}
Now we take the limit as $\epsilon \to 0$. By the conformality property (Proposition \ref{prop:basic_properties_of_spectral_gaps}\ref{itm:MB_conformality}), the left-hand side of (\ref{eqn:trace_eqn_thickened}) converges to $\mathfrak{g}_{\sigma_U,T}(U) + \mathfrak{g}_{\sigma_V,T}(V)$. By the thickening property (Proposition \ref{prop:basic_properties_of_spectral_gaps}\ref{itm:MB_thickening}) and the length property (Proposition~\ref{prop:basic_properties_of_spectral_gaps}\ref{itm:axiom_length}), the right-hand side converges to $\mathfrak{g}_{\sigma,T}(X)$. This proves the desired inequality. \end{proof}

The spectral gaps also satisfy an important property with respect to Gromov-Witten invariants. In practice, the application of this property is the primary method for showing that the gaps are finite.

\begin{prop}[Gromov-Witten] \label{prop:MB_Gromov_Witten} Fix a conformal symplectic cobordism with Morse-Bott ends and conformal factor zero.
\[
X:(Y_+,\omega_+,\theta_+) \to (Y_-,\omega_-,\theta_-),\qquad c_X=0.
\]
Let $(M,\Omega)$ be a closed symplectic manifold with $X \subset M$ and fix a curve type $\tau = (g,k,A)$ of $M$ with 
\[
\GW_\tau(M,\Omega) \cap [\on{pt}]^{\otimes m} \neq 0 \qquad\text{for some }m \le k
\]
Finally, let $T$ be the pairing $T = [\theta_+] \cdot (A \cap [\partial_+X])$. Then the spectral gap of $X$ with respect to the curve type $\sigma = (g,m,A|_X)$ satisfies the following bound.
\[
\mathfrak{g}_{\sigma,T}(X) \le \Omega \cdot A
\]
\end{prop}

\begin{proof} Choose a pair of an almost complex structure and a set of points
\[
(J,P) \in \mathcal{J}(X) \times \hat{X}^m
\]
Choose a thickening $X^u_v$ so that $J$ is cylindrical outside of $X^u_v \subset \hat{X}$ and such that $P \subset X^u_v$. We can form a corresponding thickening of $M$ by removing $X$ from $M$ and gluing in $X^u_v$. This is the trace 
\[
\tilde{M} = (M \setminus \on{int}(X) \sqcup X^u_v)_{\partial X} \qquad\text{with symplectic form }\tilde{\Omega}:=\begin{cases}
    \Omega, &\text{ on }M\setminus \on{int}(X),\\
    \hat\Omega, &\text{ on }X_v^u,
\end{cases}
\]
where $\hat\Omega$ is the completion symplectic form on $\hat X\supset X_v^u$. Here $\tilde \Omega$ is a smooth symplectic form due to the assumption that $c_X =0$. Indeed, then a small neighborhood of $\partial X_v^u$ are symplectomorphic to a neighborhood of $\partial X$ (via the flow of the conformal vector field).   
One can construct a diffeomorphism $M \simeq \tilde{M}$ from a diffeomorphism that sends $X$ to its thickening. Under this diffeomorphism,  $\Omega$ and $\tilde{\Omega}$ are deformation equivalent and
\[
\tilde{\Omega} \cdot A - \Omega \cdot A = T \cdot u + T \cdot v \qquad\text{for any homology class }A \in H_2(M)
\]
Note that here we used again the fact that $c_X=0$ and hence $\theta_\pm$ are closed 1-forms.
Therefore by the deformation equivalence axiom in Proposition \ref{prop:GW_axioms}, we have
\[
\GW_\tau(\tilde{M},\tilde{\Omega}) \cap [\on{pt}]^{\otimes m} = \GW_\tau(M,\Omega) \cap [\on{pt}]^{\otimes m} \neq 0
\]
Choose a broken almost complex structure $I$ on $\tilde{M}$ restricting to $J$ on $X^u_v$, and let $I_s$ be the corresponding neck stretching family on $\tilde{M}$.

\vspace{3pt}

By the curve axiom in Proposition \ref{prop:GW_axioms}, for each neck stretching parameter $s$, there is a connected $I_s$-holomorphic curve in $\tilde{M}$ of genus $g$ passing through the points $P$, for each $s$.  Then by Theorem \ref{thm:neck_stretching_compactness}, we can pass to a subsequence $u_k$ that BEHWZ converges to a building ${\bf u}$ in $\tilde{M}$. Note that every level of ${\bf u}$ has positive energy, since $I$ is cylindrical on the ends of each piece of $\tilde{M}$. Consider the level in $\hat{X}^u_v$, given by
\[u_X:\Sigma_X \to \hat{X}^u_v = \hat{X} \qquad\text{with ends }\Gamma_\pm \subset Y_\pm\]
This level is $J$-holomorphic and satisfies
\[
g(\Sigma_X) \le g \qquad P \subset u(\Sigma_X) \qquad \mathcal{L}(\Gamma_+) = \mathcal{L}(\Gamma_-) = T \qquad E_\Omega(u;X^u_v) \le E_\Omega({\bf u}) = \tilde \Omega \cdot A
\]
Now we note that by the thickening property of area (cf. Lemma \ref{lem:area_change_under_thickening}) we have
\[
E_\Omega(u;X) = E_\Omega(u;X^u_v) - T \cdot (u+v) \le \tilde{\Omega} \cdot A - T \cdot(u+v) = \Omega \cdot A. 
\]
Thus $\g_{\sigma,T}(X;J,P) \le \Omega \cdot A$. By supremizing over all $J$ and $P$, we acquire the desired result. \end{proof}

\subsection{Gaps Of Conformal Symplectic Cobordisms} We can now extend the ESFT spectral gaps to all conformal symplectic cobordisms with rational symplectic form, via a limiting process.

\vspace{3pt}

Given a conformal cobordism $X$ with a curve type $\sigma$ and a period $T$, consider the supremum 
\begin{equation} \label{eqn:sup_g_extension} \sup_{W,S} \mathfrak{g}_{\sigma,S}(W)\end{equation}
over the set $\mathcal{C}(X,T)$ of pairs $(W,S)$, where $W$ is a Morse-Bott, graphical deformations $W = X^u_v$ such that $W \subset X$ (or equivalently $u < 0 < v$), and $S$ is a number with $S > T$.

\vspace{3pt}

We will require two basic properties of this supremum in the rational case. First, we note that it can be computed as an ordinary limit at a generic length. 

\begin{lemma} \label{lem:gap_extension_via_sequence} Fix a conformal cobordism $(X,\Omega,Z)$ with $[\Omega] \in H^2(X;\Q)$ and a sequence $(W_i,S_i)$ where
\begin{itemize}
\item $W_i$ are Morse-Bott, graphical deformations by pairs $(u_i,v_i)$ with $u_i \to 0$ and $v_i \to 0$ in $C^\infty$.
\item $S_i$ are real numbers with $S_i \to T$ where $T \not\in \LSpec(\partial_+X)$. 
\end{itemize}
Then the supremum (\ref{eqn:sup_g_extension}) over $\mathcal{C}(X,T)$ is given by
\[\sup_{W,S} \g_{\sigma,S}(W) = \lim_{i \to \infty} \mathfrak{g}_{\sigma,S_i}(W_i)\]
\end{lemma} 

\begin{proof} Consider a sequence $(U_i,R_i)$ consisting of
\[
\text{the thinnings }U_i := (W_i)^{\smallneg \epsilon_i}_{\epsilon_i}\quad \text{for }0<\epsilon_i\rightarrow0 \qquad\text{and}\qquad \text{a sequence }R_i > T \text{ with }R_i \to T 
\]
Since $u_i$ and $v_i$ limit to $0$ in $C^\infty$, we may choose $\epsilon_i$ so that $U_i \subset \on{int}(X)$. Moreover, since $\partial_+ U_i$ converges to $\partial_+X$ in $C^\infty$, Lemma~\ref{lem:spec_closed_under_limits} guarantees that there exists a small $\delta > 0$ such that, for sufficiently large $i$, we have
\[
[T-\delta,T+\delta] \cap \LSpec(\partial_+U_i) = \emptyset .
\]
Moreover, since $R_i \to T$ and $S_i \to T$, we also know that
\[R_i \in [T-\delta,T+\delta] \qquad\text{and}\qquad S'_i := \exp(-c_X \cdot \epsilon_i) \cdot S_i \in [T-\delta,T+\delta]\]
Now we note that, by the length spectrum property (Proposition \ref{prop:basic_properties_of_spectral_gaps}\ref{itm:MB_length}), we have
\[
\mathfrak{g}_{\sigma,R_i}(U_i) = \mathfrak{g}_{\sigma,S'_i}(U_i)
\]
Then by the thickening property (Proposition \ref{prop:basic_properties_of_spectral_gaps}\ref{itm:axiom_thickening}), we have the following inequality.
\begin{equation} \label{eqn:gap_extension_via_sequence_1} 
|\mathfrak{g}_{\sigma,S_i}(W_i) - \mathfrak{g}_{\sigma,R_i}(U_i)| = |\mathfrak{g}_{\sigma,S_i}(W_i) - \mathfrak{g}_{\sigma,S'_i}(U_i)| \le 2 S_i \cdot \f(\epsilon_i)
\end{equation}
On the other hand, the sequence $(U_i,R_i)$ is in $\mathcal{C}(X,T)$, and for any $(W,S)$ in $\mathcal{C}(X,T)$, we have
\[
W \subset \on{int}(W_i) \qquad\text{and}\qquad S > S_i \qquad\text{for sufficiently large }i
\]
Thus $U_i$ is isomorphic to a composition $U \circ W \circ V$ where $U$ and $V$ are intergraphs over $\partial_\pm W$ (see Definition \ref{def:intergraph_cobordism}). By the trace property (Proposition \ref{prop:trace_property_MB}), we have
\[\lim_{i \to \infty} \mathfrak{g}_{\sigma,R_i}(U_i) \ge \sup_{W,S} \g_{\sigma,S}(W)\]
The opposite inequality is obvious. Combining this with (\ref{eqn:gap_extension_via_sequence_1}), we find that
\[
\lim_{i \to \infty} \mathfrak{g}_{\sigma,S_i}(W_i) = \lim_{i \to \infty} \mathfrak{g}_{\sigma,R_i}(U_i) = \sup_{W,S} \g_{\sigma,S}(W) \qedhere
\]
\end{proof}

Second, we note that the supremum recovers the spectral gap in the Morse-Bott case.

\begin{lemma} \label{lem:gap_extension_agrees}  Let $(X,\Omega,Z)$ be a conformal cobordism with Morse-Bott ends and $[\Omega] \in H^2(X;\Q)$. Then
\[
\g_{\sigma,T}(X) = \sup_{W,S}\mathfrak{g}_{\sigma,S}(W)
\]
where the supremum is taken over $(W,S)\in\mathcal{C}(X,T)$.
\end{lemma}

\begin{proof} First, assume that $T \not\in \LSpec(\partial_+X)$ and fix $\epsilon > 0$. By applying the trace property (Proposition \ref{prop:trace_property_MB}) and the thickening property (Proposition \ref{prop:basic_properties_of_spectral_gaps}\ref{itm:MB_thickening}), we acquire the inequalities
\begin{equation} \label{eqn:gap_extension_agrees_1}
\mathfrak{g}_{\sigma,T(-\epsilon)}(X) \ge \mathfrak{g}_{\sigma,T(-\epsilon)}(X^{\smallneg\epsilon}_{\epsilon}) \ge \mathfrak{g}_{\sigma,T}(X) - 2T \cdot \f(\epsilon) \qquad\text{where}\qquad T(-\epsilon) = \exp(-c_X \cdot \epsilon)
\end{equation}
By the length spectrum property (Proposition \ref{prop:basic_properties_of_spectral_gaps}\ref{itm:MB_length}) and the fact that $T$ is not in the length spectrum of $\partial_+X$, we have
\[\mathfrak{g}_{\sigma,T(\epsilon)}(X) = \mathfrak{g}_{\sigma,T}(X)\]for sufficiently small $\epsilon$. Thus, by applying (\ref{eqn:gap_extension_agrees_1}) and Lemma \ref{lem:gap_extension_via_sequence}, we find that
\[
\g_{\sigma,T}(X) = \lim_{\epsilon \to 0} \; g_{\sigma,T}(X^{\smallneg\epsilon}_{\epsilon}) = \sup_{W,S}\mathfrak{g}_{\sigma,S}(W)
\]
Since $\g_{\sigma,T}(X)$ and the supremum are right continuous in $T$ (Proposition \ref{prop:basic_properties_of_spectral_gaps}\ref{itm:MB_length}) and the length spectrum $\LSpec(\partial_+X)$ is discrete, the case where $T \in \LSpec(\partial_+X)$ follows. \end{proof}

We can now define the extension of the elementary spectral gaps to all conformal cobordisms, with rational symplectic form.

\begin{definition} \label{def:spectral_gap_general} Fix  a conformal symplectic cobordism $(X,\Omega,Z)$ with $[\Omega] \in H^2(X;\Q)$, a curve type $\sigma$ and period $T$. The associated \emph{elementary symplectic field theory (ESFT) spectral gap} is
\begin{equation} \label{eqn:def_gap}
\g_{\sigma,T}(X) := \sup_{W,S}\mathfrak{g}_{\sigma,S}(W) \qquad\text{ as in }(\ref{eqn:sup_g_extension})
\end{equation}
\end{definition}

These spectral gaps obey an extensive list of axioms, as in the Morse-Bott case. For convenient reference, we collect these into the following axioms list.

\begin{prop}[Gap Axioms, Cobordisms] \label{prop:axioms_of_spectral_gaps} The ESFT spectral gaps of a conformal symplectic cobordism $(X,\Omega,Z)$ with $[\Omega] \in H_2(X;\Q)$ satisfy the following axioms.
\begin{enumerate}[label=(\alph*)]
    \item \label{itm:positive} (Positive) $\g_{\sigma,T}(X) \ge 0$ with equality if and only if $\sigma$ has no points and zero homology class.
        \[\sigma = (g,0,0)\]
	
    \item \label{itm:axiom_trace} (Trace) Let $X = (U \sqcup V)_{M}$ be the trace of $U \sqcup V$ along $M$ with $M \subset \partial_-U$ and $M \subset \partial_+V$. Fix curve types $\sigma = (g,m+n,A)$ on $X$, $\sigma_U = (g,m,A|_U)$ on $U$ and $\sigma_V = (g,n,A|_V)$ on $V$. Then
\[\mathfrak{g}_{\sigma,T}(X) \ge \mathfrak{g}_{\sigma_U,T}(U) + \mathfrak{g}_{\sigma_V,T}(V)\]
	\item \label{itm:axiom_Liouville_emb} (Liouville Embeddings) Let $W \to X$ be a symplectic embedding of a Liouville domain $W$ into a conformal cobordism $X$, and let $\sigma$ be a curve type on $X$. Then
	\[\mathfrak{c}_{\sigma|_W}(W) \le \mathfrak{g}_{\sigma,T}(X) \qquad\text{for any }T\]
       \item \label{itm:axiom_Gromov-Witten} (Gromov-Witten) Let $(M,\Omega)$ be a closed symplectic manifold with $X \subset M$ and assume $c_X = 0$. Fix a curve type $\sigma = (g,m,A)$ on $X$ and suppose that
       \[\GW_\tau(M,\Omega) \cap [\on{pt}]^{\otimes m} \neq 0 \qquad\text{where}\qquad \tau = (g,k,B) \text{ with }m \le k\text{ and }A = B|_{X}\]
       Then $\g_{\sigma,T}(X) \le \Omega \cdot A$ where $T = [\theta_+] \cdot (A \cap [\partial_+X])$ is the length of the positive end of $A|_X$.
       \vspace{3pt}
       
	\item \label{itm:axiom_sub_linearity} (Sub-Linearity) If $\sigma,\tau$ are two curve types and $S,T > 0$, then
	\[\mathfrak{g}_{\sigma + \tau,S+T}(X) \le \mathfrak{g}_{\sigma,S}(X) + \mathfrak{g}_{\tau,T}(X)\]
	\item \label{itm:axiom_disj_union} (Union) If $X = U \sqcup V$ is a disjoint union of conformal cobordisms $U$ and $V$ then 
	\[\mathfrak{g}_{\rho,T}(X) =\on{min} \big\{\mathfrak{g}_{\sigma,R}(U) + \mathfrak{g}_{\tau,S}(V) \; : \; \rho = \iota_U(\sigma) + \iota_V(\tau) \text{ and }T = R + S\big\}\]
	\item \label{itm:axiom_conformality} (Conformality) If $a$ is a positive real number, then
	\[\mathfrak{g}_{\sigma,T}(X,a \cdot \Omega,Z) = a \cdot \mathfrak{g}_{\sigma,T/a}(X,\Omega,Z) \qquad \mathfrak{g}_{\sigma,T}(X,\Omega,a \cdot Z) = \mathfrak{g}_{\sigma,T/a}(X,\Omega,Z)\]
	\item \label{itm:axiom_monotonicity} (Monotonicity) If $\sigma$ and $\tau$ are curve types, and $S,T > 0$ are periods then
	\[\mathfrak{g}_{\sigma,S}(X) \geq \mathfrak{g}_{\tau,T}(X) \qquad\text{if}\qquad \sigma \preceq \tau \quad\text{and}\quad S \le T\]
    \item \label{itm:axiom_thickening} (Thickening) Fix constants $\delta > 0 > \epsilon$ and fix a period $T$. Let $T(\delta) = \exp(c_X \cdot \delta) \cdot T$. Then
	\[|\mathfrak{g}_{\sigma,T(\delta)}(X^\delta_\epsilon) -  \mathfrak{g}_{\sigma,T}(X)| \le T \cdot (\f(\delta) - \f(\epsilon))\]
	\item \label{itm:axiom_hofer} (Hofer Lipschitz)  Fix a conformal cobordism $W$ and a period $T$ that satisfies
     \[[A^{-1}T,AT] \cap \LSpec(\partial_+X) = \emptyset \qquad\text{where}\qquad A := \exp(c_X \cdot d_H(X,W))\]
     Then the spectral gaps and the Hofer distance of $X$ and $W$ are related by
\[|\mathfrak{g}_{\sigma,T}(X) - \mathfrak{g}_{\sigma,T}(W)| \le 2T \cdot \f(d_H(X,W)) \]
	\item \label{itm:axiom_spectrality} (Area Spectrality) The spectral gap $\g_{\sigma,T}$ is valued in the area spectrum $\ASpec_T$.
    \[\mathfrak{g}_{\sigma,T}(X) \in \ASpec_T(X) \qquad\text{or}\qquad \mathfrak{g}_{\sigma,T}(X) = \infty\]
 \item\label{itm:axiom_length}  (Length) If $S$ and $T$ are periods with $S < T$ and $[S,T] \cap \LSpec(\partial_+ X) = S$, then \[\mathfrak{g}_{\sigma,T}(X) = \mathfrak{g}_{\sigma,S}(X)\]
 \item \label{itm:axiom_cover} (Cover) Let $\pi:U \to X$ be a covering map of conformal cobordisms and let $\sigma = (g,m,A)$ be a curve type in $U$. Then
       \[
       \g_{\sigma,T}(U) \ge \g_{\pi(\sigma),T}(X) \qquad\text{where}\qquad \pi(\sigma) = (g,m,\pi_*A)
       \]
\end{enumerate}
\end{prop}

\begin{proof} Almost every axiom is a straightforward consequence of the Morse-Bott version in \S \ref{subsec:spectral_gaps_MB}. The exception is the new property \ref{itm:axiom_hofer}, which we prove last.

\vspace{3pt}

\emph{\ref{itm:positive} - Positive.} This follows immediately from the positivity of the spectral gaps with Morse-Bott ends and the fact that we take a supremum in Definition~\ref{def:spectral_gap_general}.
\vspace{3pt}

\emph{\ref{itm:axiom_trace} - Trace.} Let $W_U$ and $W_V$ be Morse-Bott deformations of $U$ and $V$ with $W_U \subset \on{int}(U)$ and $W_V \subset \on{int}(V)$, and fix periods $S_U, S_V > T$. Since $X = (U \sqcup V)_M$, we naturally have inclusions
\[W_U \sqcup W_V \subset X\]
The union of the components of $X \setminus (W_U \sqcup W_V)$ that contain the trace region $M \subset X$ is simply the interior of a Morse-Bott intergraph cobordism $\Gamma^u_v M$ over $M$. The union of this region with $W_U$ and $W_V$ is a deformation $W_X$ of $X$ such that
\[
W_X \subset \on{int}(X) \qquad\text{and}\qquad W_X \simeq (W_U \cup \Gamma^u_v M \cup W_V)_N
\]
Therefore we may apply Lemma \ref{lem:general_trace} to see that
\[
\mathfrak{g}_{\sigma,W_X}(W_X) \ge \mathfrak{g}_{\sigma|_U,S_U}(W_U) + \mathfrak{g}_{\sigma|_V,S_V}(W_V) 
\]

\vspace{3pt}

\emph{\ref{itm:axiom_Liouville_emb}-\ref{itm:axiom_Gromov-Witten} - Liouville embedding, Gromov-Witten.} The proofs are directly analogous to \ref{itm:axiom_trace}, using Propositions \ref{prop:monotonicity_embedding_Liouville} and \ref{prop:MB_Gromov_Witten}. 

\vspace{3pt}
\emph{\ref{itm:axiom_sub_linearity}-\ref{itm:axiom_thickening} Sub-Linearity, Union, Conformality, Monotonicity, Thickening.} These follow in a straightforward way from Definition \ref{def:spectral_gap_general} and the corresponding axioms Proposition \ref{prop:basic_properties_of_spectral_gaps}\ref{itm:MB_sub_linearity}-\ref{itm:MB_monotonicity}.

\vspace{3pt}

\emph{\ref{itm:axiom_spectrality} - Area.} Assume that $\mathfrak{g}_{\sigma,T}(X)$ is finite. Let $(W_i,S_i)$ be a sequence of pairs in $\mathcal{C}(X,T)$ such that $W_i \to X$ in $C^\infty$ and $S_i \to T$. By Proposition \ref{prop:basic_properties_of_spectral_gaps}\ref{itm:MB_spectrality}, we may find a sequence of orbit sets
\[\Gamma_i^+\text{ in }\partial_+W_i \qquad\text{and}\qquad \Gamma^-_i \text{ in }\partial_+W_i \qquad\text{with}\qquad \mathcal{L}(\Gamma^-_i) \le \mathcal{L}(\Gamma^+_i) \le S_i\]
and a sequence of immersed surfaces $\Sigma_i \subset W_i$ with
\[\partial \Sigma_i = \Gamma^+_i - \Gamma^-_i \qquad\text{and}\qquad E_\Omega(A_i) = \mathfrak{g}_{\sigma,S_i}(W_i)\]
After passing to a subsequence, we may assume that $\Gamma^+_i \to \Gamma_+$ and $\Gamma^-_i \to \Gamma_-$ for orbit sets $\Gamma_+$ in $\partial_+X$ and $\Gamma_-$ in $\partial_-X$ with $\mathcal{L}(\Gamma_+) \le T$. By Lemma \ref{lem:gap_extension_via_sequence}, we know that
\[\g_{\sigma,T}(X) = \lim_{i \to \infty} \g_{\sigma,S_i}(W_i) = \lim_{i\to\infty}E_\Omega(A_i)\]
For large $i$, we may attach a $C^\infty$-small union $C_i$ of cylinders connecting $\Gamma^+_i$ to $\Gamma^-$ and $\Gamma^+_i$ to $\Gamma^-$ to acquire a sequence of homology classes
\[
B_i \in S(X;\Gamma_+,\Gamma_-) \qquad\text{with}\qquad |E_\Omega(B_i) - E_\Omega(A_i)| \le \big| \int_{C_i} \Omega \big| \to 0
\]
Thus $E_\Omega(B_i) \to \mathfrak{g}_{\sigma,T}(X)$. If $[\Omega]$ is rational, then $\ASpec(X;\Gamma_+,\Gamma_-)$ is discrete and thus $E_\Omega(B_i)$ must be constant for large $i$. Thus
\[
\mathfrak{g}_{\sigma,T}(X) \in \ASpec(X;\Gamma_+,\Gamma_-) \subset \ASpec_T(X)
\]
\emph{\ref{itm:axiom_length} - Length.} We start by choosing $R$ such that $R < T$ and 
\[[R,T] \cap \LSpec(\partial_+X) = \emptyset.\]
Let $W_i \subset X$ be a sequence of deformations with $W_i \to X$ in $C^\infty$, and fix sequences $R_i > R$ and $T_i > T$. Since $W_i$ converges smoothly to $X$, we have
\[
[R_i,T_i] \cap \LSpec(\partial_+W_i) = \emptyset \qquad\text{for sufficiently large }i
\]
Therefore by Proposition \ref{prop:basic_properties_of_spectral_gaps}\ref{itm:MB_length} and Lemma \ref{lem:gap_extension_via_sequence}, we see that
\[\g_{\sigma,R}(X) = \lim_{i \to \infty} \g_{\sigma,R_i}(W_i) = \lim_{i \to \infty}  \g_{\sigma,T_i}(W_i) = \g_{\sigma,T}(X)\]

\vspace{3pt}

\emph{\ref{itm:axiom_hofer} - Hofer Lipschitz.} Write $X$ as a cobordism $X:M \to N$. Since the length spectrum is closed, there is an $\epsilon$ such that any $a$ with $d_H(X,W) < a < \epsilon$ satisfies
\[[\exp(-c_Xa) T, \exp(c_Xa)T] \cap \LSpec(\partial_+X) = \emptyset\]
Fix an $a$ with this property. Since the Hofer distance $d_H(X,W)$ is bounded by $a$, there is a deformation pair $u,v$ for $X$ with
\[
W \simeq X^u_v \qquad\text{and}\qquad \max{\|u\|_{C^0},\|v\|_{C^0}} < a
\]
By the composition property (\ref{eqn:deformation_composition}), we can write the thickening $X^a_{\smallneg a}$ and the thinning $X^{\smallneg a}_a$ as
\[
X^a_{\smallneg a} = \Gamma N^v_{\smallneg a} \circ X^u_v \circ \Gamma M^a_u \qquad\text{and}\qquad X^u_v = \Gamma N_v^a \circ X^{\smallneg a}_a \circ \Gamma M_{\smallneg a}^u
\]
Therefore, by the trace property \ref{itm:axiom_trace}, the spectral gaps satisfy
\begin{equation} \label{eqn:Hofer_Lipschitz_1}
\mathfrak{g}_{\sigma,T}(X^{\smallneg a}_a) \le \mathfrak{g}_{\sigma,T}(X^u_v) = \mathfrak{g}_{\sigma,T}(W) \le \mathfrak{g}_{\sigma,T}(X^{a}_{\smallneg a})\end{equation}
On the otherhand, the thickening property \ref{itm:axiom_thickening} implies that
\begin{equation} \label{eqn:Hofer_Lipschitz_2}
|g_{\sigma,T}(X^{\smallneg a}_{a}) - g_{\sigma,R}(X)| \le 2T \cdot \f(a) \quad\text{where}\quad R = \exp(c_X \cdot a) \cdot T
\end{equation}
\begin{equation} \label{eqn:Hofer_Lipschitz_3}
|g_{\sigma,T}(X^{a}_{\smallneg a}) - g_{\sigma,S}(X)|  \le 2T \cdot \f(a) \quad\text{where}\quad S := \exp(-c_X \cdot a) \cdot T
\end{equation}
Recalling our choice of $a$, the interval $[S,R]=[\exp(-c_Xa) T, \exp(c_Xa)T] $ is disjoint from the length spectrum of $\partial_+X$. We apply the length axiom \ref{itm:axiom_length} to find that
\begin{equation} \label{eqn:Hofer_Lipschitz_4}
g_{\sigma,T(-a)}(X) = g_{\sigma,T(a)}(X) = g_{\sigma,T}(X)
\end{equation}
Combining (\ref{eqn:Hofer_Lipschitz_1}-\ref{eqn:Hofer_Lipschitz_4}), we find that
\[\g_{\sigma,T}(X) - 2T \cdot \f(a) \le \g_{\sigma,T}(W) \le \g_{\sigma,T}(X) + 2T \cdot \f(a)\]
\emph{\ref{itm:axiom_cover} - Cover.} This follows from the corresponding axiom in the Morse-Bott case. \end{proof}
 
\subsection{Gaps Of Stable Hamiltonian Manifolds} Finally, we introduce spectral gaps of conformal Hamiltonian manifolds with rational Hamiltonian $2$-form.

\begin{definition} \label{def:spectral_gap_manifold} Fix a conformal Hamiltonian manifold $(Y,\omega,\theta)$ with $[\omega] \in H^2(Y;\Q)$, a curve type $\sigma$ and period $T$. The  \emph{ESFT spectral gap} $\mathfrak{g}_{\sigma,T}(Y)$ is the limit
\[\mathfrak{g}_{\sigma,T}(Y) := \lim_{\epsilon \to 0} \; \mathfrak{g}_{\sigma,T}(\Gamma Y^0_{\smallneg\epsilon}) \qquad\text{for the intergraph}\qquad \Gamma Y^0_{\smallneg\epsilon} = [-\epsilon,0] \times Y \subset \R \times Y \]
\end{definition}

\begin{prop} \label{prop:axioms_of_spectral_gaps_manifolds} The ESFT spectral gaps of a closed, conformal Hamiltonian manifold $(Y,\omega,\theta)$ satisfy the following fundamental axioms.
\begin{enumerate}[label=(\alph*)]
    \item \label{itm:mfld_boundary} (Boundary) Let $(X,\Omega)$ be a conformal cobordism with $Y \subset \partial X$ and $[\Omega] \in H^2(X;\Q)$. Fix curve types $\sigma = (g,m+n,A)$ and $\tau = (g,m,A)$ on $X$, and a curve type $\rho = (g,n,A|_Y)$ on $Y$. Then
    \[\g_{\sigma,T}(X) \ge \g_{\tau,T}(X) + \g_{\rho,T}(Y)\]
    \item \label{itm:mfld_Gromov_Witten} (Gromov-Witten) Let $(M,\Omega)$ be a closed symplectic manifold with $Y \subset M$ and assume $Y$ has conformal constant zero. Fix a curve type $\sigma = (g,m,A)$ on $Y$ and suppose that
       \[\GW_\tau(M,\Omega) \cap [\on{pt}]^{\otimes m} \neq 0 \qquad\text{where}\qquad \tau = (g,k,B) \text{ with }m \le k\text{ and }A = B|_Y\]
       Then $\g_{\sigma,T}(Y) \le \Omega \cdot A$ where $T = [\theta_+] \cdot (A \cap [Y])$.
       
       \vspace{3pt}
       
	\item \label{itm:mfld_sub_linearity} (Sub-Linearity) If $\sigma,\tau$ are two curve types and $S,T > 0$, then
	\[\mathfrak{g}_{\sigma + \tau,S+T}(Y) \le \mathfrak{g}_{\sigma,S}(Y) + \mathfrak{g}_{\tau,T}(Y)\]
	\item \label{itm:mfld_disj_union} (Union) If $Y = M \sqcup N$ is a disjoint union of conformal Hamiltonian manifolds $M$ and $N$ then 
	\[\mathfrak{g}_{\rho,T}(Y) =\on{min} \big\{\mathfrak{g}_{\sigma,R}(M) + \mathfrak{g}_{\tau,S}(N) \; : \; \rho = \iota_M(\sigma) + \iota_N(\tau) \text{ and }T = R + S\big\}\]
	\item \label{itm:mfld_conformality} (Conformality) If $a$ is a positive real number, then
	\[\mathfrak{g}_{\sigma,T}(Y,a \cdot \omega,\theta) = a \cdot \mathfrak{g}_{\sigma,T/a}(Y,\omega,\theta) \qquad \mathfrak{g}_{\sigma,T}(X,\omega,a \cdot \theta) = \mathfrak{g}_{\sigma,T/a}(Y,\omega,\theta)\]
	\item \label{itm:mfld_monotonicity} (Monotonicity) If $\sigma$ and $\tau$ are curve types, and $S,T > 0$ are periods then
	\[\mathfrak{g}_{\sigma,S}(Y) \geq \mathfrak{g}_{\tau,T}(Y) \qquad\text{if}\qquad \sigma \preceq \tau \quad\text{and}\quad S \le T\]
        \item \label{itm:mfld_hofer} (Hofer Estimate) Fix a conformal Hamiltonian manifold $M$ of Hofer distance $d = d_H(Y,M)$ from $Y$ and a period $T$. Then
        \[
        \g_{\sigma,T(d)}(Y) \le \g_{\sigma,T}(M) + 2T \cdot \f(d) \qquad\text{where}\qquad T(d) = \exp(c_Y \cdot d) \cdot T
        \]
	\item \label{itm:mfld_spectrality} (Area Spectrality) If the cohomology class of $\omega$ is rational, i.e. $[\omega] \in H^2(Y,\Q)$, then
    \[\mathfrak{g}_{\sigma,T}(Y) \in \ASpec_T(Y) \qquad\text{or}\qquad \mathfrak{g}_{\sigma,T}(Y) = \infty\]
 \item\label{itm:mfld_length}  (Length) If $S$ and $T$ are periods with $S < T$ and $[S,T] \cap \LSpec(Y) = S$, then \[\mathfrak{g}_{\sigma,T}(Y) = \mathfrak{g}_{\sigma,S}(Y)\]
\end{enumerate}
\end{prop}

\begin{proof} The axioms of sub-linearity \ref{itm:mfld_sub_linearity}, union \ref{itm:mfld_conformality}, monotonicity \ref{itm:mfld_monotonicity}, area spectrality \ref{itm:mfld_spectrality} and length \ref{itm:mfld_length} follow easily from the corresponding axioms in Proposition \ref{prop:axioms_of_spectral_gaps} and Definition \ref{def:spectral_gap_manifold}. We will only prove the boundary and Hofer Estimate axioms.

\vspace{3pt}

\emph{\ref{itm:mfld_boundary} - Boundary.} First assume that $Y = \partial_-X$ and choose a small constant $\epsilon > 0$.  By the trace axiom (Proposition \ref{prop:axioms_of_spectral_gaps}\ref{itm:axiom_trace}), we have
\begin{equation} \label{eqn:mfld_boundary_1} \g_{\sigma,T}(X_{\smallneg\epsilon}^0) \ge \g_{\tau,T}(X) + \g_{\rho,T}(\Gamma Y^0_{\smallneg\epsilon}) \qquad\text{since}\qquad X_{\smallneg\epsilon}^0 = \Gamma Y^0_{\smallneg\epsilon} \circ X\end{equation}
By the thickening axiom (Proposition \ref{prop:axioms_of_spectral_gaps}\ref{itm:axiom_thickening}) and Definition \ref{def:spectral_gap_manifold},  we have
\[\lim_{\epsilon \to 0} \g_{\sigma,T}(X_{\smallneg\epsilon}^0) = \g_{\sigma,T}(X)\qquad\text{and}\qquad \lim_{\epsilon \to 0} \g_{\rho,T}(\Gamma Y^0_{\smallneg\epsilon}) = \g_{\rho,T}(Y)\]
Thus, taking the limit as $\epsilon \to 0$ in (\ref{eqn:mfld_boundary_1}) yields the desired inequality. Second, assume that $Y = \partial_+X$. Apply the trace axiom to see that
\begin{equation} \label{eqn:mfld_boundary_2} \g_{\sigma,T}(X) \ge \g_{\tau,T}(X^{\smallneg\epsilon}_0) + \g_{\rho,T}(\Gamma Y^0_{\smallneg\epsilon}) \qquad\text{since}\qquad X = X^{\smallneg\epsilon}_0 \circ \Gamma Y^0_{\smallneg\epsilon}\end{equation}
To take the limit as $\epsilon \to 0$, we note by the thickening axiom that
\[
|\g_{\tau,T}(X^{-\epsilon}_0) - \g_{\tau,T(\epsilon)}(X)| \le 
T(\epsilon) \cdot \f(\epsilon) \qquad\text{where}\qquad  T(\epsilon) = \exp(c_X \cdot \epsilon) \cdot T
\]
Moreover, $\g_{\tau,T}$ is right continuous in $T$ by the length axiom (Proposition \ref{prop:axioms_of_spectral_gaps}\ref{itm:axiom_length}). Thus
\[\lim_{\epsilon \to 0} \g_{\tau,T}(X^{-\epsilon}_0) = \lim_{\epsilon \to 0} \g_{\tau,T(\epsilon)}(X) = \g_{\tau,T}(X)\]
Thus, taking the limit as $\epsilon \to 0$ in (\ref{eqn:mfld_boundary_2}) yields the desired inequality. The general case where $Y$ is a union of components of $\partial_+X$ and $\partial_-X$ is analogous. 

\vspace{3pt}

\emph{\ref{itm:mfld_hofer} - Hofer Estimate.} Fix $\epsilon > 0$ and let $c = d + \epsilon$. By Definition \ref{def:hofer_distance_manifolds}, there is an $a:Y \to \R$ such that $M \simeq Y_a$ and $\|a\|_{C^0} < d + \epsilon/2$. Therefore we have an isomorphism
\[\Gamma M^c_{\smallneg c} \simeq \Gamma (Y_a)^c_{\smallneg c} = \Gamma Y^{a+c}_{a - c} = \Gamma Y^{\smallneg\epsilon/2}_{a-c} \circ \Gamma Y^0_{\smallneg\epsilon/2} \circ \Gamma Y^{a+c}_0\]
Thus, using the trace axiom for cobordisms (Proposition \ref{prop:axioms_of_spectral_gaps}\ref{itm:axiom_trace}) and Definition \ref{def:spectral_gap_manifold}, we have
\begin{equation} \label{eqn:hofer_estimate_1} \g_{\sigma,T(c)}(Y) \le \g_{\sigma,T(c)}(\Gamma Y^0_{\smallneg\epsilon/2}) \le \g_{\sigma,T(c)}(\Gamma M^c_{\smallneg c})\end{equation}
Now using the thickening axiom for cobordisms (Proposition \ref{prop:axioms_of_spectral_gaps}\ref{itm:axiom_thickening}), we deduce that
\begin{equation} \label{eqn:hofer_estimate_2} 
 \g_{\sigma,T(c)}(\Gamma M^c_{\smallneg c}) \le \g_{\sigma,T}(\Gamma M^0_{\smallneg\epsilon}) + T \cdot (\f(c) - \f(-c + \epsilon)) \le \g_{\sigma,T}(\Gamma M^0_{\smallneg\epsilon}) + 2T \cdot \f(c)
\end{equation}
Combining (\ref{eqn:hofer_estimate_1}) and (\ref{eqn:hofer_estimate_2}) and taking the limit as $\epsilon \to 0$ yields the desired result.
\end{proof}

\section{Applications To Closing Properties And Examples} \label{sec:apps_and_examples}
In this section, we develop the formal relationship between the ESFT spectral gaps and several strong variants of the closing property. 

\subsection{Closing Widths And Properties} \label{subsec:closing_width} We start by introducing generalizations of quantitative invariants and closing properties originally due to Hutchings \cite{hutchings2022elementary} and Irie \cite{irie2022strong}.

\vspace{3pt}

First, we have a measure of the size of a deformation of a conformal Hamiltonian manifold, formulated in terms of the Gromov width. This generalizes \cite[Def 1.1]{hutchings2022elementary}.

\begin{definition} \label{def:function_width} The \emph{width} $\wY(Y,f)$ of a non-negative, smooth function $f:Y \to [0,\infty)$ on a conformal Hamiltonian manifold $Y$ is given by
\[\wY(Y,f) := \cGr(\Gamma Y^f_0)\]
That is, $\wY(Y,f)$ is the Gromov width of the intergraph $\Gamma Y^f_0 \subset \R \times Y$  (i.e. the region below the graph of $f$ and above $0 \times Y$, see Definition \ref{def:intergraph_cobordism}). \end{definition}

There is a corresponding notion of the minimal size of a deformation required to guarantee a smooth closing behavior. This generalizes \cite[Def 1.2]{hutchings2022elementary}.

\begin{definition}\label{def:closing_width} The \emph{closing width of length $T$} for a conformal Hamiltonian manifold $Y$, denoted by
\[\close^T(Y,\omega,\theta) \qquad\text{or simply}\qquad \close^T(Y)
\]
is the minimum $\delta \in [0,\infty)$ such that, for any family of smooth functions on $Y$, given by
\begin{equation} \label{eqn:closing_width_def}
f:[0,1] \times Y \to [0,\infty) \qquad\text{with}\qquad f_0 = 0 \quad \text{and}\quad \wY(Y,f_1) > \delta
\end{equation}
there is an $s \in [0,1]$ such that the graphical deformation $Y$ by $f_s$ (see Definition \ref{def:graphical_deformation_SHS}) has a closed orbit of period $T$ or less that intersects the set
\[\on{supp}_Y(f) := \big\{y \in Y \; : \; f_s(y) > 0 \text{ for some }s\big\} \subset Y\]\end{definition}

\begin{lemma} \label{lem:right_subcontinuity_of_CloseT} The closing width $\close^T(Y)$ is non-increasing and right continuous in $T$. 
\end{lemma}

\begin{proof} If $S < T$, then 
 $\close^S(T)$ is a valid choice of $\delta$ in Definition \ref{def:closing_width} for $\close^T(Y)$. Thus $\close^T(Y) \le \close^S(Y)$. For the second claim, assume that $T_i > T$ and $T_i \to T$ and let
 \[
 \delta = \lim_{i \to \infty} \close^{T_i}(Y) \le \close^T(Y)
 \]
Fix a family $f$ as in (\ref{eqn:closing_width_def}) with $\wY(Y,f_1) > \delta$ and let $f(s) = f_s$ be the function at $s \in [0,1]$. Then for each $i$, there is an $s_i \in [0,1]$ and an orbit $\gamma_i$ of the deformation of $Y$ by $f(s_i)$  of period $\le T_i$ through $\supp_Y(f)$. By passing to a subsequence, we may assume that $s_i \to s$ and $\gamma_i \to \gamma$ where $\gamma$ is an orbit of $Y_{f(s)}$, with period bounded by $\lim_i T_i = T$. This shows that
\[\delta \ge \close^T(Y) \qquad\text{and thus}\qquad \close^T(Y) = \lim_{i \to \infty} \close^{T_i}(Y) \qedhere\] \end{proof}

We can relate this closing width to a non-quantitative, smooth closing property by using the following generalization of the strong closing property defined by Irie \cite{irie2022strong}.

\begin{definition}[Strong Closing] \label{def:strong_closing_property} A conformal Hamiltonian manifold $Y$ has the \emph{strong closing property} if for every non-zero function $f:Y \to [0,\infty)$, there exists $s \in[0,1]$ such that 
\[
Y_{sf} \qquad \text{has a periodic orbit passing through}\quad \on{supp}(f).
\]
\end{definition}

\begin{lemma}[Width Limit] \label{lem:width_limit} A conformal Hamiltonian manifold $Y$ has the strong closing property if
\[\inf_T \close^T(Y) = 0\]\end{lemma}

\begin{proof} If $f:Y \to [0,\infty)$ is not identically $0$, then $\wY(Y,f) > \close^T(Y)$ for some $T$. Thus the family $Y_{sf}$ has a closed orbit of period $T$ or less through $\supp_Y(sf) = \supp(f)$ for some $s \in [0,1]$.
\end{proof}

There is also a periodicity criterion using the closing width, stated as follows.

\begin{lemma}[Zero Width] \label{lem:zero_width} A conformal Hamiltonian manifold $Y$ is $S$-periodic with $S \le T$ if and only if
\[\close^T(Y) = 0\] \end{lemma}

\begin{proof} If $f:[0,1] \times Y \to \R$ is a family of functions as in (\ref{eqn:closing_width_def}) and $Y$ is $S$-periodic, then the deformation of $Y$ by $f_0 = 0$ has a period $\le T$ orbit through $\supp_Y(f)$ tautologically. Thus
\[\close^T(Y) = 0\]
Conversely, if $\close^T(Y) = 0$, then for any neighborhood $U$ we may choose a non-zero $f:Y \to [0,\infty)$ supported in $U$. Then
\[\wY(Y,sf) > \close^T(Y) \qquad\text{ for every }s > 0\]
Thus for some $t < s$ there is an orbit of $Y_{tf}$ of period $T$ through $U$. By taking the limit as $s \to 0$ and applying Arzel\'{a}–Ascoli, we extract a closed orbit of $Y$ through $U$, with period $\le T$. Since $U$ is arbitrary, we may shrink $U$ and acquire a closed orbit of period $\le T$ through any point.

\vspace{3pt}

 The Reeb flow of a stable Hamiltonian manifold is geodesible by a result of Sullivan \cite{s1978}, since it is transverse to the codimension $1$, invariant plane field $\on{ker}(\theta)$. Thus, by a result of Wadsley \cite{w1975}, the Reeb flow of $Y$ descends to a circle action and is periodic. \end{proof}

\subsection{Closing Criteria} \label{subsec:closing_criteria} The main result of this section is the following quantitative closing criterion, providing an estimate of the closing width in terms of certain spectral gaps. 
\begin{thm}[Closing Criterion, Quantitative] \label{thm:quant_closing}
    Let $(Y,\omega,\theta)$ be a conformal Hamiltonian manifold and assume that $[\omega]\in H^2(Y,\Q)$. Then
    \[\g_{\sigma, T}(Y)\geq \close^T(Y) \qquad\text{ for any pointed curve type $\sigma$ and length $T$}\]
\end{thm}

\begin{proof} Since $\mathfrak{g}_{\sigma,T}(Y)$ is right continuous in $T$ by Proposition \ref{prop:axioms_of_spectral_gaps}\ref{itm:axiom_length} and $\close^T(Y)$ is right sub-continuous in $T$ by Lemma \ref{lem:right_subcontinuity_of_CloseT}, we may assume that
\begin{equation} \label{eqn:quant_closing_1}  T \not\in \LSpec(Y)\end{equation}
Fix a smooth family of non-negative functions $f$ on $Y$, of the form in Definition \ref{def:closing_width}.
\[
f:[0,1] \times Y \to [0,\infty) \qquad\text{with}\qquad f_0 = 0 \quad \text{and}\quad \wY(Y,f_1) > \delta
\]
We also use the notation $f(s) = f_s$. For contradiction assume that, for all $s$, $Y_{f(s)}$ does not have a closed orbit (or more generally, orbit set) $\Gamma$ with
\[\mathcal{L}(\Gamma) \le T \qquad\text{and}\qquad \Gamma \cap \supp_Y(f) \neq \emptyset\]
We consider the family of functions $g_\epsilon:[0,1] \to [0,\infty)$ given by the spectral gap
\[g_\epsilon(s) = \mathfrak{g}_{\sigma,T}(\Gamma Y^{f(s)}_{-\epsilon}) \qquad\text{of the intergraph cobordism}\qquad \Gamma Y^{f(s)}_{-\epsilon}\]
Note that $T$ is not in the length spectrum of $Y_{f(s)}$ for any $s$. Thus by the Hofer Lipschitz axiom (Proposition \ref{itm:axiom_hofer}), $g_\epsilon$ is continuous in $s$. Moreover, we apply Lemma \ref{lem:area_spec_same} to see that
\[\ASpec_T\big(\Gamma Y^{f(s)}_{-\epsilon}\big) \qquad \text{ is independent of $s$ for any fixed $\epsilon$}\]
Thus $g_\epsilon$ is a continuous map to a meager set, and is constant for each fixed $\epsilon$. 

\vspace{3pt}

To acquire a contradiction, note that by Definition \ref{def:function_width} and our hypotheses on $f$, we know that
\[
\cGr\big(\Gamma Y^{f(1)}_{\smallneg\epsilon}) \ge \cGr\big(\Gamma^{f(1)}_0\big) = \wY(A,f(1)) > \mathfrak{g}_{\sigma,T}(Y) 
\]
In particular, there is a symplectic embedding $B \to \Gamma^{f(s)}_{\smallneg\epsilon}$ of a ball $B$ of Gromov width greater than $\mathfrak{g}_{\sigma,T}(Y)$. By the Liouville embedding axiom (Proposition \ref{prop:axioms_of_spectral_gaps}\ref{itm:axiom_Liouville_emb}) and Lemma \ref{lem:ball_gap} below, we have
\begin{equation} \label{eqn:quant_closing_2}
g_\epsilon(1) = \mathfrak{g}_{\sigma,T}\big(\Gamma^{f(1)}_{\smallneg\epsilon}\big) \ge \mathfrak{c}_\sigma(B) > \mathfrak{g}_{\sigma,T}(Y)
\end{equation}
However, by Definition \ref{def:spectral_gap_manifold}, we know that the spectral gap of $Y$ can be computed as
\begin{equation} \label{eqn:quant_closing_3}
\mathfrak{g}_{\sigma,T}(Y) = \mathfrak{g}_{\epsilon,T}(\Gamma Y^0_{\smallneg\epsilon}) = \lim_{\epsilon \to 0} g_\epsilon(0)
\end{equation}
The inequalities (\ref{eqn:quant_closing_2}) and (\ref{eqn:quant_closing_3}) contradict the fact that $g_\epsilon$ is constant in $s$, concluding the proof.\end{proof}

\begin{lemma} \label{lem:ball_gap} Let $B \subset \C^n$ be a ball of radius $r$ and let $\sigma$ be a pointed curve class. Then
\[\c_\sigma(B) \ge \cGr(B) = \pi r^2 \]
\end{lemma}

\begin{proof} By the spectrality axiom (Proposition \ref{prop:axioms_of_spectral_gaps}\ref{itm:axiom_spectrality}) and Proposition \ref{prop:basic_properties_of_spectral_gaps}\ref{itm:MB_positive}, we have
\[
\c_\sigma(B) \in \ASpec_T(B) \text{ for $T$ large}\qquad\text{and}\qquad \c_\sigma(B) \neq 0
\]
It follows that $\c_\sigma(B)$ is bounded below by the minimum period of a Reeb orbit. The boundary $\partial B$ is foliated by simple orbits of period $\cGr(B) = \pi r^2$ (cf. Example \ref{ex:ellipsoids}). \end{proof}

By combining Theorem \ref{thm:quant_closing} and Lemma \ref{lem:width_limit}, we immediately acquire a criterion for the strong closing property using the ESFT spectral gaps. 

\begin{cor}[Closing Criterion, Strong] \label{cor:gap_0_and_strong_closing}
    Let $(Y,\omega,\theta)$ be a conformal Hamiltonian manifold with $[\omega]\in H^2(Y,\Q)$. Suppose that
    \[
     \inf_{\sigma,T} \mathfrak{g}_{\sigma,T}(Y) = 0 \qquad\text{ taken over pointed curve classes $\sigma$ on $Y$ and all $T$}
    \]
    Then $Y$ satisfies the strong closing property.
\end{cor}

\subsection{Spectral Gaps Via GW Theory} \label{subsec:proof_of_periodic_gap_thm} In this section, we prove an estimate for the spectral gaps of periodic conformal Hamiltonian manifolds. 

\vspace{3pt}

We begin with the case where the conformal constant is zero. In this case, the result can be acquired by an easy application of the Gromov-Witten axiom for the spectral gaps.

 \begin{thm}  \label{thm:periodic_spectral_gap_cY0} Let $Y$ be a $T$-periodic, conformal Hamiltonian manifold with conformal constant zero. Let $\sigma$ be the curve class with genus $0$, $1$ point and homology class
\[A \in H_1(Y) \simeq H_2(X,\partial X) \qquad\text{where}\qquad X = \Gamma Y^0_{\smallneg\epsilon}\]
given by the homology class of a period $T$ orbit. Then
\[\g_{\sigma,T}(\Gamma^0_{\smallneg\epsilon} Y) \le -\f(-\epsilon) \cdot T
\]
 \end{thm}

\begin{proof} We first address the special case of a trivial mapping torus using the Gromov-Witten axiom, and then apply this to the general case using the covering axiom. 

\vspace{3pt}

\emph{Special Case.} To start, we consider the case of a trivial mapping torus. 
\[
(Y,\omega,\theta) = (\R/\Z \times M, \rho, \theta)
\]
for a closed symplectic manifold $(M,\rho)$. Consider the symplectic manifold
\[
S^2 \times M \qquad\text{with symplectic form} \qquad \Omega_\delta = \delta \cdot \omega_{S^2} \oplus \rho
\]
Here $\omega_{S^2}$ has unit area, so that the area of $S^2$ in $S^2 \times M$ with respect to $\Omega_\delta$ is $\delta$. There is a natural inclusion
\[
\R/\Z \times M \subset S^2 \times M
\]
where $\R/\Z$ is mapped to a small circle around one point. We may choose a  stabilizing vector-field $Z$ inducing the closed stabilizing $1$-form $\theta=dt$ on $(\R/\Z)_t \times M$, and this generates an embedding
\[
\Gamma^0_{\smallneg\epsilon} Y = [-\epsilon,0] \times \R/\Z \times M \to (S^2 \times M,\Omega_\delta), \qquad \delta > \epsilon \cdot T \quad\text{for }T = [\theta] \cdot [\R/\Z \times \on{pt}].
\]
Note the lower bound on $\delta$, which is imposed by area considerations. Now consider the curve type in $S^2 \times M$ given by
\[
\tau = (g,m,B) = (0,1,[S^2] \times [\on{pt}])
\]
By Lemma \ref{lem:GW_axiom_S2_times_X}, the corresponding Gromov-Witten invariant of $S^2 \times M$ has non-zero cup product with the fundamental class.
\[
\GW_\tau(S^2 \times M) \cap [S^2 \times M] \neq 0
\]
Therefore, by the Gromov-Witten axiom (Proposition \ref{prop:axioms_of_spectral_gaps}\ref{itm:axiom_Gromov-Witten}) of the spectral gaps, we have
\[
\g_{\sigma,T}(\Gamma^0_{\smallneg \epsilon} Y) \le \Omega_\delta \cdot B = \delta \quad \text{where}\quad \sigma = (0,1,[\R/\Z \times \on{pt}]) \text{ and }T = [\theta] \cdot [\R/\Z \times \on{pt}]
\]
Taking the infimum over $\delta > \epsilon \cdot T$ yields the desired result.

\vspace{3pt}

\emph{General case.} In the general case, we proceed as follows.  Let $(Y,\omega,\theta)$ be a closed conformal Hamiltonian manifold of conformal constant zero.  We may find an arbitrarily close closed $1$-form $\theta'$ to $\theta$ such that
\[
\theta'|_{\ker(\omega)} > 0 \text{ and }[\theta'] \in H^1(Y;\Q) \subset H^1(Y;\R)\]
Since $\theta'$ is non-vanishing, there is a submersion
\[
f:Y \to S^1 \qquad\text{with}\qquad df = \theta'
\]
Now let $M = f^{-1}(0)$ be a fiber of $f$ with symplectic form $\rho = \omega|_M$. The flow of $M$ by the Reeb vector-field $R$ of $(\omega,\theta)$ defines a covering map
\[
\R/\Z \times M \simeq \R/T\Z \times M \to Y 
\]
We can pullback $\omega$ and $\theta$ to get a conformal Hamiltonian structure $(\rho,\varphi)$ on $\R/\Z \times M$. We may then thicken the covering map to a covering of symplectic cobordisms
\[\Gamma^0_{\smallneg\epsilon}(\R/\Z \times M) \to \Gamma^0_{\smallneg\epsilon}Y\]
By the covering axiom (Proposition \ref{prop:axioms_of_spectral_gaps}\ref{itm:axiom_cover}) and the special case, we thus find that
\[
\mathfrak{g}_{\sigma,T}(\Gamma^0_{\smallneg\epsilon}Y) \le \mathfrak{g}_{\tau,T}(\Gamma^0_{\smallneg\epsilon}(\R/\Z \times M)) \le \epsilon \cdot T
\]
Here $\tau = (0,1,[\R/\Z \times \on{pt}])$ is the curve type pushing forward to a period $T$ orbit. This is the desired result. \end{proof}

In the general case, we can carry out a similar argument. However, this case relies on certain foundations of orbifold Gromov-Witten theory and SFT neck stretching in an orbifold setting (see Assumption \ref{ass:GW_Orbifolds}).

\begin{thm} \label{thm:periodic_spectral_gap} Let $Y$ be a $T$-periodic, conformal Hamiltonian manifold and assume Assumption \ref{ass:GW_Orbifolds}. Let
\[A \in H_1(Y) \simeq H_2(X,\partial X) \qquad\text{where}\qquad X = \Gamma Y^0_{\smallneg\epsilon}\]
be the homology class of a simple orbit, and let $\sigma$ be the curve class with genus $0$, $1$ point and homology class $A$. Then
\[\g_{\sigma,T}(\Gamma^0_{\smallneg\epsilon} Y) \le -\f(-\epsilon) \cdot T
\]
\end{thm}

\begin{proof} For simplicity, we restrict to the case where $Y$ is contact. The other cases admit a similar treatment. Let $X:= \Gamma_{\smallneg\epsilon}^0Y$ and choose a pair of an almost complex structure and a point
\[
(J,P) \in \mathcal{J}(X) \times \hat{X} \qquad\text{such that}\qquad \g_{\sigma,T}(X;J,P) \ge 0
\]
We proceed with the proof in four steps, as follows.

\vspace{3pt}

{\bf Step 1.} We start by constructing a closed symplectic orbifold that (the thickening of) $Y$ embeds into. Choose constants $a,c > 0$. By choosing $a$ large and flowing $(J,P)$ by the conformal vector-field (cf. Proposition \ref{prop:trace_property_MB}), we may assume that $J$ is cylindrical outside of
\[\Gamma^a_{\smallneg\epsilon}Y \subset \hat{X} \qquad\text{and}\qquad \Gamma^a_{\smallneg\epsilon}Y \subset \on{int}(W) \text{ where }W = \Gamma^{a+c}_{\smallneg (\epsilon + c)}Y\]
We may form the (orbifold) associated sphere bundle $\Sigma Y$ of $Y$ by taking $W$ and quotienting by the Reeb flow on the ends (see Definition \ref{def:associated_sphere_bundle}). This yields a closed symplectic orbifold
\[
(\Sigma Y,\Omega) \qquad\text{with a symplectic inclusion}\qquad  \Gamma^a_{\smallneg\epsilon} Y \subset \Sigma Y
\]
where the latter inclusion is disjoint from the orbifold points. Recall that $\Sigma Y$ may be viewed as an orbifold sphere bundle
\[
\Sigma Y \to V \qquad\text{over}\qquad V = Y/S^1
\]
We may also view $\Sigma Y$ as a composition $M_+ \circ \Gamma^a_{\smallneg\epsilon} Y \circ M_-$ where
\[
M_+ = \Gamma_a^{a+c}Y/\sim \qquad\text{and}\qquad  M_- = \Gamma^{\smallneg\epsilon}_{\smallneg(\epsilon + c)}Y/\sim
\]
Here $\sim$ is the quotient by the $S^1$-action at the positive end for $M_+$ and the negative end for $M_-$. Let $V_+ \subset M_+$ and $V_- \subset M_-$ be the symplectic sub-orbifolds given by the image of the positive and negative boundary of $W$. 

\vspace{3pt}

{\bf Step 2.} Next, we consider the fiber class of $\Sigma Y$ as a source of pseudo-holomorphic spheres. Consider the generic fiber $F$ of the projection $\Sigma Y \to V$. We note that
\[
F \cap V_\pm \text{ is a single point}
\]
Moreover, by Lemma \ref{lem:GW_axiom_orbifold_sphere_bundle}, we have
\begin{equation} \label{eqn:proof2_periodic_gap_1} \GW_\tau(\Sigma Y) \cap [\on{pt}] \neq 0 \qquad\text{where}\qquad \tau = (1,0,[F])\end{equation}
Finally, we note that the area of $F$ is the area in $W$ of a trivial cylinder over a period $T$ orbit in $Y$. This is easily computed as
\[
\Omega \cdot [F] = T \cdot (e^{a + c} - e^{-(\epsilon + c)}) 
\]

\vspace{3pt}

Next, choose a broken almost complex structure $I = (J_+,J,J_-)$ on $\Sigma Y$. Note that the singular sets $V_\pm \subset M_\pm$ are automatically $J_\pm$-holomorphic sub-orbifolds (see \cite[Lemma 2.1.2]{chen2002orbifold}). In particular, any $J_\pm$-holomorpic curve in $\hat{M}_\pm$ intersects $V_\pm$ transversely once. Let $I_s$ be the corresponding neck stretching structure. By the curve property of Gromov-Witten theory (see Theorem \ref{prop:GW_axioms}), there is a $I_s$-holomorphic sphere
\[u_s:S^2 \to (\Sigma Y,I_s) \quad\text{with}\quad P \in u_s(S^2) \text{ for each }s \in [0,\infty)\] 
Thus, as in Theorem \ref{prop:MB_Gromov_Witten}, we can extract a BEHWZ limit, i.e. a holomorphic building ${\bf u}$ in $\Sigma Y$ with respect to the broken almost complex structure $I$. Let
\[
u:\Sigma \to \hat{X} \qquad\text{asymptotic to }\Gamma_\pm\text{ at }\pm \infty
\]
By the usual arguments (see Lemma \ref{lem:minimizer_and_JP_semicont} or Proposition \ref{lem:general_trace}) we have
\[
P \in u(\Sigma) \qquad g(u) \le g({\bf u}) = 0 \qquad [u] = [{\bf u}] \cap X = A \in H_2(X,\partial X)
\]
Moreover, the energy $E_\Omega(u;\Gamma^a_{\smallneg\epsilon}Y)$ with respect to the cobordism $\Gamma^a_{\smallneg\epsilon}Y$ is
\[
E_\Omega(u;\Gamma^a_{\smallneg\epsilon}Y) \le \Omega \cdot F = T \cdot (e^{a+c} - e^{-\epsilon - c})
\]
In order to estimate the energy of $u$ with respect to the original cobordism $\Gamma_{\smallneg \epsilon}^0 Y$ we first need some estimates on the ends of $u$.\\

{\bf Step 3.} We next show that the ends of $u$ satisfy the following length estimate. 
\begin{equation}\label{eq:periodic_spectral_gap:periods_of_ends}
    \mathcal{L}(\Gamma_+) \le e^c \cdot T \qquad\text{ and }\qquad \mathcal{L}(\Gamma_-) \ge e^{-c} \cdot T.
\end{equation}
We prove the result for $\Gamma_+$, as the $\Gamma_-$ case is analogous. Let $u_+:\Sigma_+ \to \hat{M}_+$ be the $M_+$-level of the building ${\bf u}$. Note that $M_+$ is an (orbifold) symplectic cobordism
\[
M_+:\emptyset \to Y_a \qquad\text{and thus}\qquad \hat{M}_+ = M \cup (-\infty,0] \times Y_a
\]
Let $\pi:\hat{M}_+ \to M_+$ be the map given by the identity on $M_+$ and that projects the cylindrical end $(-\infty,0] \times Y_a$ into $\partial M_+ = Y_a$. Let $\bar{u}_+:\bar{\Sigma}_+ \to M_+$ be the map acquired by compactifying the composition $\pi \circ u_+$ along the punctures of $u_+$. Then $\bar{u}_+$ is a positive area surface with boundary on the negative ends $\Xi$ of $u_+$, and intersecting $V_+$ at one point transversely.

\vspace{3pt}

Next, let $W_+:Y_{a+c} \to Y_a$ be the cobordism $W_+ = \Gamma Y^{a+c}_a$. Recall that there is a quotient map
\[
W_+ \to M_+ = W_+/\sim
\]
that takes the positive boundary $Y_{a+c}$ of $W_+$ to the orbifold locus $V_+$ of $M_+$, and is a diffeomorphism elsewhere. We may lift $\bar{u}_+$ by this quotient map to get a surface
\[
S \subset W_+ \qquad\text{with}\qquad \partial S \subset \partial W_+
\]
This surface $S$ is of positive area in $W_+$. Since $u_+$ intersected $V_+$ at a single point transversely, the positive end of $W_+$ is a single closed Reeb orbit $\gamma$ of period $T$ (with respect to the original contact form on $Y$). The negative end is the orbit set $\Xi$. 

\vspace{3pt}

Since $W_+$ is an exact cobordism of contact manifolds and $S$ is positive area, this implies that
\[
e^a \cdot \mathcal{L}(\Xi) \le e^{a + c} \cdot \mathcal{L}(\gamma) = e^{a + c} \cdot T \qquad\text{and thus}\qquad \mathcal{L}(\Xi) \le e^c \cdot T
\]
Finally, the levels of ${\bf u}$ in the negative end $Y_a$ of $M_+$ form a $J$-holomorphic building with positive end $\Xi$ and negative end $\Gamma_+$. Thus $\mathcal{L}(\Gamma_+) \le \mathcal{L}(\Xi) \le e^c \cdot T$.

\vspace{3pt}

{\bf Step 4.}  Finally, we combine the other steps to prove the desired result. By (\ref{eq:periodic_spectral_gap:periods_of_ends}), we know that
\[
\mathcal{L}(\Gamma_+) \le e^c \cdot T
\]
Since the length spectrum of a Morse-Bott Hamiltonian manifold is discrete, it follows that for sufficiently small $c > 0$, we have $\mathcal{L}(\Gamma_+) \le T$. In particular,
\[
u \in \mathcal{M}_\sigma^T(X;J,P)
\]
By our assumption that $\g_{\sigma,T}(X;J,P) \ge 0$, it follows that $u$ has non-negative $\Omega$-energy with respect to the cobordism $\Gamma^0_{\smallneg\epsilon} Y$. In particular, by Lemma \ref{lem:energy_bounds_cobordism}\ref{itm:energy_bound_cobordism:length} and using (\ref{eq:periodic_spectral_gap:periods_of_ends}) again, we have
\[
e^{-c} \cdot T \le \mathcal{L}(\Gamma_-) \le \mathcal{L}(\Gamma_+)
\]
Again by the discreteness of the length spectrum, for small $c$ this implies that $\mathcal{L}(\Gamma_+) \ge T$ and thus that $\mathcal{L}(\Gamma_+) = T$. Finally, by the transformation rule for the energy under thickening (Lemma \ref{lem:area_change_under_thickening}), we have
\[
E_\Omega(u;\Gamma^0_{\smallneg\epsilon} Y)\le E_\Omega(u;\Gamma^a_{\smallneg\epsilon}Y) - \mathcal{L}(\Gamma_+) \cdot  (e^a - 1) \le   T \cdot (e^{a+c} - e^{-\epsilon - c}) - (e^a - 1) \cdot T
\]
Taking the limit as $c \to 0$, we acquire the inequality
\[
\g_{\sigma,T}(\Gamma^0_{\smallneg\epsilon} Y) \le E_\Omega(u;\Gamma^0_{\smallneg\epsilon} Y) \le T \cdot (1 - e^{-\epsilon}) = -\f(-\epsilon) \cdot T \qedhere
\]
\end{proof}

\noindent By combining Theorems~\ref{thm:periodic_spectral_gap_cY0}-\ref{thm:periodic_spectral_gap}, Theorem \ref{thm:quant_closing} and Lemma \ref{lem:zero_width}, we acquire the following criterion for a conformal Hamiltonian Reeb flow to be periodic.

\begin{cor}[Periodicity Criterion] \label{cor:gapT_0_implies_periodic}
    Let $(Y,\omega,\theta)$ be a conformal Hamiltonian manifold with $[\omega]\in H^2(Y,\Q)$, and either Assumption~\ref{ass:GW_Orbifolds} or $c_Y=0$.  Then the following are equivalent.
    \begin{enumerate}[label=(\alph*)]
        \item \label{itm:criterion_periodic} $Y$ is $T$-periodic (i.e. has $T$-periodic Reeb flow).
        \vspace{2pt}
        \item \label{itm:criterion_width} The closing width $\close^T(Y)$ of $Y$ at period $T$ is zero.
        \vspace{2pt}
        \item \label{itm:criterion_gap} The infimum of the ESFT spectral gaps $\g_{\sigma,T}(Y)$ over all pointed curve classes $\sigma$ is zero.
    \end{enumerate}
\end{cor}

\subsection{Near-Periodicity And Strong Closing} We next discuss the spectral gaps and closing properties for a special  class of conformal Hamiltonian manifolds.

\begin{definition}[Near-Periodic]\label{def:Hofer_near_periodic} A conformal Hamiltonian manifold $Y$ is \emph{Hofer near-periodic} if there exists a sequence of conformal Hamiltonian manifolds $Y_i$ such that
\[Y_i \text{ is $T_i$-periodic} \qquad\text{and}\qquad T_i \cdot d_H(Y_i,Y) \to 0\]
\end{definition}

We will provide several concrete and non-periodic examples in Sections \ref{subsec:contact_examples} and \ref{subsec:symplectomorphism_examples}. These will include the Reeb flows of irrational ellipsoids (cf. \cite{irie2022strong}) and more general examples.

\begin{remark} If the sequence $Y_i$ in Definition \ref{def:Hofer_near_periodic} arises via graphical deformation, e.g. as the graphs of functions $a_i$ in $\R \times Y$, then 
 the Hofer convergence condition is implied by
\[T_i \cdot \|a_i\|_{C^0} \to 0\]
Thus, Hofer near-periodicity may be viewed as a $C^0$-closeness condition on the Hamiltonian structures. In particular, the flow of a Hofer near-periodic $Y$ may not even be \emph{$C^0$-close} to a periodic flow. However, we are not aware of a Hofer near-periodic flows that are not high-regularity approximated by periodic ones. \end{remark}

Hofer near-periodic Hamiltonian manifolds have asymptotically vanishing spectral gaps, and thus satisfy the strong closing property.  
\begin{thm}\label{thm:closing_for_Hofer_nearly_periodic} Let $(Y,\omega,\theta)$ be a Hofer near-periodic, conformal Hamiltonian manifold with rational cohomology class $[\omega]\in H^2(Y,\Q)$, and assume either $c_Y=0$ or Assumption~\ref{ass:GW_Orbifolds}. Then
\[\g(Y) = \inf_{\sigma,T} \mathfrak{g}_{\sigma,T}(Y) =0\]\end{thm}

\begin{proof} Since $Y_i$ is periodic for each $i$, we apply Theorem~\ref{thm:periodic_spectral_gap_cY0} if $c_Y=0$, or Theorem~\ref{thm:periodic_spectral_gap} otherwise, to acquire a curve class $\sigma_i$ of $Y$ with
\[\g_{\sigma_i,T_i}(Y_i) = 0\]
Then by the Hofer estimate axiom  (Proposition \ref{prop:axioms_of_spectral_gaps_manifolds}\ref{itm:mfld_hofer}), we have
\[
\g_{\sigma_i,S_i}(Y) \le \g_{\sigma_i,T_i}(Y_i) + 2T_i \cdot \f(d_H(Y_i,Y)) = 2T_i \cdot \f(d_H(Y_i,Y))
\]
where $S_i$ is given by $T_i \cdot \exp(c_Y \cdot d_H(Y_i,Y))$. The function $\f$ is smooth, with $\f(0) = 0$ and $\f'(0) = 1$. Since $d_H(Y_i,Y)$ approaches $0$ as $i \to \infty$, it follows that
\[\inf_{\sigma,T} \g_{\sigma,T}(Y) \le \lim_{i \to 0} \g_{\sigma_i,S_i}(Y) \le \lim_{i \to 0} 4T_i \cdot d_H(Y_i,Y) \cdot T_i = 0 \qedhere \]
\end{proof}

By applying the strong closing criterion in Corollary \ref{cor:gap_0_and_strong_closing}, we have the following conclusion.

\begin{cor} Assuming either $c_Y=0$ or Assumption~\ref{ass:GW_Orbifolds}, every Hofer near-periodic, conformal Hamiltonian manifold $(Y,\omega,\theta)$ with $[\omega]\in H^2(Y,\Q)$ satisfies the strong closing property.
\end{cor}

\subsection{Contact Examples} \label{subsec:contact_examples} 
Contact manifolds are simply conformal Hamiltonian manifolds of conformal constant $1$. We now translate our results into more standard language in contact topology.

\vspace{3pt}

We start by rephrasing the strong closing property and near-periodicity.

\begin{definition}[Strong Closing, Contact Version]
A contact manifold  $(Y,\xi)$ with contact form $\alpha$ has the \emph{strong closing property} if for every non-zero $f:Y \to [0,\infty)$, there is an $s \in[0,1]$ such that 
\[
\exp(sf) \cdot\alpha \qquad \text{has a periodic orbit passing through }\supp(f).
\]
\end{definition}

\begin{definition}[Near-Periodicity, Contact Case]
    A contact form $\alpha$ on $(Y,\xi)$ is \emph{Hofer near-periodic} if there exist contact forms $\alpha_i = \exp(f_i) \cdot \alpha$ on $(Y,\xi)$ such that
    \[
    \text{the Reeb flow of }\alpha_i\text{ is }T_i\text{ periodic} \quad \text{and}\quad \|f_i\|_{C^0}\cdot T_i\xrightarrow[i\rightarrow \infty]{}0.
    \]
\end{definition}
In this setting, Theorem~\ref{thm:closing_for_Hofer_nearly_periodic} states that every near-periodic contact form has the strong closing property. Let us describe examples of non-periodic Reeb flows to which this result applies.

\begin{example}[Ellipsoids] \label{ex:ellipsoids} Our first example is the Reeb flows on ellipsoids, for which the strong closing property was first proved in \cite{chaidez2022contact} (and later by \cite{x2022,cineli2022strong}). 

\vspace{3pt}

For this example, fix real numbers $a_1,\dots,a_n$ with $ 0 < a_1 \le \dots \le a_n$. Define the function
\[
f_a:S^{2n-1} \to \R \qquad\text{given by}\qquad f_a(z_1,\dots, z_n) = \sum_i \frac{\pi}{a_i} \cdot |z_i|^2 
\]
Consider the contact manifold $S^{2n-1}$ equipped with the contact form $\alpha_a = f_a^{-1} \cdot \lambda|_{S^{2n-1}}$, where $\lambda$ is the standard Liouville form on $\C^n$. That is
\[\lambda = \frac{1}{2} \cdot \sum_{j=1}^n x_j dy_j - y_j dx_j.\]
The Reeb flow is extremely explicit in this case, and given by
\[
\phi^t (z_1,\dots,z_n) = \left(e^{2\pi i t/a_1} z_1,\dots, e^{2\pi i t/a_n} z_n \right). 
\]
Note that this flow coincides with the Reeb flow on the contact boundary of the ellipsoid
\[
E(a) = f_a^{-1}((-\infty,1]) \qquad\text{with Liouville form }\lambda|_{E(a)}
\]
If $a$ is a multiple of an integer vector, the Reeb flow of $\alpha_a$ is periodic. However, if $a_1,\dots,a_n$ are rationally independent, then the Reeb flow is non-degenerate with $n$ simple orbits. Nonetheless, we have the following result.

\begin{lemma} \label{lem:ellipsoid_near_periodic} The contact form $\alpha_a$ on $S^{2n-1}$ is Hofer near-periodic for any choice of $a = (a_1,\dots,a_n)$. \end{lemma}

\begin{proof} The result \cite[Lemma 5.6]{chaidez2022contact} states that, for each $\varepsilon > 0$, there exists $r=(r_1,\dots,r_n)$ with rationally dependent entries such that
\begin{equation} \label{eqn:approx_ellipsoid} \frac{T}{T+\varepsilon}\cdot f_r \leq f_a\leq f_r.\end{equation}
Here $T$ is the period of the Reeb flow on $(S^{2n-1},\alpha_r)$, given explicitly by
\[T_r :=\min{N\in \N: r_i^{-1} \cdot N\in \N \text{ for all }i}\]
Note that the \cite[Lemma 5.6]{chaidez2022contact} is stated in terms of $\partial E(a)$, instead of $(S^{2n-1},\alpha_a)$. For each $i\in \N$, let $\alpha_i = \alpha_r$ for the (multiple of a) rational vector $r=(r_1,\dots,r_n)$ such that (\ref{eqn:approx_ellipsoid}) holds for  $\varepsilon_i:=1/i$, and let $T_i$ the period of $\alpha_i$. Then
\[\|\alpha_a-\alpha_i\|_{C^0}\cdot T_i = \|1/f_a - 1/f_r\|_{C^0}\cdot T_i \leq \frac{1}{\|f_a\|_{C^0}^2} \cdot \|f_r - f_a|\cdot T_i \leq \frac{1}{\|f_a\|_{C^0}^2} \cdot \frac{\varepsilon_i}{T_i+\varepsilon_i}\cdot T_i \xrightarrow[i\rightarrow\infty]{}0.
\]
We conclude that $\alpha_a$ is Hofer near-periodic and satisifes the strong closing property.\end{proof} \end{example}

\begin{example}[Contact $\T^2$-Manifolds] \label{ex:contactT2manifolds} Our next class of examples consists of contact manifolds admitting certain $\mathbb{T}^2$-actions.

\vspace{3pt}

This class of examples is originally due to Albers-Geiges-Zehmisch \cite{albers2018reeb}, and includes $S^1$ bundles over $\C P^n$, as well as over certain blow-ups of $\C P^n$. In \cite{chaidez2022contact} the strong closing property was established for a very constrained subclass of such manifolds. 

\vspace{3pt}

Let $(Y,\xi)$ be a contact manifold with contact form $\alpha_1$ and Reeb vector-field $R_1$ generating an $S^1$-action. Let $B = Y/S^1$ be the symplectic orbifold quotient, and consider a Morse function
\[H:B \to \R\]
generating an $S^1$-action on $B$ and with minimum value $1$. Let $\tilde H:Y\rightarrow \R$ be the lift of $H$ to a Hamiltonian on $Y$. Moreover, denote by $(X_H)_\xi$  the lift of $X_H$ to a vector-field tangent to $\xi = \on{ker}(\alpha_1)$.
Then 
\[R_2 = \widetilde{H} \cdot R_1 + (X_H)_\xi\]
is a Reeb vector-field with contact form $\alpha_2 = \alpha_1/\widetilde{H}$.
As explained in \cite[Lemma 6.1]{chaidez2022contact}, $R_1$ and $R_2$ generate a $\mathbb{T}^2$-action by Reeb vector-fields on $(Y,\xi)$. For each $a = (a_1,a_2)$ with $a_i > 0$, we have a Reeb vector-field
\[R_a = a_1 R_1 + a_2 R_2 \quad\text{with contact form}\quad \alpha_a = \frac{\alpha}{a_1 + a_2\widetilde{H}} \ .\]
\noindent The assumption that $H$ is Morse implies that the above $\mathbb{T}^2$ action is effective. 

\begin{lemma} \label{lem:contact_T2_near_periodic} The contact forms $\alpha_a$ on $Y$ are Hofer near-periodic, for any $a = (a_1,a_2)$. \end{lemma} 

\begin{proof} By \cite[Proposition 6.10]{chaidez2022contact}, we know that for any $\varepsilon >0$ there exist rationally dependent $r_1,r_2>0$ such that
\[\alpha_r\leq \alpha_a \leq \frac{T+\varepsilon}{T}\alpha_r\]
where $T$ is the period of the Reeb flow of $\alpha_r$. As in the ellipsoid case, choosing $\varepsilon_i=1/i$, denoting  $\alpha_i:=\alpha_r$ for corresponding $r$ and by $T_i$ the period of $\alpha_i$, we have
\[
\|\alpha_a - \alpha_i\|_{C^0}\cdot T_i\leq \|\alpha_a\|_{C^0}\cdot \frac{\varepsilon_i}{T_i+\varepsilon_i}\cdot T_i \xrightarrow[i\rightarrow \infty]{}0.
\]
Thus the contact form $\alpha_a$ is Hofer nearly periodic and hence has the strong closing property.
\end{proof}
\end{example}

\begin{example}[Anosov-Katok contact forms]
    We construct Hofer nearly periodic contact forms in the spirit of Anosov-Katok pseudorotations. 
    Let $Y$ be a contact manifold that admits a $\T^k$ action by Reeb flows,
    \[
    {a}=(a_1,\dots,a_k)\in\T^k\quad\mapsto\quad \alpha_a \text{ generating the Reeb vector field }R_a = a_1R_1+\cdots +a_kR_k.
    \]
    \begin{lemma}
         For every ${a} = (a_1,\dots, a_k)\in \T^k$ and any sequence of diffeomorphisms $\psi_i:Y\rightarrow Y$, there exist rational points $a^{i}\in \T^k$ converging to $a$ such that 
         \[
         \lim_{i\rightarrow \infty} \ (\psi_i\circ\cdots\circ\psi_1)^* \alpha_{a^{i}} \quad \text{is Hofer nearly periodic, whenever it exists.}
         \]
    \end{lemma}
    \begin{proof}
        As mentioned above, \cite[Proposition 6.10]{chaidez2022contact} guarantees that for any $\varepsilon_i$ there exits a rational point $a^i\in \T^k$ such that 
        \begin{equation}\label{eq:AKPR_T_C0norm_bnd}
            \|\alpha_a - \alpha_{a^i}\|_{C^0}\cdot T_i\leq \|\alpha_a\|_{C^0} \cdot \frac{T_i}{T_i+\varepsilon_i}\cdot\varepsilon_i\leq \|\alpha_a\|_{C^0} \cdot \varepsilon_i
        \end{equation}
        where $T_i$ is the period of the flow generated by $\alpha_{a^i}$. In what follows, we will assume that the periods $T_i$ tend to infinity. Otherwise, the limit mentioned above is going to periodic, whenever it exists, and the conclusion is clear. By passing to a subsequence we may assume that the periods are non-decreasing, i.e. $T_i\leq T_{i+1}$ for all $i$. Choose 
        \[
        \varepsilon_i<\frac{1}{2^i}\min{\|\psi_{i+1}\|_{C^0}, \|\psi_i\|_{C^0}}\quad \text{ and denote }\quad \alpha_i:= (\psi_i\circ\cdots\circ \psi_1)^*\alpha_{a^i}.\]
        First, let us show that $\{\alpha_i\}_i$ is a Cauchy sequence with respect to the $C^0$ norm. Then, we will show that, if  its $C^0$-limit is a smooth contact form, then it is Hofer nearly periodic. 
        Indeed, 
        \[
        \|\alpha_{i+\ell} - \alpha_i\|_{C^0} \leq \sum_{j=i+1}^{i+\ell} \|\alpha_{j} - \alpha_{j-1} \|_{C^0} = \sum_{j=i+1}^{i+\ell} \| \psi_j^*\alpha_{a^{j}} - \alpha_{a^{j-1}} \|_{C^0}
        \]
        where the equality follows from the fact that the $C^0$-norm is invariant under pull-back. For each $j\in\{i+1,\dots,i+\ell\}$, we have
        \[
        \| \psi_j^*\alpha_{a^{j}} - \alpha_{a^{j-1}} \|_{C^0}\leq \|\psi_j^*\|_{C^0} \cdot \| \alpha_{a^{j}} - \alpha_{a^{j-1}} \|_{C^0} \leq \|\psi_j^*\|_{C^0} \cdot \Big(\| \alpha_{a^{j}} -\alpha_a\|_{C^0}+\|\alpha_a- \alpha_{a^{j-1}} \|_{C^0} \Big).
        \]
        Applying the bound in (\ref{eq:AKPR_T_C0norm_bnd}) to both summands above and recalling our choice of $\varepsilon_j$, we obtain
        \[
        \| \psi_j^*\alpha_{a^{j}} - \alpha_{a^{j-1}} \|_{C^0}\leq \|\psi_j^*\|_{C^0} \cdot \|\alpha_a \|_{C^0} \cdot \Big(\frac{\varepsilon_{j}}{T_{j}} +\frac{\varepsilon_{j-1}}{T_{j-1}}\Big)\leq \frac{1}{2^{j-2}}\cdot \|\alpha_a\|_{C^0} \cdot \frac{1}{T_{j-1}},
        \]
        Where we used our assumption that the periods are non-decreasing. Summing now over $j$ from $i$ to $i+\ell$ we obtain
        \begin{align*}
            \|\alpha_{i+\ell} - \alpha_i\|_{C^0}& \leq \frac{\|\alpha_a\|_{C^0}}{T_i} \cdot \sum_{j=i+1}^{i+\ell} \frac{1}{2^{j-2}}\\ 
            &\leq\frac{\|\alpha_a\|_{C^0}}{ T_i} \cdot \frac{1}{2^{i-1}} \xrightarrow[i\rightarrow\infty]{}0.
        \end{align*}
        This proves that $\alpha_i$ is a Cauchy sequence and hence converges in the $C^0$ norm. Assume that the limit $\alpha$ is a smooth form, and let us show that it is Hofer nearly periodic. Since the forms $\alpha_i$ converging to $\alpha$ are periodic of periods $T_i$, it is sufficient to prove that 
        \[\|\alpha_i - \alpha\|_{C^0}\cdot T_i\xrightarrow[i\rightarrow \infty]{}0.\]
        Indeed, for large enough $\ell$,
        \[
            T_i\cdot \|\alpha_i - \alpha\|_{C^0}\leq 2 T_i \cdot \|\alpha_i - \alpha_{i+\ell}\|_{C^0}\leq \|\alpha_a\|_{C^0}\cdot \frac{1}{2^{i-2}}\xrightarrow[i\rightarrow\infty]{}0.\qedhere
        \]
    \end{proof} 
\end{example}

\subsection{Symplectomorphism Examples} \label{subsec:symplectomorphism_examples} The mapping torus of a symplectomorphism is conformal Hamiltonian with conformal constant $0$. We next translate our results into this setting.

\vspace{3pt}

Let $(P,\omega_P)$ be a symplectic manifold and let $\Phi:P \to P$ be a symplectomorphism. Recall that the mapping of $\Phi$ is given by
\[M_\Phi:=[0,1]\times P/\sim\quad \text{ where }\quad (0,\Phi(x))\sim(1,x)\]
equipped with the Hamiltonian $2$-form $\omega_\Phi$ lifting $\omega_P$ and the conformal, stabilizing $1$-form $dt$. We will work primarily with symplectomorphisms satisfying a rationality hypothesis.

\begin{definition}\label{def:rational_symplectomorphism} A symplectomorphism $\Phi$ is \emph{rational} if $[\omega_\Phi] \in H_2(M_\Phi;\Q)$, or equivalently if $[\omega_P]$ is rational and the image of the Flux map
\[\on{Flux}(\Phi) \in H^1(P;\R) \qquad\text{satisfies}\qquad \on{Flux}(\Phi) \cdot A \in \Q \text{ for any }A \in \on{ker}(\on{Id} - \Phi_*) \subset H_1(P)\]\end{definition}

\begin{definition}[Strong Closing, Symplectomorphisms]\label{def:strong_closing_mapping tori}
A symplectomorphism $\Phi:P\rightarrow P$ has the \emph{strong closing property} if for every non-zero Hamiltonian
\[h:\R/\Z \times P \to [0,\infty) \qquad\text{with}\qquad h_t = 0 \text{ near }t = 0\]
there exists $s \in [0,1]$ such that the symplectomorphism
\[
\phi_{sh}^1\circ \Phi \quad \text{has a periodic point in}\quad \supp_P(h_t) = \big\{p \in P\; : \; h_t(p) > 0 \text{ for some }t \in \R/\Z\big\}
\]
Here $\phi_{sh}^t$ is the Hamiltonian flow of $s h$.    
\end{definition}

\begin{definition}[Near-Periodicity, Symplectomorphisms]
    A symplectomorphism $\Phi:P\rightarrow P$ is \emph{Hofer near-periodic} if there exist Hamiltonian diffeomorphisms  $\phi_i\in \on{Ham}(P,\omega_P)$ such that
    \[
    \phi_{i}^1\circ \Phi \text{ is } T_i \text{ periodic}\quad \text{and}\quad \|\phi_i\|_{H}\cdot T_i\xrightarrow[i\rightarrow \infty]{}0,
    \]
    Here $\|\phi\|_H$ is the standard \emph{Hofer norm} of the Hamiltonian diffeomorphism $\phi$, given by the supremum over all generating Hamiltonians $h:\R/\Z \times P \to P$ of the quantity
    \[\|h\|_H = \int_0^1 \big(\on{max}_x h(t,x) - \on{min}_x h(t,x)\big)\ dt\]
\end{definition}
Note that by the definition of the Hofer norm, there exist Hamiltonians $h_i:\R/\Z\times P\rightarrow \R$ that generate $\phi_i$ at time 1 and $\|h_i\|_H\cdot T_i\rightarrow 0$. By a time reparametrization one can find Hamiltonians $h_i$ generating $\phi_i$ such that $2\cdot \|h_i\|_{C^0}\cdot T_i\rightarrow 0$. In this setting, Theorem~\ref{thm:closing_for_Hofer_nearly_periodic} states the following.
\begin{cor}\label{cor:closing_for_symplectomorphisms}
    Let $(P,\omega_P)$ is a symplectic manifold such that $[\omega_P]\in H^2(P,\Q)$. Every rational, Hofer near-periodic symplectomorphism has the strong closing property.
\end{cor}

\noindent We now describe some examples of symplectomorphisms to which Corollary~\ref{cor:closing_for_symplectomorphisms} can be applied.

\begin{example}[Hamiltonian Torus Actions] \label{ex:ham_torus_action} Let $(P,\omega_P)$ be a symplectic manifold such that $[\omega_P]\in H^2(P,\Q)$ and assume it admits a  Hamiltonian $\T^k$ action
\[
\phi:\T^k \times P \to P \qquad\text{with moment map}\qquad \mu:P \to \R^k
\]
Given a vector $a \in \R^k$, we let $\phi_a:P \to P$ be the Hamiltonian diffeomorphism giving action by $a \mod \Z^k$. It is generated by the Hamiltonian
\[
h_a := \langle \mu,a\rangle = \sum_i a_i \cdot \mu_i
\]
If $r \in \Q^n \subset \R^n$ is a rational vector, then the corresponding map $\phi_r$ is periodic with period
\[
T_r := \on{min}\big\{N\in \N\ :\  {N}\cdot r\in \Z^k\big\}. 
\]

\begin{prop}\label{prop:Ham_torus_action_nearly_per} The Hamiltonian diffeomorphism $\phi_a$ is Hofer nearly periodic for any $a \in \R^k$.
\end{prop}

\begin{proof} Fix $a \in \R^n$. By simultaneous Dirichlet approximation, for any $L > 0$ there exist integers $q_1,\dots ,q_k$ such that $\gcd{q_1,\dots,q_k} =1$ and a $T \le N$ such that
\[\Big|a_i - \frac{q_i}{T}\Big|\leq \frac{1}{T\cdot L^{1/k}} \quad \text{for every}\quad i=1,\dots,k.\]
Take $L=\varepsilon^{-k}$ and let $r \in \Q^k$ be the rational vector with entries $r_i = q_i/T$ with $q_i$ as above. Then
\[T_r = \min{N\in \N: N\cdot \frac{q_i}{T_r}\in \N\text{ for all }i } =  \min{N\in \N: N\cdot r\in \N^k } \]
Moreover, the approximation estimate translates to $T_r \cdot |a_i-r_i|\leq \varepsilon$. Then we have 
\[T_r \cdot \|h_{a - r}\|_{C^0} = T_r \cdot \|\langle \mu,r\rangle\|_{C^0}  \le \|\mu\|_{C^0} \cdot T_r \cdot \on{max}_i|a_i-r_i|\leq C(\mu) \cdot \varepsilon \qquad\text{where}\qquad C(\mu) = \|\mu\|_{C^0}\]
Now note that that $\phi_a = \phi_{a - r} \circ \phi_r$ and $\phi_{a-r}$ has Hamiltonian $h_{a - r}$. Thus
\[T_r \cdot \|\phi_{r - a}\|_H \le T_r \cdot \|h_{r-a}\|_{C^0} \le C(\mu) \cdot \epsilon\]
Since $\epsilon$ is arbitrary, we may choose $\epsilon = 1/j$ and use $\phi_j = r_j - a$ to get the near-periodicity. \end{proof}\end{example}

\begin{example}[Rotations of $\C P^n$.] As a specific case of Example \ref{ex:ham_torus_action}, consider the Hamiltonian action of $\T^n$ on $\C P^n$ given by
\[\phi_a([z_0:z_1:\cdots:z_n]):= [z_0:e^{2\pi i a_1}z_1:\cdots:e^{2\pi i a_n}z_n] \quad\text{for}\quad a \in \R^n\]
By Proposition~\ref{prop:Ham_torus_action_nearly_per}, these Hamiltonian diffeomorphisms $\phi_a$ are Hofer nearly periodic and thus have the strong closing property. As noted in \cite{x2022,cineli2022strong}, the strong closing property for ellipsoids, established in \cite{chaidez2022contact} and in Section \ref{subsec:contact_examples}, is closely related to the  strong closing property for these rotations. \end{example}

\section{Spectral Gap For Periodic Flows} \label{sec:gaps_of_periodic_SHS} In this section, we provide an alternate, more direct proof of Theorems \ref{thm:periodic_spectral_gap_cY0}-\ref{thm:periodic_spectral_gap}, with the caveat that we must add some technical hypotheses. Precisely, we prove the following result.

\vspace{3pt}

\begin{thm}\label{thm:gap0_for_periodic}
Let $(Y, \omega, \theta)$ be a closed conformal Hamiltonian manifold with $T$-periodic Reeb flow. Additionally, assume that
\begin{itemize}
    \item $Y$ has no contractible Reeb orbits and
    \item $Y$ is positively monotone. That is, $[\omega] = a \cdot c_1(\ker \theta)$ for some $a \ge 0$.
\end{itemize}
Let $\sigma = (0,1,[\gamma])$ where $[\gamma] \in H_1(Y)$ is the homology class of a closed orbit of period $T$. Then \[
\g_{\sigma,T}([-\epsilon,0] \times Y)\leq -\f(-\epsilon) \cdot T \qquad\text{for any}\qquad \epsilon > 0
\]
\end{thm}

\begin{remark}[Hypotheses] Note that the first hypothesis of Theorem \ref{thm:gap0_for_periodic} is automatic when $Y$ has conformal constant zero (i.e. is a mapping torus). In that case, $\theta$ is closed and the pairing of a closed orbit $\gamma$ with the cohomology class $[\theta]$ agrees with the period. In particular, $[\gamma] \in H_1(Y)$ is non-zero. Similarly, note that the second hypothesis of  Theorem \ref{thm:gap0_for_periodic} is automatic when $Y$ has conformal constant non-zero (i.e. is contact) Indeed, in this case
\[[\omega] = [d\alpha] = 0 = 0 \cdot [c_1(\ker \theta])\] \end{remark}

Specializing to the settings of contact manifolds and mapping tori separately, we can state the result as follows.

\begin{cor}
    Let $(Y,\alpha)$ be a hypertight contact manifold with $T$-periodic Reeb flow. Then
    \[
    \g_{\sigma,T}([-\epsilon,0] \times Y)\leq (1-e^{-\epsilon}) \cdot T.
    \]
\end{cor}
\begin{cor}
    Let $(P, \omega)$ be a positively monotone symplectic manifold, i.e., $[\omega]=a\cdot c_1(TP)$ for some $a\geq 0$. Let $\Phi:P \rightarrow P$ be a $T$-periodic rational symplectomorphism (as in Definition~\ref{def:rational_symplectomorphism}) and $(M_\Phi,\omega,dt)$ be its mapping torus. Then 
    \[
    \g_{\sigma,T}([-\epsilon,0] \times M_\Phi)\leq \epsilon \cdot T   . 
    \]
\end{cor}

The basic approach to Theorem \ref{thm:gap0_for_periodic} in this section is to directly demonstrate the existence of $J$-holomorphic cylinders in $[-\epsilon,0] \times Y$ that have small area and pass through any given point $p$ in the completion of $[-\epsilon,0] \times Y$. For a translation invariant $J$, these curves are precisely the trivial cylinders over the period $T$ Reeb orbits. The main difficulty is proving that these cylinders persist as one deforms a translation invariant $J$ to an arbitrary almost complex structure. This requires the proof of a compactness statement for the moduli space of these cylinders. The proof of this statement occupies most of this section.

\vspace{3pt}

We now outline the rest of the section. In Section~\ref{sec:Fredholm_MorseBott}, we give the necessary background for Fredholm theory and transversality in the Morse-Bott setting. In Section~\ref{subsec:transversality_triv_cyl}, we prove that trivial cylinders with a point constraint are cut-out transversely. In Section~\ref{subsec:compactness_low_energy} we discuss the compactification of a parametric moduli spaces of cylinders of type $\sigma = (0,1,A)$. In Section~\ref{subsec:bounding_gap} we show that these parametric moduli spaces are generically smooth manifolds, and that the count of cylinders with a point constraint, for a generic $(J,p)$, is the same as the count for trivial cylinders, which is 1. This shows that the relevant spectral gap is bounded by the energy of a trivial cylinder, which is $(\f(0)-\f(-\epsilon)) T$ for generic $(J,p)$. We then use   Lemma~\ref{lem:minimizer_and_JP_semicont}.\ref{itm:axiom_semicontinuity} to conclude that this bound holds for all pairs $(J,p)$. 

\begin{setup}\label{set:gap_for_periodic} For brevity and convenience, we fix the following notation for the rest of the section.
\begin{itemize}
    \item  $(Y,\omega,\theta)$ is a conformal Hamiltonian manifold with $T$-periodic Reeb flow, satisfying the assumptions of Theorem \ref{thm:gap0_for_periodic}.
    \vspace{3pt}
    \item $X$ is the thickening $\Gamma^0_{\smallneg \epsilon}$ for a small $\epsilon>0$. That is, $(X,\Omega,Z)$ is the symplectic cobordism
    \[X:=[-\epsilon,0]\times Y,\qquad \Omega:=\omega+d(\f(r)\theta),\quad \text{and}\quad \Theta_+ := \theta,\ \Theta_-:=\f'(- \epsilon)\theta,\]
    Here $\f$ is the function described in Lemma \ref{lem:symplectic_form_on_completion} and Equation \ref{eq:conformal_function}.
    \vspace{3pt}
    \item $\hat{X}$ is the completion of $X$, or equivalently the symplectization of $Y$. 
    \vspace{3pt}
    \[
    \hat X:=(-\infty,0]\times \partial_-X\cup X \cup [0,\infty)\times \partial_+ X \cong \R\times Y, \qquad \hat \Omega=\omega+d(\f(r)\theta).  
    \]
    \item $\sigma$ is the curve type $(0,1,A)$, where $A\in H_2(X,\partial X)$ is the homology class of trivial cylinders over the $T$-periodic orbits.
\end{itemize}
\end{setup}

\subsection{Morse-Bott Fredholm Theory And Transversality}\label{sec:Fredholm_MorseBott} Here we discuss transversality results for moduli spaces of cylinders in the Morse-Bott setting. These results are common folklore, but adequate references are lacking. For the non-degenerate case, see \cite{wendlsymplectic,wendlholomorphic}.

\begin{remark} In this section, it is sufficient to assume that $Y$ has Morse-Bott Reeb flow (see Definition \ref{def:Morse_Bott_SHS}). In particular, the full set of assumptions in Setup \ref{set:gap_for_periodic} are not needed.\end{remark} 

\subsubsection{Fredholm Theory Setup} We start by defining the relevant Banach manifolds and Fredholm sections, adapting the discussion in \cite[Section 7.2]{wendlsymplectic} to the Morse-Bott setting. Fix a vector-bundle
\[E \to \Sigma \qquad\text{over the cylinder}\qquad \Sigma = \R \times \R/\Z\]
Fix a \emph{weight} $\delta > 0$, and constants $k \in \N$ and $p \in (0,\infty)$ satisfying $k \cdot p > 2$. Given these choices, there is an associated \emph{weighted Sobolev space}
\begin{equation}
    W^{k,p,\delta}(E) = \left\{ f \ | \ \mu\cdot f\in W^{k,p}(E)\right\} \qquad\text{with norm}\qquad  \|f\|_{k,p,\delta}:=\|\mu\cdot f\|_{k,p}.
\end{equation}
Here $\|-\|_{k,p}$ is the usual $W^{k,p}$-norm with respect to some arbitrary metric on $E$, and $\mu:\Sigma\rightarrow (0,\infty)$ is a weight function that is 1 away from the ends and $e^{\delta|s|}$ near the ends \footnote{The role of the weigh $\delta$ is to deal with the degeneracy of the asymptotic operators, due to the fact that Reeb orbits are unparameterized, so the parameterized orbits come in $S^1$-families. This is demonstrated in the proof of Lemma~\ref{lem:Fredholm_index} below.}. Note that the topology on this Banach space is independent of the choice of $\mu$ and the metric on $E$. 

\vspace{3pt}

We are interested in certain Banach manifolds of maps of the cylinder into $\hat{X}$. More precisely, fix a pair Morse-Bott families of closed Reeb orbits
\[S_\pm \qquad\text{with period}\qquad T\pm\]
Given choices of $\delta,k,p$ as above, there is an associated space of maps
\begin{equation*}
    \cB^{k,p,\delta}:= W^{k,p,\delta}(\Sigma, \hat X;S_+, S_-) 
\end{equation*}
consisting of maps $u:\Sigma\rightarrow\hat X$ of the form $u=\exp_f h$ where  $h\in W^{k,p,\delta}(f^* T\hat X)$ and  $f:\Sigma\rightarrow \hat X$ is a smooth function such that 
\[f(s,t) = (T_\pm s + a_\pm, \gamma_\pm(t+b_\pm))\quad \text{for }|s| \text{ sufficiently large}\]
Here $a_\pm\in \R$, $b_\pm\in S^1$ and $\gamma_\pm \in S_{\pm}$ are closed Reeb orbits in the Morse-Bott families. The space $\cB^{k,p,\delta}$ is a Banach manifold (see \cite[Section 7.2]{wendlsymplectic}) with tangent space given by 
\begin{equation*}
T_u \cB^{k,p,\delta} =W^{k,p,\delta}(u^*T\hat X)\oplus V \qquad\text{for}\qquad u \in \cB^{k,p,\delta}
\end{equation*}
Here $V \subset \Gamma(u^*T\hat X)$ is a space of sections of dimension $\dim V = 4+\dim{S_{+}}+\dim{S_{-}}$ with the following property. Let $N_\pm\subset \partial_\pm X$ be the lifts of $S_\pm$ under the Reeb flow, namely $N_\pm/S^1=S_\pm$. Choose a basis for $T_{\gamma_\pm(0)}N_{\pm}$, then we have a natural frame of $\gamma_\pm^*TN_{\pm}$ induced by the linearized Reeb flow, since $d\phi^{T_\pm}|_{TN_{\pm}}=\id$. Together with the $\R$-direction, we obtain a frame on $\gamma_\pm^*(\R\oplus TN_{\pm})$. The space $V$ consists of sections that are asymptotic, as $s\rightarrow \pm\infty$, to the elements of this frame.  The subspace $V$ accounts for the $\R$-shifts $a_\pm$, for the $S^1$-shifts $b_\pm$ and  for changing the asymptotic to Reeb orbits within the Morse-Bott families\footnote{In the non-degenerate case, the Morse-Bott families are single points and we only have the $\R$ and $S^1$ shifts $a_\pm, b_\pm$.}. There is a natural Banach bundle 
\[\cE^{k-1,p,\delta}\rightarrow \cB^{k,p,\delta} \quad \text{with fibers}\quad \cE^{k-1,p,\delta}_u = W^{k-1,p,\delta }(\overline{\on{Hom}}_\C(T\Sigma,u^*T\hat X)).\] 
For a compatible almost complex structure $J$ on $\hat X$, consider the Cauchy-Riemann operator 
\[\bar \partial_{j,J}:\cB^{k,p,\delta}\rightarrow \cE^{k-1,p,\delta},\qquad u\mapsto du+J\circ du\circ j\]
Given a choice of compatible almost complex structure $J$ on $\hat{X}$, the Cauchy-Riemann equations for maps $u:\Sigma \to \hat{X}$ can naturally be viewed as a section
\[
\bar \partial_{j,J}:\cB^{k,p,\delta}\rightarrow \cE^{k-1,p,\delta},\qquad u\mapsto du+J\circ du\circ j.
\]
The corresponding \emph{linearized operator} $D_u$ is the linearization of $\bar \partial_{j,J}$ at a given map $u \in \cB^{k,p,\delta}$. It is a Fredholm operator
\[
 D_u:W^{k,p,\delta}(u^*T\hat X)\oplus V\rightarrow W^{k-1,p,\delta }(\overline{\on{Hom}}_\C(T\Sigma,u^*T\hat X))
\]
Explicitly, if $\nabla$ is a symmetric connection on $T\hat X$, the linearized operator takes the following form.
\begin{equation}\label{eq:lin_op_connection}
    D_u (w) = \nabla w + J(u)\circ \nabla w \circ j +(\nabla_w J)\cdot du \circ j.
\end{equation}
See \cite[Section 2.1]{wendlsymplectic} for a proof of this formula. We denote the Fredholm index of $D_u$ by
\[
\on{ind}(u) \in \Z
\]
Finally, a $J$-holomorphic map is called \emph{transversely cut-out} or \emph{regular} if the linearization $D_u$ is a surjective operator.

\vspace{3pt}

\subsubsection{Index Formula} We will need the following formula for the Fredholm index of the linearized operator in the Morse-Bott case. This generalizes the proof of \cite[Lemma 7.10]{wendlsymplectic}.
\begin{lemma}\label{lem:Fredholm_index}
    Let $(Y,\omega,\theta)$ be a Morse-Bott stable Hamiltonian manifold and let $J$ be a compatible almost complex structure on $\hat X$. Fix a $J$-holomorphic cylinder
    \[u:\Sigma\rightarrow\hat X \qquad\text{with ends in the Morse-Bott families}\qquad S_\pm\]
    Finally, fix a trivialization $\tau$ of $\ker(\theta)$ along the ends $\gamma_\pm$ of $u$. Then the Fredholm index of $D_u$ is given by
    \begin{equation}
        \on{ind}(u) =  2c_1^\tau(u^* T\hat X)+  \indRS^\tau(\gamma_+)-\indRS^\tau(\gamma_-) + 2 + \frac{1}{2} \cdot (\on{dim}(S_+) + \on{dim}(S_-))
    \end{equation}
\end{lemma}
Here $\indRS^\tau(\gamma_\pm)$ is the Robbin-Salamon index of the linearized flow $d\phi^t$ restricted to $\xi:=\ker\theta$ along $\gamma_\pm$ under the trivialization $\tau$. We refer to \cite[Section 2.1.2]{chaidez2022contact} for a short exposition matching our notations. More detailed exposition are \cite{robbin1993maslov,gutt2014generalized}. The relative Chern number $c_1^\tau$ of a vector bundle over a surface with boundary is a generalization of the first Chern number $c_1$ defined for a complex bundle over a closed surface, see \cite[Definition 5.1]{wendlsymplectic}.
\begin{proof}[Proof of Lemma~\ref{lem:Fredholm_index}]
    The proof is an adaptation of that of \cite[Lemma 7.10]{wendlsymplectic}, together with other arguments and computations from that section, to the Morse-Bott case. We skip some of the details that are identical to the proof in the standard case and refer to \cite[Section 7.2]{wendlsymplectic} for a thorough exposition. 
    
    \vspace{3pt}
    Let $D_\delta$ denote the linearized operator $D_u$ restricted to the Sobolev space $W^{k,p,\delta}(u^*T\hat X)$. Then $D_\delta$ is conjugate to the operator 
\begin{equation}\label{eq:delta_conjugation}
    D_\delta': W^{k,p}(u^* T\hat X)\rightarrow W^{k-1,p }(\overline{Hom}_\C(T\Sigma,u^*T\hat X)),\quad D_\delta':= (\mu\cdot-)\circ D_\delta \circ (\frac{1}{\mu}\cdot -),
\end{equation}
where $\mu:\Sigma\rightarrow \R_{>0}$ is the weight function in the setup above. Consider the splitting $u^*T\hat X\cong \left<\partial_r, R\right>\oplus \xi$ then $D_\delta'$ is asymptotic to the perturbed asymptotic operators
\begin{equation*}
    \tilde \bA_\pm :=\Big((-J\partial_t)\oplus \bA_{\gamma_\pm}\Big)\pm\delta.
\end{equation*}
Here, $\bA_{\gamma_\pm}$ denotes the asymptotic operators corresponding to the $T_\pm$-periodic orbits $\gamma_\pm$. It is given by the formula
\[
\bA_{\gamma_\pm}(\eta):=-J(\partial_t\eta-T_\pm \nabla_\eta R)
\]
Fixing a trivialization $\tau$ of $\gamma_\pm^*\xi$, we can consider  $\bA_{\gamma_\pm}$ as operators on $W^{k,p}(\R^{2n-2})$. The indices of these operators coincide with the Robbin-Salamon indices of $\gamma_\pm$ with respect to $\tau$. Our assumption that the Reeb flow is Morse-Bott implies that the operators $\bA_{\gamma_\pm}$ are identity on $T N_{\pm}\cap \xi$ and are non-degenerate on the complement. Therefore, the index of the perturbation $\bA_{\gamma_\pm}\pm\delta$ is (see \cite[Theorem 3.36]{wendlsymplectic}):
\[
\indRS^\tau(\gamma_\pm)\mp\frac{1}{2}\dim{TN_{\pm}\cap\xi} = \indRS^\tau(\gamma_\pm) \mp\frac{1}{2}(\dim{TN_{\pm}}-1) = \indRS^\tau(\gamma_\pm) \mp \frac{1}{2}\dim{S_\pm}
\]
Finally, considering the first component $-i\partial_t\pm\delta$  of $\tilde\bA_\pm$ is non-degenerate and its index is $\mp 1$.  
Overall we conclude that
\[
\on{ind}^\tau(\tilde \bA_\pm) = \mp 1 +\indRS^\tau(\gamma_\pm) \mp \frac{1}{2} \dim{S_\pm}. 
\]
Recall the general formula for the Fredholm index (e.g. \cite[Theorem 5.4]{wendlsymplectic}),
\[
\on{ind}(D_\delta) = 2c_1^\tau(u^* T\hat X)  +\on{ind}^\tau(\tilde\bA_+)-\on{ind}^\tau(\tilde\bA_-).
\]
Plugging in the indices of $\tilde\bA_\pm$ we obtain 
\begin{align*}
    \on{ind}(D_\delta)&= 2c_1^\tau(u^* T\hat X)+ \Big(-1+ \indRS^\tau(\gamma_+)-\dim{S_+}\Big) - \Big(+1 +\indRS^\tau(\gamma_-)+\dim{S_-}\Big)\\
    &= 2c_1^\tau(u^* T\hat X)-2+  \indRS^\tau(\gamma_+)-\indRS^\tau(\gamma_-) -\frac{1}{2}(\dim{S_+} + \dim{S_-}).
\end{align*}
Adding now $V$ back to the domain of our linearized operator, the index grows by $\dim{V}=4+\dim{S_{+}}+\dim{S_{-}}$ and we get
\begin{equation*}
    \on{ind}(D_u) = 2c_1^\tau(u^* T\hat X)+  \indRS^\tau(\gamma_+)-\indRS^\tau(\gamma_-)+2+\frac{1}{2}(\dim{S_+} + \dim{S_-}).
    \qedhere
\end{equation*}
\end{proof}

\subsubsection{Transversality} Our next goal is to state the generic transversality for somewhere injective cylinders in the Morse-Bott setting. Recall the following definition.
\begin{definition}[Injectivity] A smooth map $u:\Sigma\rightarrow \hat{X}$ is called \emph{somewhere injective} if it has an injective point. An \emph{injective point} of $u$ is a point $z \in \Sigma$ that satisfies
    \[
    du(z):T_z\Sigma\rightarrow T_{u(z)} X\quad \text{ is injective and }\quad u^{-1}(u(z)) = \{z\}.
    \]
\end{definition}

For proper, finite energy $J$-holomorphic curves whose asymptotic Reeb orbits are non-degenerate or Morse-Bott,  somewhere injectivity is equivalent to being \emph{simple}, i.e. not multiply covered. More generally, every proper, finite energy, non-constant $J$-holomorphic curve $u:\Sigma \to \hat{X}$ with connected domain factorizes as
\[
u:\Sigma \xrightarrow{\varphi} \Sigma'\xrightarrow{v} \hat{X}
\]
Here $\varphi$ is a holomorphic map of positive degree between closed and connected Riemann surfaces, and $v$ is a $J$-holomorphic curve which is embedded except at a finite set of critical points and self-intersections. See, for example, \cite[Theorem 6.19]{wendlsymplectic} for the non-degenerate case and \cite{siefring2008relative}  for the Morse-Bott setting.   

\vspace{3pt}

We will consider the following moduli spaces in our discussion of transversality.
\begin{definition}
    Let $S_\pm$ be Morse-Bott families. Given $J\in \cJ(X)$, the moduli space of cylinders with 1 marked point, denoted 
    \[\modulitrans,\]
    is the moduli space of equivalence classes of pairs $(u,z)$ consisting of the following data.
\begin{itemize}
    \item $z\in \Sigma = \R\times S^1$ is a point.
    \vspace{3pt}
    \item $u$ is a proper, finite energy $J$-holomorphic cylinder, i.e., $u:\Sigma \to \hat{X}$, with ends in $S_\pm$.
\end{itemize}
The equivalence is given by $(u,z)\sim(u',z')$ if there is a biholomorphism $\varphi:\Sigma\rightarrow \Sigma$ that fixes the ends such that $u'\circ \varphi = u$ and $\varphi(z) = z'$. 
\end{definition}
\begin{remark}
    Note that $\varphi$ is equivalent to a biholomorphism of $S^2$ that fixes 2 points and sends $z$ to $z'$. Since such a biholomorphism exists and is unique, there is a bijection between classes and maps:
    \[
    [(u,z)]\quad \iff\quad u':\Sigma\rightarrow \hat X,\ \text{ such that }\  u'\circ \varphi =u,\ \varphi(z)=(0,0)\in \Sigma.
    \]
    Therefore, the space  $\modulitrans$ is naturally identified with the space of parameterized cylinders (namely, maps $u:\Sigma\rightarrow \hat X$ rather than equivalent classes) with no marked points. In what follows we will identify classes $[(u,z)]$ with the corresponding map. In particular, we will use both of these descriptions of $\modulitrans$.  
\end{remark}
Now let us state the generic transversality for cylinders in the Morse-Bott setting. The non-degenerate versions of these statements appear in \cite[Theorem 7.2 and Remark 7.4]{wendlsymplectic}.

\begin{prop}[Transversality] \label{prop:transversality_simple_cuves}
    Fix $a >0$ and let $\cJ_\rho(X)\subset \cJ(X)$ be the subset of almost complex structures that are cylindrical outside of $X^{\rho}_{-\rho}$. Let $S_\pm$ be Morse-Bott families. There exists a comeager subset 
    \begin{equation*}
        \cJ^{\on{reg}}_\rho\subset \cJ_\rho( X),
    \end{equation*}
     such that for every $J\in \cJ^{\on{reg}}_\rho$, every $u$ in $ \modulitrans$ that has an injective point is regular. In particular the simple curves form an open subset 
     \[\modulireg\ \subset\ \modulitrans\] that is a smooth manifold of dimension $\on{ind}(u)$.
     Finally, the evaluation map 
     \[\on{ev}:\modulireg\rightarrow \hat{X}\quad \text{given by} \quad [(u,z)]\mapsto u(z)\] is smooth.
\end{prop}
\begin{prop}[Parametric Transversality] \label{prop:param_transversality_simple_cuves}
     Fix $\rho>0$ and let $\cJ_\rho(X)\subset \cJ(X)$ be the subset of almost complex structures that are cylindrical outside of $X_{-\rho}^\rho$. Let $S_\pm$ be Morse-Bott families. For a smooth path $\bJ:[0,1]\rightarrow \cJ_\rho(X)$, $\tau\mapsto J_\tau$ consider the parametric moduli space 
    \begin{equation*}
        \moduliparam (X;\bJ, S_+,S_-):=\Big\{(\tau,u)\ :\ \tau\in[0,1], u\in \cM(X;J_{\tau},S_+,S_-)\Big\}.  
    \end{equation*}
    Then, for a generic path $\bJ$,
    the set of pairs $(\tau,u)\in \moduliparam (X;\bJ,S_+,S_-)$ such that $u$ has an injective point
    \[\moduliparam^{\on{reg}}(X;\bJ,S_+,S_-)\ \subset\  \moduliparam(X;\bJ,S_+,S_-)\] is open and is a smooth manifold of dimension $\on{ind}(u)+1$.
    Finally, the evaluation map 
    \[\on{ev}:\moduliparam^{\on{reg}}(X;\bJ,S_+,S_-)\rightarrow \hat X\quad \text{given by} \quad  [(u,z)]\mapsto u(z)\] is smooth.
\end{prop}
With Lemma~\ref{lem:Fredholm_index} at hand, the proofs of Propositions~\ref{prop:transversality_simple_cuves} and \ref{prop:param_transversality_simple_cuves} are identical to the proofs in the non-degenerate case. In other words, the only difference between the non-degenerate and the Morse-Bott cases is the computation of the Fredholm index of the linearized operator.  For the convenience of the reader we give here a summary of the proof of Proposition~\ref{prop:transversality_simple_cuves}, following \cite[Section 7]{wendlsymplectic} and \cite[Section 4.3]{wendlholomorphic}\footnote{Note that Wendl's notation for $\cJ_\rho(X)$ is $\cJ(\omega_\psi,r_0,\mathcal{H}_+,\mathcal{H}_-)$ where $\omega_\psi =\Omega$, $r_0=\rho$ and $\mathcal{H}_\pm = \partial_\pm X$ as stable Hamiltonian manifolds.}. 
\begin{proof}[Outline of proof of Proposition~\ref{prop:transversality_simple_cuves}]
    The claim follows from the Sard-Smale theorem applied to the projection 
    \[
    \pi:\cM^*\rightarrow\cJ_\rho(X), \qquad (u,J)\mapsto J,
    \]
    where $\cM^*$ is the universal moduli space given by
    \[
    \cM^*:=\Big\{(u,J):\ J\in \cJ_\rho(X),\ (u,z)\in \modulitrans \text{ and $u$ has an injective point} \Big\}.
    \]
    In order to apply the Sard-Smale theorem we need to show that $\cM^*$ is a smooth Banach manifold and that the projection $\pi$ is a smooth Fredholm map.  To show that $\cM^*$ is a smooth Banach manifold, consider it as the zero set of the smooth section
    \[
    \bar\partial:\cB^{k,p,\delta}\times \cJ_\rho(X) \rightarrow \cE^{k-1,p,\delta},\qquad (u,J)\mapsto du+J\circ du\circ j.
    \]
    When $u$ is a simple curve one can show that the linearization of this section is surjective, and thus $\cM^*$ is a Banach manifold. The fact that $\pi$ is smooth now follows from viewing it as a restriction of a smooth map $\cB^{k,p,\delta}\times \cJ_\rho(X) \rightarrow \cJ_\rho(X)$ to a submanifold. The Fredholm property of $\pi$ follows from  the fact that the linearized operator $D_u$ is Fredholm. Applying the Sard-Smale theorem, we conclude that there exists a comeager set $\cJ^{\on{reg}}_\rho\subset \cJ_\rho(X)$ of regular values of $\pi$. For each such regular value $J$, the space $\modulireg$ is a smooth manifold.

    The fact that the evaluation map is smooth on $\modulireg$ follows from the fact that the latter is a smooth submanifold of $\cB^{k,p,\delta}$ and the evaluation map is smooth on $\cB^{k,p,\delta}$ (essentially since the exponential map on $X$ is smooth and $W^{k,p}\subset C^0$). 
\end{proof}

\subsection{Transversality Of Trivial Cylinders} \label{subsec:transversality_triv_cyl}
Let $(Y,\omega,\theta)$ is a stable Hamiltonian manifold with $T$-periodic Reeb flow and let $X$ be the trivial cobordism over $Y$, as in Setup~\ref{set:gap_for_periodic}. Our main goal for this section is to show that  trivial cylinders in $\hat X$ over $T$-periodic orbits are transversely cut-out. 

\vspace{3pt}

Let us begin by presenting the class of almost complex structures that we consider, and define trivial cylinders. Recall that the completion
\[\hat X :=(-\infty, 0]\times Y_-\cup_\iota X \cup_\iota [0,\infty)\times Y_+\] 
is canonically identified with $\R \times Y$. In particular, there is a distinguished class of \emph{translation invariant}, compatible almost complex structures
\[\cJ_{\on{tr}}(X)\subsetneq \cJ(X).\]
This is none other than the set $\mathcal{J}(Y)$ if compatible almost complex structures on $Y$ (see Definition \ref{def:acs_symplectization}). For any $J \in \mathcal{J}_{\on{tr}}(X)$ and any periodic orbit $\gamma$ of period $L$, the map
    \[u_\gamma:\R\times S^1\rightarrow \hat X \cong \R\times Y, \quad (s,t)\mapsto (L\cdot s, \gamma(L\cdot t))\]
is $J$-holomorphic. We refer to $u_\gamma$ as the \emph{trivial cylinder over }$\gamma$. The main result of this subsection can now be stated as follows.

\begin{lemma}\label{lem:transversality_triv_cyl}
For any translation invariant $J\in \cJ_{\on{tr}}(X)$ and any $T$-periodic orbit $\gamma$ of the Reeb flow on $Y$, the holomorphic cylinder
$$
u_\gamma:\R\times S^1\rightarrow \hat X \cong \R\times Y, \quad (s,t)\mapsto (T\cdot s, \gamma(T\cdot t))
$$
is cut out transversely. Moreover, the kernel of the linearized operator is of dimension $2n$.
\end{lemma}

\noindent We will need the following lemma, which is a standard  computation (cf. \cite[Lemma 8.8.5]{audin2014morse}).
\begin{lemma}[{\cite[Lemma 8.8.5]{audin2014morse}}]\label{lem:AD_kernel_dim}
Fix a function $a:\R\rightarrow\R$ such that $a(s)=a_-$ for $s$ near $-\infty$ and $a(s)=a_+$ for $s$ near $+\infty$, and consider the operator
\begin{equation*}
    F:W^{k,p}(\R\times S^1;\R^2)\rightarrow W^{k-1,p}(\R\times S^1;\R^2),
    \qquad F= \partial_s +i\partial_t + \left(\begin{array}{cc}
       a(s)  &  0\\
       0  & a(s)
    \end{array}\right) 
\end{equation*}
 Then the dimension of the kernel of $F$ satisfies
\begin{equation*}
    \dim{\ker F} =2\#\{\ell\in \Z: a_-<2\pi \ell< a_+\}. 
\end{equation*}
\end{lemma}

\begin{proof}[Lemma \ref{lem:transversality_triv_cyl}]
    Let us start by writing the index of $D_{u_\gamma}$. Consider a trivialization $\tau$ of $u_\gamma^*T\hat X$ that is induced by the linearized Reeb flow, namely $\xi_{\gamma(t)}\mapsto d\phi^t\xi_{\gamma(0)}$ for $t\in[0,T]$, and $\partial_r,R$ determine the other two directions. Note that this is possible since the Reeb flow is $T$-periodic and thus $d\phi^T=\id$.  Under this trivialization, $c_1^\tau(u_\gamma)=0$ and $\indRS^\tau(\gamma_+) = \indRS^\tau(\gamma_-)$, where $\gamma_\pm$ are the copies of $\gamma$ at the ends of the trivial cylinder.
    By Lemma~\ref{lem:Fredholm_index},
    \[\on{ind}(D_{u_\gamma}) = 2+\frac{1}{2}(2n-2+2n-2) = 2n.\] 
    Therefore, it is enough to show that the kernel of $D_{u_\gamma}$ is of dimension $2n$, since this will imply that $\dim{\on{coker}(D_{u_\gamma})} = \dim{\ker{D_{u_\gamma}}} - \on{ind}(D_{u_\gamma}) =2n-2n=0$. 
    As in some of the previous proofs, we rely on computations from \cite[Section 7.2]{wendlsymplectic} and adjust them to our setting. Consider the direct sum splitting $u_\gamma^*T\hat X =T\Sigma\oplus \gamma^*\xi$, it is shown in \cite[Section 7.2]{wendlsymplectic} that  the linearized operator takes the form:
    \begin{equation*}
        (D_{u_\gamma}\eta)(\partial_s) = \left(\partial_s-\left(\begin{array}{cc}
            -i\partial_t  & 0 \\
            0 & \bA_\gamma
        \end{array}\right)\right)\eta,
    \end{equation*}
    where $\bA_\gamma$ is the asymptotic operator of $\gamma$. 
    After applying the conjugation (\ref{eq:delta_conjugation}), and writing $\mu(s):=e^{\beta(s)}$ for convenience, the operators $D_{u_\gamma}$ on $W_{k,p,\delta}$ is conjugate to the operator
    \begin{equation}
        D':=   \left(\partial_s-\left(\begin{array}{cc}
            -i\partial_t  & 0 \\
            0 & \bA_\gamma
        \end{array}\right)\right)-\beta'(s),
    \end{equation}
    where $\beta'(s)=\delta$ when $s$ is close to $+\infty$ and $\beta'(s)=-\delta$ when $s$ is close to $-\infty$. Under our previously chosen trivialization $\tau$ the asymptotic operator $\bA_\gamma$ is simply $-i\partial_t$ (e.g. \cite[Remark 3.24]{wendlsymplectic}), and therefore the kernel of $D_{u_\gamma}$ is in bijection with that of 
    \[F:=\partial_s+i\partial_t-\beta'(s) \quad\text{on}\quad W^{k,p}(\R\times S^1;\R^{2n}).
    \]
    Note that here we used the trivialization $\tau$ to identify the fibers of the bundle $u_\gamma^*T\hat X$ with $\R^{2n}$. The  operator $F$ is diagonal with free term depending only on $s$. Splitting it into diagonal operators over $W^{k,p}(\R\times S^1;\R^2)$, we may apply Lemma~\ref{lem:AD_kernel_dim} and sum the dimensions of the kernels we get. 
    We conclude that dimension of the kernel is given by:
    \begin{equation*}
        \sum_{i=1}^n 2\cdot \#\{\ell\in \Z: \beta'(-\infty)<2\pi\ell<\beta'(+\infty)\}=
        2n\cdot \#\{\ell\in \Z: -\delta<2\pi\ell<\delta\} = 2n,
    \end{equation*}
    assuming that $\delta$ is small.
\end{proof}

Our next and final goal for this section is to prove the surjectivity of the linearized evaluation map, in order to show that certain trivial cylinders with a point constraint are cut out transversely. 
Let $P:=\{p\}$ where $p\in \on{int}(X)$ is any point. There exists a unique $T$-periodic orbit $\gamma$ that passes through $p$. Consider the evaluation map 
\[\on{ev}:\moduli\rightarrow \hat X,\quad (u,z)\mapsto u(z),\]
and denote  $\moduliP:=\on{ev}^{-1}(p)$. The linearized evaluation map at $(u_\gamma,z)$ is 
\begin{equation*}
    D_{\on{ev}}:\ker{D_{u_\gamma}}\oplus T_z\Sigma\rightarrow T_p\hat X.
\end{equation*}
We will show that it is surjective over $T$-periodic orbits.
\begin{prop}\label{prop:transversality_triv_cyl_w_pt}
    For any translation invariant $J\in \cJ_{\on{tr}}(X)$ and any $T$-periodic orbit $\gamma$ of the Reeb flow on $Y$, the operator
    \[D_{\on{ev}}:\ker{D_{u_\gamma}}\oplus T_z\Sigma\rightarrow T_p\hat X\]
    is surjective. Thus, every trivial cylinder over a $T$-periodic orbit in $\moduliP$ is cut out transversely.
\end{prop}
\begin{proof}
    We construct $2n$ linearly independent sections of $W^{k,p,\delta}(u^*T\hat X)$ that lie in the kernel of $D_{u_\gamma}$, and whose images under $D_{\on{ev}}$ span $T\hat X$. 
    
    \vspace{3pt}
    
    Our first step is to notice that $\partial_r$ and $R$ lie in this kernel, since a shift of $u_\gamma$ in the $\R$ direction, as well as $\phi^t(u_\gamma)$, are $J$-holomorphic curves. This is due to the fact that $\phi^t(u_\gamma)$ is a trivial cylinder over an orbit obtained from $\gamma$ by a time shift. Let us now find the other $2n-2$ sections.
    Let
    \[v_1,\dots,v_{2n-2}\in T_{\gamma(0)}N_T\cap \xi_{\gamma(0)}\]
    be a basis. For each $i$, let $\rho_i:(-\varepsilon,\varepsilon)\rightarrow N_T$ be a smooth path, $x\mapsto \rho_i(x)$, such that 
    \begin{equation*}
        \frac{\partial \rho_i}{\partial x}\Big|_{x=0} = v_i.
    \end{equation*}
    Since $\rho_i(x)\in N_T$, each point in these paths lies in a $T$-periodic orbit, $\gamma_i^x(t) :=\phi^t\rho_i(x)$. The vector fields
    \begin{equation*}
        v_i(t):= \frac{\partial \gamma_i^x(t)}{\partial x}\Big|_{x=0}  = d\phi^t(v_i)
    \end{equation*} 
    form a frame for $\gamma^*(TN_T\cap\xi)$. Since the trivial cylinders over $\gamma_i^x$ are $J$-holomorphic, the vector fields $v_i$ lie in the kernel of $D_{u_\gamma}$.
    Clearly, these vector fields together with $\partial_r$ and $R$ span $u^*T\hat X$. Therefore we conclude that $D_{\on{ev}}$ is surjective.
\end{proof}

\subsection{Compactness For Low Energy Curves} \label{subsec:compactness_low_energy} Our next goal for this section is the following compactnes statement.
\begin{prop}[Compactness]\label{prop:low_energy_compactness} Let $X$ be as in Setup \ref{set:gap_for_periodic} and let $J_k \rightarrow  J\in\cJ(X)$ in a $C^\infty$-cylindrical manner (see Definition \ref{def:acs_cob}). Let $v_n$ be $J_n$-holomorphic cylinders with ends $\Gamma^\pm\in S_T$, such that    \begin{equation}\label{eq:energy_and_period_bnd_compactness}
       [v_n]=A \qquad \text{is the class of the trivial cylinder.}
   \end{equation}
    There is a subsequence of $v_n$ converging {$C^\infty_{loc}$} to a $J$-holomorphic cylinder $u$ with ends in $S_T$ such that 
    \[
    [u]=A\qquad \text{ or }\qquad c_1([u]\# \bar A)<0.
    \]
\end{prop}
\begin{remark}
    The condition  (\ref{eq:energy_and_period_bnd_compactness}) implies that  
    \[E_\Omega(v_n)=\big(\f(0)-\f(-\epsilon)\big) T\qquad \text{since}\qquad \omega(A)=0.  
    \]
\end{remark}

The proof of Proposition~\ref{prop:low_energy_compactness} requires the following lemmas. 
\begin{lemma}\label{lem:cylindrical_building} Let $X$ be as in Setup \ref{set:gap_for_periodic} and let $\mathbf{u}$ be a genus zero holomorphic building in $\hat X$ whose ends $\Gamma^\pm$  are single orbits. Then $\mathbf{u}$ consists only of cylinders with sphere bubbles. 
\end{lemma}
\begin{proof}
     Suppose that $\mathbf{u} = (u_1,\dots ,u_\ell)$ and denote by $\Gamma_i^\pm$ the ends of $u_i$, then $\Gamma_i^+ = \Gamma_{i-1}^-$. We start by noticing that such a building $\mathbf{u}$ must be connected (up to sphere bubbles). Indeed, suppose $\mathbf{u}$ had more than one connected component, then since it has only one orbit as its negative end, there would be at least one connected component with no negative ends. Since the building is of genus 0, so is this component. Moreover, it has at most one boundary component, and thus it is either a disk or a sphere. Since there are no contractible orbits, it has to be a single holomorphic sphere.
     \vspace{0.1cm}     
     
     Assume for the sake of contradiction that not all levels of $\mathbf{u}$ are cylinders, and fix $i$ to be  the first level that is not a cylinder. Then $u_{i-1}$ is a cylinder and thus
     \[\Gamma_i^+=\Gamma^+\ \text{ if }\ i=1,\quad \text{ or }\quad  \Gamma_i^+ = \Gamma_{i-1}^-\  \text{ if }\  i>1.\]
     In any case, $\Gamma_i^+$ consists of a single orbit. Note that if the level $u_i$ is disconnected, only one of its connected components contains $\Gamma_i^+$. In this case we restrict to this connected component and, if it is a cylinder, continue in the sequence until this connected component is not a cylinder. Note that if all such connected components are cylinders, then there are no other connected components (up to spheres), since $\mathbf{u}$ is connected, and we are done. Therefore, assume that $u_i$ is the first non-cylindrical connected component, then 
     \[\#\Gamma_i^+=1, \qquad \#\Gamma_i^-\geq 2.\] Fix $\gamma^1\neq\gamma^2\in \Gamma_i^-$ and let us define inductively two sequences $\{v^j_k\}$, $j=1,2$, of connected components of $\mathbf{u}$: \\
     Let $v^j_1$ be the connected component of $\mathbf{u}$ for which $\gamma^j$ is a positive end. Note that $v^j_1$ have at least two boundary components, since no orbits are contractible and the genus is zero. Let
     \[v^j_{k+1}\notin\{v^1_1,\dots, v^1_k, v^2_1,\dots, v^2_k\}\]
     be a connected component of $\mathbf{u}$ that has an end coinciding with one of the  ends of $v^j_k$ (if there is more than one such component, choose one).  This process must stop at some finite point since $\mathbf{u}$ has finitely many connected components. If we arrive at a component that has two ends in $\cup_{i,j}\ v_i^j$, this contradicts the fact that $\mathbf{u}$ is genus zero. Hence the only way this process can stop is by reaching the bottom level of $\mathbf{u}$ and having $\Gamma_-$ as an end. This is clearly only possible for one of the sequences and not for both. 
\end{proof}

\begin{lemma}\label{lem:cylinders_have_same_ends}
     Let $X$ be as in Setup \ref{set:gap_for_periodic}. Then any two Reeb orbits connected by a cylinder in $\hat X$ must have the same $\theta$-period.
\end{lemma}
\begin{proof}
    Denote by $a\in \pi_1(Y)$ the homotopy class of orbits of period $T$, where $T$ is the minimal period of the Reeb flow on $Y$. Then, by assumption, $a\neq 0$. Let $u$ be a cylinder connecting two Reeb orbits $\gamma^\pm$. There exists $q_\pm, p_\pm\in \N$ such that 
    \[
    \cL(\gamma^\pm;\theta) = (q_\pm/p_\pm) \cdot T
    \]
    The $p_+p_-$-fold cover of $u$ connects Reeb orbits of period $q_\pm T$, and thus of homotopy class $q_\pm p_\mp\cdot a$. Since they are connected by a cylinder, \[q_+ p_- a = q_- p_+ a\qquad\text{and therefore}\qquad (q_+ p_- - q_- p_+)a=0\]
    This implies that $N:=|q_+ p_- - q_-p_+|\in \N$ must be zero. Otherwise, the Reeb orbit which is the $N$-the iterate of a $T$-periodic orbit would be contractible, in contradiction to our assumption. Therefore, $q_+ p_-=q_- p_+$ and we conclude that $\gamma_\pm$ are of the same period.    
\end{proof}

Having the above two lemmas we are ready to prove the convergence stated in Proposition~\ref{prop:low_energy_compactness}.
\begin{proof}[Proof of Proposition~\ref{prop:low_energy_compactness}]
    Let $v_n$ be a sequence of $J_n$-holomorphic cylinders with ends of period $T$ and $\Omega$-energy equal to $\big(\f(0)-\f(-\epsilon)\big) T$. Recall that we assume that $J_n$ converge to $J\in \cJ(X)$ in the $C^\infty$ cylindrical topology.
    
    \vspace{3pt}
    
    We now wish to apply the SFT compactness stated in Theorem~\ref{thm:SFT_compactness} to the sequence $v_n$. The genus and $\Omega$-energy of $v_n$ are bounded by assumption, and hence the SFT compactness theorem guarantees that the sequence $v_n$ has a subsequence converging to a holomorphic building $\mathbf{u}$.  Since there are no contractible orbits, Lemma~\ref{lem:cylindrical_building} guarantees that $\mathbf{u}$ consists only of cylinders, possibly with sphere bubbles. Let us write $\mathbf{u}=(u_1,\dots,u_k)$ where each $u_i$ is a  cylinder, possibly with sphere bubbles. By definition of $\mathbf{u}$, there exists $1\leq j \leq k$ such that
    \begin{align*}
        u_j&\quad \text{is}\quad J\text{-holomorphic in the completion of }X , \\
        u_1,\dots, u_{j-1} &\quad\text{are}\quad J_+\text{-holomorphic in the symplectization of }\partial_+X,
    \\
        u_{j+1},\dots, u_k &\quad\text{are}\quad J_-\text{-holomorphic in the symplectization of }\partial_-X.
    \end{align*} 
    Note that by Lemma~\ref{lem:cylinders_have_same_ends}, all of the ends $\gamma_i^\pm$ are in the Morse-Bott family $S_T$. 
    Assume for the sake of contradiction that $[u_j]\neq A$, otherwise we are done. Since $[\mathbf{u}]=A$, there must be symplectization levels with positive area relative homology class. Denote by $B_i$ the relative homology class of $u_i$, then $\omega(B_i)>0$ for all $i\neq j$. Since $A=\sum_i B_i$ and  $\omega(A)=0$, we conclude that $\omega(B_j\#\bar A) = \omega(B_j)<0$. Due to our assumption that $(Y,\omega,\theta)$ is positively monotone, this guarantees that $c_1([u_j]\# \bar A) <0.$
\end{proof}

\subsection{Counting Cylinders}\label{subsec:bounding_gap} We now prove a curve counting result for certain low area cylinders.

\vspace{3pt}

To be more precise, let $Y,X$ and $\sigma$ be the conformal cobordism and curve type in Setup \ref{set:gap_for_periodic}. Choose an almost complex structure $J \in \mathcal{J}(X)$ and a point $P \in \hat{X}$, and let
\[
\moduliLEJ\subset  \mathcal{M}_\sigma^T(X;J,P)
\]
be the moduli space of $J$-holomorphic cylinders in the class of the trivial cylinder, with ends in $S_T$ and passing through the point $P$. We will prove the following.

\begin{prop}[Cylinder Count] \label{prop:vir_count_triv_cyl} Let $J\in \cJ^{\on{reg}}(X)$ be a translation invariant almost complex structure and let $ P\in \hat  X$ be a regular value of the evaluation map
\[
\on{ev}:\cM^{\on{reg}}(X;J,S_+,S_-) \to \hat{X}
\]
Then $\moduliLEJ$ is a $0$-dimensiona, compact space and the number of points satisfies
\[\#\moduliLEJ =  1\mod 2\]
\end{prop}

\noindent The proof of Proposition \ref{prop:vir_count_triv_cyl} will occupy the remainder of this section. The strategy will be as follows. Given a path
\[\tau \mapsto (J_\tau, P_\tau) \quad \text{in}\quad C^\infty\big([0,1], \cJ_\rho( X)\times \hat X\big)\]
we will consider the parametric moduli space 
\[
\moduliparam:= \{(\tau, u): u\in \moduliLEJO{\tau}{\tau}\},
\]
We will show that for a generic path $(J_\tau,P_\tau)$, the parametric moduli space is a compact manifold with boundary. Taking a path between a translation invariant almost complex structure and any generic pair of $(J,P)$, we will conclude that the corresponding moduli spaces are cobordant. Finally, we will show that the count of such cylinders for a translation invariant $J$ does not vanish, and conclude that the same holds for a generic pair $(J,p)$.

\subsubsection{Mean Index Bounds} In order to proceed with our proof of Proposition \ref{prop:vir_count_triv_cyl}, we require a brief digression into the relationship between the Robbin-Salamon index and the mean index. 

\begin{definition}
    The \emph{mean index} $\hat{\mu}(\Phi)$ of a path $\Phi:[0,L]\rightarrow\on{Sp}(2m)$ is the homogenezation of the Robbin-Salamon index. That is, $\hat{\mu}$ is given by
    \begin{equation}
        \hat\mu(\Phi):= \lim_{\ell\rightarrow\infty} \frac{\indRS(\Phi^\ell)}{\ell}.  
    \end{equation}
\end{definition}
The difference between the Robbin-Salamon index and the mean index is generally bounded by half the dimension (cf \cite{ginzburg2020lusternik}). More precisely
\begin{equation}\label{eq:RS_vs_mean_GG}
    |\indRS(\Phi) - \hat\mu(\Phi)|\leq m.
\end{equation}
 In the next lemma, we show that the difference is strictly smaller if $(\Phi(L))^k=\id$ for some $k\in \N$.     
\begin{lemma}\label{lem:RS_vs_mean_index}
    Suppose $\Phi:[0,L]\rightarrow\on{Sp}(2m)$ is a path of symplectic matrices, such that $(\Phi(L))^k=\id$ for some $k>1$ and $\Phi(L)\neq \id$. Then, $|\indRS(\Phi) - \hat\mu(\Phi)|< m-\frac{1}{2}\dim{\ker(\Phi(L)-\id)}$.
\end{lemma}
\begin{proof}
    For convenience, let us first renormalize $t$ so that $\Phi$ will be a path on the unit interval, $\Phi:[0,1]\rightarrow\on{Sp}(2m)$.
    We note with the observation that if $\Phi(1)=\id$, then the RS index is homogeneous and thus coincides with the mean index. Assume from now on that $\Phi(1)\neq \id$. Up to a conjugation by a symplectic matrix, $\Phi(1)\in U(m)$. Indeed, 
    \[\Psi:=\frac{1}{k}\sum_{i=1}^k (\Phi(1))^i, \quad \text{then} \quad \Psi^*(\Phi(1))^*\Phi(1)\Psi = \Psi^*\Psi.\]
    So after changing bases via $\Psi$, the matrix $\Phi(1)$ becomes  unitary. Recalling that the inclusion $U(m)\hookrightarrow \on{Sp}(2m)$ is a homotopy equivalence, we can homotope the path $\Phi$, fixing the ends, to a path in $U(m)$. By the conjugation and homotopy invariance of the RS-index, the index stays the same under these changes (as well as the mean index). Therefore assume now that $\Phi:[0,L]\rightarrow U(m)$. Diagonalizing this path of unitary matrices, we end up with a direct sum of elements in $U(1)$. That is, up to conjugation and homotopy with fixed ends,  $\Phi$ is equivalent to $\oplus_{j=1}^m e^{2\pi i \theta_j t}$ for some $\theta_j\in \R$.
    By the direct sum property, both indices split into a direct sum:
    \begin{align*}
        \indRS(\Phi) = \sum_{j=1}^m \indRS(\{e^{2\pi i \theta_j t}\}_{[0,1]}), \qquad \hat\mu (\Phi) =  \sum_{j=1}^m \hat \mu(\{e^{2\pi i \theta_j t}\}_{[0,1]}).
    \end{align*}
    Note that, since $\Phi(1)^k=\id$, $\theta_j\in \frac{1}{k}\Z$. 
    The RS and mean indices of a 2-dimensional rotation are
    \begin{equation}\label{eq:rotation_RS_and_mean}
        \indRS(\{e^{2\pi i \theta t}\}_{[0,1]})= \lfloor\theta\rfloor+\lceil \theta\rceil, \qquad \hat \mu(\{e^{2\pi i \theta t}\}_{[0,1]})= 2\theta,
    \end{equation}        
    where the second equality follows from the fact that the mean index of a path $\Phi$ with a $k$-periodic endpoint is always equal to $\indRS(\Phi^k)/k$. It is now clear that for each $j=1,\dots, m$, the difference of the RS and mean indices is at most 1. Clearly, whenever $\theta_j$ is an integer, the difference between the RS and mean indices is zero. The number of $j\in\{1,\dots,m\}$ such that $\theta_j\in \Z$ is $\frac{1}{2}\dim{\ker{\Phi(1)-\id}}=:d$. Note that $d<m$ by our assumption that $\Phi(1)\neq \id$. In particular, there exists a $j$ such that $\theta_j\notin\Z$. Fix such $j$ and write $\theta_j= \ell +r/k$ for $\ell\in \Z$ and $r\in \{1,\dots, k-1\}$. Then, 
    \[
        |\hat \mu(\{e^{2\pi i \theta_j t}\}_{[0,1]})-\indRS(\{e^{2\pi i \theta_j t}\}_{[0,1]})|= |2\theta_j - \lfloor\theta_j\rfloor-\lceil \theta_j\rceil| = |2\ell+2r/k-\ell-\ell-1| = |2r/k-1|.
    \]
    Notice that, in this case, the difference cannot be equal to 1. Indeed, if $2r/k-1 = 1$, then $r=k$, in contradiction to $r\in \{1,\dots, k-1\}$. Similarly, if $2r/k-1 = -1$ then $r=0$ and we get a contradiction again. Overall we conclude
    \begin{align*}
        |\indRS(\Phi) - \hat \mu(\Phi)| &\leq \sum_{j=1}^m|\indRS (\{e^{2\pi i \theta t}\}_{[0,1]})-\hat\mu(\{e^{2\pi i \theta t}\}_{[0,1]})|\\
        &= \sum_{\theta_j\in\Z}|\indRS (\{e^{2\pi i \theta t}\}_{[0,1]})-\hat\mu(\{e^{2\pi i \theta t}\}_{[0,1]})| + \sum_{\theta_j\notin \Z}|\indRS (\{e^{2\pi i \theta t}\}_{[0,1]})-\hat\mu(\{e^{2\pi i \theta t}\}_{[0,1]})| \\
        &= 0+ \sum_{\theta_j\notin\Z}|\indRS (\{e^{2\pi i \theta t}\}_{[0,1]})-\hat\mu(\{e^{2\pi i \theta t}\}_{[0,1]})| \\
        &<1\cdot\#\{j:\theta_j\notin \Z\} = m-d.\qedhere
    \end{align*}
\end{proof}

\subsubsection{Index Bound} We next use the property of the Robbin-Salamon in Lemma \ref{lem:RS_vs_mean_index} to acquire a general index bound in our setting.
\begin{lemma}[Index Bound] \label{lem:index_multiply_covered}
    Let $u$ be a $J$-holomorphic cylinder with ends in $S_T$. Then $u= v\circ \varphi$ where  $\varphi:\Sigma\rightarrow\Sigma$ is a degree $k\geq 1$ holomorphic map and $v:\Sigma\rightarrow X$ is a simple holomorphic cylinder.  Assume, in addition, that at least one of the following holds:
    \begin{enumerate}
        \item $c_1([u]\#\bar A)< 0$ where $A$ is the relative class of the trivial cylinder.
        \item $c_1([u]\#\bar A)\leq 0$ and $k>1$, i.e., $u$  is not  a simple curve.
    \end{enumerate}
    Then,   
    \begin{equation}
        \on{ind}(v) \leq 2n- 2.
    \end{equation}
\end{lemma}
\begin{proof}
    The fact that $u$ factorizes through a simple curve is standard, see, e.g. \cite[Theorem 6.19]{wendlsymplectic}. Note that if $u$ is already simple then $k=1$ and $\varphi=\id$. Write $\varphi:\Sigma\rightarrow \Sigma'$ and let us show that $\Sigma'$ is also a cylinder. To see this, compactify $\varphi$ to a branched cover $S^2 \to \bar{\Sigma}'$ where $\bar{\Sigma}'$ is a closed surface given by compactifying $\Sigma'$ along the punctures. The Riemann-Hurwitz formula,
    \[
    2=\chi(S^2) = k\cdot \chi(\bar{\Sigma}') - \sum_p (e_p-1)  
    \]
    implies that $ \chi(\bar{\Sigma}')=2$, the number of ramification points in $\bar\Sigma'$ is 2 and the ramification index of each one is $e_P=k$. Over all $\Sigma'=\Sigma$ is a cylinder, and the map $\varphi:\Sigma\rightarrow\Sigma$ is a (non-branched) $k$-covering. 

    \vspace{3pt}

    It remains to estimate the index of $v$. Given the index formula in Lemma~\ref{lem:Fredholm_index}, it is enough to bound the difference of Robbin-Salamon indices of the ends $\gamma_\pm$ of $v$. Let $\tau$ be a trivialization of $v^*\xi$. 
    Notice that $\gamma_\pm$ are $T/k$-periodic and denote by $\Phi_\pm:[0,1]\rightarrow \on{Sp}(2n-2)$ their linearized flows on $\gamma_\pm^*\xi$ under the trivialization $\tau$, with the time interval renormalized to be $[0,1]$ instead of $[0,T/k]$. Then, 
    \[
    \Phi_\pm(1)\neq \id, \quad \Phi_\pm(1)^k=\id\quad \text{ and } \quad \hat\mu(\Phi_\pm) = \indRS(\Phi^k)/k = \indRS^{\tau'}(\gamma_\pm^k)/k,
    \]
    where $\tau'$ is a trivialization of $u^*\xi$ induced by $\tau$. Since $\gamma_\pm^k$ lie in the Morse-Bott family $S_T$ which covers all of $Y$, their indices with respect to a trivialization over the trivial cylinder coincide. Since $\tau'$ is defined on $u$, which is not necessarily of the relative homology class of the trivial cylinder, this difference is corrected by the first Chern number of the class $[u]\#\bar A\in H_2(X)$:
    \[
    \hat\mu(\Phi_+)-\hat\mu(\Phi_-) = \frac{\indRS^{\tau'}(\gamma_+^k) - \indRS^{\tau'}(\gamma_-^k)}{k} = -\frac{2}{k}c_1^{\tau'}(A)=\frac{2}{k}c_1([u]\#\bar A).
    \] 
    We conclude that
    \begin{align*}
        \indRS^\tau(\gamma_+) -\indRS^\tau(\gamma_-) &= 
        \indRS^\tau(\gamma_+)-\hat\mu(\Phi^+) +\hat \mu(\Phi_-)-\indRS^\tau(\gamma_-) +\frac{2}{k}c_1([u]\#\bar A)
    \end{align*}
    We now split into cases. If $u$ is a simple curve, then $k=1$, $\gamma_\pm$ are $T$-periodic and therefore $\indRS^\tau(\gamma_\pm) = \hat\mu(\Phi_\pm)$ and so
    \begin{equation}
         \indRS^\tau(\gamma_+) -\indRS^\tau(\gamma_-)= \frac{2}{k}c_1([u]\#\bar A) <0
    \end{equation}
    where the inequality follows from our assumption that, if $u$ is a simple curve then the above Chern number is negative. Moving on to the other case, assume that  $k>1$ and $c_1([u]\#\bar A)\leq 0$. By  Lemma~\ref{lem:RS_vs_mean_index},\begin{align*}
        \indRS^\tau(\gamma_+) -\indRS^\tau(\gamma_-) &= 
        \indRS^\tau(\gamma_+)-\hat\mu(\Phi^+) +\hat \mu(\Phi_-)-\indRS^\tau(\gamma_-) +\frac{2}{k}c_1([u]\#\bar A)\\
        &\leq \indRS^\tau(\gamma_+)-\hat\mu(\Phi^+) +\hat \mu(\Phi_-)-\indRS^\tau(\gamma_-)\\
        &< n-1-d_+ +n-1-d_-,
    \end{align*}
    where $d_\pm:= \frac{1}{2}\dim{\ker{\Phi_\pm-\id}}$.
    Since $n-1-d_\pm\geq 0$ we conclude that in both cases we have the strict upper bound
    \[\indRS^\tau(\gamma_+) -\indRS^\tau(\gamma_-)<n-1-d_+ +n-1-d_-.\]
    The RS index of an orbit in a $2d$-dimensional Morse-Bott family of a periodic flow has the same parity as $n-1-d$ (see, e.g., \cite[equation (4.3)]{chaidez2022contact}). Therefore,
    \[
        \indRS^\tau(\gamma_+) -\indRS^\tau(\gamma_-) \leq n-1-d_+ +n-1-d_- - 2 = 2n-4-d_+-d_-. 
    \]
    Overall, and using the index formula from Lemma~\ref{lem:Fredholm_index}, the index of $v$ is bounded by
    \begin{align*}
        \on{ind}(v)&=  2c_1^\tau(u^* T\hat X)+  \indRS^\tau(\gamma_+)-\indRS^\tau(\gamma_-)+2+(d_+ + d_-)\\
        &\leq 0 +2n-4-d_+-d_-+2+d_+ + d_- = 2n-2.  \qedhere 
    \end{align*}
    \end{proof}

\subsubsection{Parametric Moduli Space} Next, we apply the index bound in Lemma \ref{lem:index_multiply_covered} to show that multiply covered curves and negative Chern curves do not intersect a generic point $p\in X$.
\begin{prop} \label{prop:parametric_curves_missing_covers}
Let $X$ be the conformal cobordism in Setup \ref{set:gap_for_periodic}. 
Fix $\rho>0$ and let $\cJ_\rho(X)\subset \cJ(X)$ be the subset of almost complex structures that are cylindrical outside of $X^\rho_{-\rho}$. 
\begin{enumerate}
    \item Let $J\in \cJ^{\on{reg}}_\rho$ (as in Proposition~\ref{prop:transversality_simple_cuves}) and consider the evaluation map 
    \begin{equation*}
        \on{ev}: \cM_\sigma(X;J,S_T,S_T)\rightarrow \hat X,\qquad (u,z)\mapsto u(z).
    \end{equation*}
   Then for a generic $p\in \hat X$, $\on{ev}^{-1}(p)\subset\cM_\sigma^{\on{reg}}(X;J,S_T,S_T)$ and is compact.
   \item \label{itm:generic_p_parametric} Let $\bJ$ be a generic path in $\cJ_\rho(X)$ and consider the evaluation map 
    \begin{equation*}
        \on{ev}: \moduliparam_\sigma (X;\bJ, S_T,S_T)\rightarrow \hat X,\qquad (\tau,(u,z))\mapsto u(z).
    \end{equation*}
   Then for a generic $p\in \hat X$, $\on{ev}^{-1}(p)\subset\moduliparam_\sigma^{\on{reg}}(X;\bJ,S_T,S_T)$ and is compact.
\end{enumerate}
\end{prop}
\begin{proof}
    We will prove only item (\ref{itm:generic_p_parametric}), as the proof of the other item is similar. Let $(\tau,u)$ be an element of the $C^\infty_{\on{loc}}$-compactification of $\on{ev}^{-1}(p)$. We need to show that 
    \[(\tau,u)\in {\on{ev}^{-1}(p)}\qquad \text{and}\qquad (\tau,u)\in \moduliparam_\sigma^{\on{reg}}(X;\bJ,S_T,S_T).\]
    Assume otherwise, then either $u$ is multiply covered, or, by Proposition~\ref{prop:low_energy_compactness}, $c_1([u]\#\bar A)<0$. 
    Lemma~\ref{lem:index_multiply_covered} guarantees that in any of these cases, $u=v\circ\varphi$  factors through a simple curve $v$, with index bounded by $2n-2$. We note that in the case where $c_1(u\#\bar A)<0$, $u$ could be simple, in which case $v=u$ and $\varphi = \id$. Denoting by $S_\pm$ the Morse-Bott families of the ends on $v$, of period $T/k$, and by $\nu:=(0,1,[v])$ the curve type of $v$, then 
    \[v\in \cM_\nu(X;J_\tau,S_+,S_-),\quad\text{ and therefore }\quad (\tau,v)\in \moduliparam_\nu(X;\bJ,S_+,S_-).\]
    Since $v$ is assumed to be a simple curve and $\bJ$ is a generic path, $(\tau,v)\in \moduliparam_\nu^{\on{reg}}(X;\bJ,S_+,S_-)$. By Proposition~\ref{prop:param_transversality_simple_cuves} the space   $\moduliparam_\nu^{\on{reg}}(X;\bJ,S_+,S_-)$ is a smooth manifold of dimension $\on{ind}(v)+1$. Recalling that $\on{ind}(v)\leq 2n-2$, we conclude that
    \[
        \on{dim} \moduliparam_\nu^{\on{reg}}(X;\bJ,S_+,S_-) = \on{ind}(v)+1 \leq 2n-1.
    \]
    Considering the evaluation map $\on{ev}:\moduliparam_\nu^{\on{reg}}(X;\bJ,S_+,S_-)\rightarrow \hat X$, Sard's theorem states that the image is a null set. Finally, notice that $u$ passes through a point $p$ if and only if $v$ does (though, with respect to a possibly different marked point).

    \vspace{3pt}
    
    We conclude that for every regular path $\bJ$, the image of 
    \[
    \on{ev}:\overline{\moduliparam_\sigma (X;\bJ, S_T,S_T)}\setminus \moduliparam_\sigma^{\on{reg}}(X;J,S_T,S_T)\rightarrow \hat X 
    \]
    is contained in the union of images of the evaluation map on the moduli spaces $\moduliparam_\nu^{\on{reg}}(X;\bJ,S_+,S_-)$ over all families $S_\pm$ and curve types $\nu=(0,1,B)$ such that:
    \begin{itemize}
        \item the period of $S_\pm$ is $\leq T$  
        \item and $c_1(B\#\bar A)\leq 0$,
    \end{itemize}
    and at least one of these two inequalities is strict. Since there are finitely many Morse-Bott families and countably many relative homology classes $B$, this is a countable union of null sets, and thus is a null set. Therefore,  for a generic point $p$, the preimage $\on{ev}^{-1}(p)$ is compact and contains only somewhere injective curves.
\end{proof}

\begin{lemma}\label{cor:same_count}
    Fix any $\rho>0$. For any $J_0,J_1\in \cJ^{\on{reg}}_\rho(X)$ and for a generic $p\in \hat X$,
    \begin{equation*}
        \# \cM_{\sigma}(X;J_0, p,S_T,S_T) = \# \cM_{\sigma}(X;J_1, p,S_T,S_T) \mod 2.
    \end{equation*}
\end{lemma}

\begin{proof} Proposition \ref{prop:parametric_curves_missing_covers} immediately implies that for a generic point $p\in \hat X$ and a generic path $\bJ$ in $\cJ_\rho(X)$ and, the parametric moduli space 
\begin{equation}
    \moduliparam_{\sigma}(X;\bJ, P,S_T,S_T):=\on{ev}^{-1}(p)\subset \moduliparam_\sigma^{\on{reg}} (X;\bJ, S_T,S_T)
\end{equation}
is a smooth compact manifold of dimension $\on{ind}(u)+1-2n=2n+1-2n = 1$. Since we only consider cylinders in the homology class of the trivial cylinder, the $\Omega$-energy of the curves in this moduli space is equal to  
\[E_\Omega(u)=(\f(0)-\f(-\epsilon)) T .\]
When the endpoints $J_0$ and $J_1$ of $\bJ$ are regular, the boundary of $\moduliparam_{\sigma}(X;\bJ, p,S_T,S_T)$ is the union of the zero dimensional manifolds 
\begin{equation*}
    \cM_{\sigma}(X;J_i, P, S_T,S_T):=\on{ev}^{-1}(P)\subset \cM_{\sigma}^{\on{reg}}(X;J_i, S_T,S_T),\qquad \text{for $i=0,1$}.
\end{equation*}
Thus the count of elements in $\cM_{\sigma}(X;J_0, P,S_T,S_T)$ and $\cM_{\sigma}(X;J_1, P,S_T,S_T)$ coincides modulo 2. Since any pair of regular almost complex structures can be connected by a regular path (or, alternatively, a regular count remains the same under perturbation), we conclude the result. \end{proof}

Our final step towards the proof of Proposition~\ref{prop:vir_count_triv_cyl} is the following simple observation.
\begin{lemma}\label{lem:vir_count_triv_cyl}
    Let $J\in \cJ_{\on{tr}}(X)$ be a translation invariant almost complex structure and let $P \in \hat  X$. Then, 
    \[
        \#\moduliLEJ =  1\mod 2.
    \]
\end{lemma}
\begin{proof}
    Let $\gamma$ be the unique $T$-periodic Reeb orbit in $Y$ that passes through the point $P$. Let us show that $\moduliLEJ$ consists of the single trivial cylinder  $u_\gamma$ of $\gamma$. Clearly, $u_\gamma$ is in the moduli space. Thus suppose that there exists another cylinder $u\in\moduliLEJ$. Then, \begin{equation}\label{eq:omega_wannabe_triv_difference}
        \int_{\R\times S^1} u^*\omega = E_\Omega(u) - \int_{u^{-1}(X)}u^*d(\f(r)\theta)) = E_\Omega(u) - \big(\f(0)-\f(-\epsilon)T\big) = 0,
    \end{equation}
    where the last equality is due to the fact that the $\Omega$-energy depends only on the relative homology class and the ends (see Definition~\ref{def:energy_cobordism}). 
    Since $J\in\cJ_{\on{tr}}(X)$ is cylindrical on all of $X$, the fact that $\omega(u)=0$ implies that $u$ must be tangent to $\partial_r$ and the Reeb vector field, everywhere. Hence, $u$ must be a trivial cylinder over an orbit $\gamma':=\gamma_+ = \gamma_-\in S_T$. Moreover, we assumed that $u$ passes through $p$, and thus we conclude that $\gamma'$ is the unique $T$-periodic orbit that passes through $\pi(p)$, where $\pi:X=\R\times Y\rightarrow Y$ is the obvious projection. As a consequence, $\gamma' = \gamma$ and $u= u_\gamma$. 

    \vspace{3pt}

    Having concluded that $u_\gamma$ is the unique element in $\moduliLEJ$, in order to conclude the proof it is sufficient to show that $u_\gamma$ is cut-out transversely. This is exactly Proposition~\ref{prop:transversality_triv_cyl_w_pt}.
\end{proof}

\subsection{Proof Of Theorem \ref{thm:gap0_for_periodic}} We conclude by applying the curve count result (Proposition \ref{prop:vir_count_triv_cyl}) to prove the main result of the section, Theorem \ref{thm:gap0_for_periodic}.

\begin{proof} (Theorem~\ref{thm:gap0_for_periodic}) It suffices to prove the following bound.
\begin{equation} \label{eqn:gJP_periodic_bound}
\g_{\sigma,T}(X;J,P) \le -\f(\epsilon) \cdot T \qquad\text{for each pair}\qquad (J,P) \in \mathcal{J}(X) \times \hat{X} \end{equation}
Then supremizing over all pairs $(J,P)$ will yield the desired result. To prove (\ref{eqn:gJP_periodic_bound}), we start by assuming that $J \in \mathcal{J}^{\on{reg}}_a(X)$ for some $a > 0$ and assume that $P \in \hat{X}$ is a regular value of the evaluation map
\begin{equation} \label{eqn:gJP_periodic_bound_ev} \on{ev}:\cM^{\on{reg}}(X;J,S_+,S_-) \to \hat{X}\end{equation}
By Proposition \ref{prop:vir_count_triv_cyl}, the moduli space of $J$-holomorphic cylinders in $\hat{X}$ passing through $P$ connecting $S_T$ to $S_T$ is compact and $0$-dimensional, with count $1$ mod $2$. That is
\[\# \cM_{\sigma}(X;J,P,S_T,S_T) = 1 \mod 2\]
In particular, the moduli space $\cM_{\sigma}(X;J,P,S_T,S_T)$ is non-empty. Any cylinder $u:\Sigma \to \hat{X}$ in this moduli space has positive end of period $T$ and $\Omega$-energy given by $-\f(-\epsilon) \cdot T$. Thus
\[\g_{\sigma,T}(X;J,P) \le E_\Omega(u) \le -\f(-\epsilon) \cdot T\]
For a general pair $(J,P) \in \mathcal{J}(X) \times \hat{X}$, we reason as follows. The almost complex structure is cylindrical outside of $X^a_{\smallneg a}$ for some $a > 0$. We can choose a sequence
\[
(J_i,P_i) \in \mathcal{J}^{\on{reg}}_a(X) \times \hat{X}
\]
such that $P_i$ is a regular value of the evaluation map and $(J_i,P_i) \to (J,P)$. In fact, the set of such pairs is comeager in $\mathcal{J}_a(X) \times \hat{X}$. Then by the lower semicontinuity of the spectral gap (Lemma \ref{lem:minimizer_and_JP_semicont}), we have
\[
\g_{\sigma,T}(X;J,P) \le \lim_{i \to \infty} \g_{\sigma,T}(X;J_i,P_i) \le -\f(-\epsilon) \cdot T \qedhere
\]
\end{proof}

\bibliographystyle{alpha}
\bibliography{refs.bib}

\newcommand{\etalchar}[1]{$^{#1}$}
\begin{thebibliography}{BEH{\etalchar{+}}03}

\bibitem[ABE23]{abbondandolo2023symplectic}
Alberto Abbondandolo, Gabriele Benedetti, and Oliver Edtmair.
\newblock Symplectic capacities of domains close to the ball and banach-mazur
  geodesics in the space of contact forms.
\newblock {\em arXiv preprint arXiv:2312.07363}, 2023.

\bibitem[ADE14]{audin2014morse}
Michele Audin, Mihai Damian, and Reinie Ern{\'e}.
\newblock {\em Morse theory and Floer homology}.
\newblock Springer, 2014.

\bibitem[AGZ18a]{agz2018}
Peter Albers, Hansj{\"o}rg Geiges, and Kai Zehmisch.
\newblock Reeb dynamics inspired by katok’s example in finsler geometry.
\newblock {\em Mathematische Annalen}, 370(3):1883--1907, 2018.

\bibitem[AGZ18b]{albers2018reeb}
Peter Albers, Hansj{\"o}rg Geiges, and Kai Zehmisch.
\newblock Reeb dynamics inspired by katok’s example in finsler geometry.
\newblock {\em Mathematische Annalen}, 370(3):1883--1907, 2018.

\bibitem[AI16]{a2016}
Masayuki Asaoka and Kei Irie.
\newblock A $c^\infty$ closing lemma for hamiltonian diffeomorphisms of closed
  surfaces.
\newblock {\em Geometric and Functional Analysis}, 26(5):1245--1254, 2016.

\bibitem[BEH{\etalchar{+}}03]{sftCompactness}
F.~Bourgeois, Y.~Eliashberg, H.~Hofer, K.~Wysocki, and E.~Zehnder.
\newblock Compactness results in symplectic field theory.
\newblock {\em Geom. Topol.}, 7:799--888, 2003.

\bibitem[Bou02]{b2002}
F.~Bourgeois.
\newblock {\em A Morse--Bott approach to contact homology}.
\newblock PhD thesis, Stanford University, 2002.

\bibitem[CDPT22]{chaidez2022contact}
Julian Chaidez, Ipsita Datta, Rohil Prasad, and Shira Tanny.
\newblock Contact homology and higher dimensional closing lemmas.
\newblock {\em arXiv preprint arXiv:2206.04738}, 2022.

\bibitem[CGPZ21]{cpz2021}
Dan Cristofaro-Gardiner, Rohil Prasad, and Boyu Zhang.
\newblock Periodic {F}loer homology and the smooth closing lemma for
  area-preserving surface diffeomorphisms.
\newblock {\em arXiv preprint arXiv:2110.02925}, 2021.

\bibitem[CR02]{chen2002orbifold}
Weimin Chen and Yongbin Ruan.
\newblock Orbifold gromov-witten theory.
\newblock {\em Orbifolds in Mathematics and Physics}, 310:25--85, 2002.

\bibitem[CS22]{cineli2022strong}
Erman Cineli and Sobhan Seyfaddini.
\newblock The strong closing lemma and hamiltonian pseudo-rotations.
\newblock {\em arXiv preprint arXiv:2210.00771}, 2022.

\bibitem[CV15]{cv2014}
Kai Cieliebak and Evgeny Volkov.
\newblock First steps in stable hamiltonian topology.
\newblock {\em Journal of the European Mathematical Society}, 17(2):321--404,
  2015.

\bibitem[Edt22]{edtmair2022elementary}
Oliver Edtmair.
\newblock An elementary alternative to pfh spectral invariants.
\newblock {\em arXiv preprint arXiv:2207.12553}, 2022.

\bibitem[EGH00]{egh2000}
Yakov Eliashberg, Alexander Givental, and Helmut Hofer.
\newblock Introduction to symplectic field theory.
\newblock {\em Geom Funct Anal}, 2000, 11 2000.

\bibitem[EH21]{eh2021}
Oliver Edtmair and Michael Hutchings.
\newblock {PFH} spectral invariants and {$C^\infty$} closing lemmas.
\newblock {\em arXiv preprint arXiv:2110.02463}, 2021.

\bibitem[FH22]{fish2022almost}
Joel~W Fish and Helmut~HW Hofer.
\newblock Almost existence from the feral perspective and some questions.
\newblock {\em Ergodic Theory and Dynamical Systems}, 42(2):792--834, 2022.

\bibitem[GG20]{ginzburg2020lusternik}
Viktor~L Ginzburg and Basak~Z Gurel.
\newblock Lusternik--schnirelmann theory and closed reeb orbits.
\newblock {\em Mathematische Zeitschrift}, 295(1):515--582, 2020.

\bibitem[Gut87]{g1987}
Carlos Gutierrez.
\newblock A counter-example to a {$C^2$} closing lemma.
\newblock {\em Ergodic Theory and Dynamical Systems}, 7(4):509--530, 1987.

\bibitem[Gut14]{gutt2014generalized}
Jean Gutt.
\newblock Generalized {C}onley-{Z}ehnder index.
\newblock In {\em Annales de la Facult{\'e} des sciences de Toulouse:
  Math{\'e}matiques}, volume~23, pages 907--932, 2014.

\bibitem[Her79]{h1979}
Michael~R Herman.
\newblock Sur la conjugaison diff{\'e}rentiable des diff{\'e}omorphismes du
  cercle {\`a} des rotations.
\newblock {\em Publications Math{\'e}matiques de l'IH{\'E}S}, 49:5--233, 1979.

\bibitem[Her91a]{herman1991exemples}
M-R Herman.
\newblock Exemples de flots hamiltoniens dont aucune perturbation en topologie
  ${C}^{\infty}$ n'a d'orbites p{\'e}riodiques sur un ouvert de surfaces
  d'{\'e}nergies.
\newblock {\em CR Acad. Sci. Paris S{\'e}r. I Math.}, 312:989, 1991.

\bibitem[Her91b]{herman1991differentiabilite}
Michael-R Herman.
\newblock Diff{\'e}rentiabilit{\'e} optimale et contre-exemples {\'a} la
  fermeture en topologie ${C}^\infty$ des orbites r{\'e}currentes de flots
  hamiltoniens.
\newblock {\em Comptes rendus de l'Acad{\'e}mie des sciences. S{\'e}rie 1,
  Math{\'e}matique}, 313(1):49--51, 1991.

\bibitem[Hir23]{h2023}
Amanda Hirschi.
\newblock Properties of gromov-witten invariants defined via global kuranishi
  charts.
\newblock {\em arXiv preprint arXiv:2312.03625}, 2023.

\bibitem[HS22]{hs2022}
Amanda Hirschi and Mohan Swaminathan.
\newblock Global kuranishi charts and a product formula in symplectic
  gromov-witten theory.
\newblock {\em arXiv preprint arXiv:2212.11797}, 2022.

\bibitem[Hut10]{h2010}
M.~Hutchings.
\newblock Quantitative embedded contact homology.
\newblock {\em Journal of differential geometry}, 88, 05 2010.

\bibitem[Hut22]{hutchings2022elementary}
Michael Hutchings.
\newblock Elementary spectral invariants and quantitative closing lemmas for
  contact three-manifolds.
\newblock {\em arXiv preprint arXiv:2208.01767}, 2022.

\bibitem[HWZ17]{hofer2017applications}
Helmut Hofer, Kris Wysocki, and Eduard Zehnder.
\newblock {\em Applications of polyfold theory I: the polyfolds of
  Gromov--Witten theory}, volume 248.
\newblock American Mathematical Society, 2017.

\bibitem[Iri15]{i2015}
Kei Irie.
\newblock Dense existence of periodic reeb orbits and ech spectral invariants.
\newblock {\em Journal of Modern Dynamics}, 9(01):357, 2015.

\bibitem[Iri18]{i2018}
Kei Irie.
\newblock Equidistributed periodic orbits of $c^\infty$-generic
  three-dimensional reeb flows.
\newblock {\em arXiv preprint arXiv:1812.01869}, 2018.

\bibitem[Iri22]{irie2022strong}
Kei Irie.
\newblock Strong closing property of contact forms and action selecting
  functors.
\newblock {\em arXiv preprint arXiv:2201.09216}, 2022.

\bibitem[Ler95]{lerman1995symplectic}
Eugene Lerman.
\newblock Symplectic cuts.
\newblock {\em Mathematical Research Letters}, 2(3):247--258, 1995.

\bibitem[Mcd09]{mcduff2009hamiltonian}
Dusa Mcduff.
\newblock Hamiltonian $s^1$-manifolds are uniruled.
\newblock {\em Duke Math. J.}, 146(1):449--507, 2009.

\bibitem[MS12]{mcduff2012j}
Dusa McDuff and Dietmar Salamon.
\newblock {\em J-holomorphic curves and symplectic topology}, volume~52.
\newblock American Mathematical Soc., 2012.

\bibitem[MS21]{mcduff2021symplectic}
Dusa McDuff and Kyler Siegel.
\newblock Symplectic capacities, unperturbed curves, and convex toric domains.
\newblock {\em arXiv preprint arXiv:2111.00515}, 2021.

\bibitem[Par15]{p2015}
John Pardon.
\newblock Contact homology and virtual fundamental cycles.
\newblock {\em Journal of the American Mathematical Society}, 32, 08 2015.

\bibitem[Par16a]{p2016}
John Pardon.
\newblock An algebraic approach to virtual fundamental cycles on moduli spaces
  of pseudo-holomorphic curves.
\newblock {\em Geometry \& Topology}, 20(2):779--1034, 2016.

\bibitem[Par16b]{pardon2016algebraic}
John Pardon.
\newblock An algebraic approach to virtual fundamental cycles on moduli spaces
  of pseudo-holomorphic curves.
\newblock {\em Geometry \& Topology}, 20(2):779--1034, 2016.

\bibitem[Pei88]{px1988}
Maria L{\'u}cia~Alvarenga Peixoto.
\newblock The closing lemma for generalized recurrence in the plane.
\newblock {\em Transactions of the American Mathematical Society},
  308(1):143--158, 1988.

\bibitem[Poi99]{p1899}
Henri Poincar{\'e}.
\newblock {\em Les m{\'e}thodes nouvelles de la m{\'e}canique c{\'e}leste},
  volume~3.
\newblock Gauthier-Villars et fils, imprimeurs-libraires, 1899.

\bibitem[PR83]{pr1983}
Charles~C Pugh and Clark Robinson.
\newblock The c1 closing lemma, including hamiltonians.
\newblock {\em Ergodic Theory and Dynamical Systems}, 3(2):261--313, 1983.

\bibitem[Pra21]{p2021}
Rohil Prasad.
\newblock Generic equidistribution of periodic orbits for area-preserving
  surface maps.
\newblock {\em arXiv preprint arXiv:2112.14601}, 2021.

\bibitem[Pug67a]{p1967a}
Charles~C Pugh.
\newblock The closing lemma.
\newblock {\em American Journal of Mathematics}, 89(4):956--1009, 1967.

\bibitem[Pug67b]{p1967b}
Charles~C Pugh.
\newblock An improved closing lemma and a general density theorem.
\newblock {\em American Journal of Mathematics}, 89(4):1010--1021, 1967.

\bibitem[Pug75]{p1975}
Charles~C Pugh.
\newblock Against the $c^2$ closing lemma.
\newblock {\em Journal of Differential Equations}, 17(2):435--443, 1975.

\bibitem[Rob70]{r1970}
R~Clark Robinson.
\newblock A global approximation theorem for hamiltonian systems.
\newblock In {\em Global Analysis (Proc. Sympos. Pure Math}, volume~14, pages
  233--243, 1970.

\bibitem[RS93]{robbin1993maslov}
Joel Robbin and Dietmar Salamon.
\newblock The {M}aslov index for paths.
\newblock {\em Topology}, 32(4):827--844, 1993.

\bibitem[Sch18]{schmaltz2018gromov}
Wolfgang~William Schmaltz.
\newblock {\em Gromov-Witten axioms for symplectic manifolds via polyfold
  theory}.
\newblock University of California, Berkeley, 2018.

\bibitem[Sie08]{siefring2008relative}
Richard Siefring.
\newblock Relative asymptotic behavior of pseudoholomorphic half-cylinders.
\newblock {\em Communications on Pure and Applied Mathematics: A Journal Issued
  by the Courant Institute of Mathematical Sciences}, 61(12):1631--1684, 2008.

\bibitem[Sma98]{s1998}
Steve Smale.
\newblock Mathematical problems for the next century.
\newblock {\em The mathematical intelligencer}, 20(2):7--15, 1998.

\bibitem[Sul78]{s1978}
Dennis Sullivan.
\newblock A foliation of geodesics is characterized by having no “tangent
  homologies”.
\newblock {\em Journal of Pure and Applied Algebra}, 13(1):101--104, 1978.

\bibitem[Wad75]{w1975}
Andrew~W Wadsley.
\newblock Geodesic foliations by circles.
\newblock {\em Journal of Differential Geometry}, 10(4):541--549, 1975.

\bibitem[Wen15a]{wendlholomorphic}
Chris Wendl.
\newblock Lectures on holomorphic curves in symplectic and contact geometry.
\newblock {\em arxiv preprint}, 2015.

\bibitem[Wen15b]{wendlsymplectic}
Chris Wendl.
\newblock Symplectic field theory lecture notes.
\newblock {\em arxiv preprint}, 2015.

\bibitem[Xue22]{x2022}
Jinxin Xue.
\newblock Strong closing lemma and {KAM} normal form.
\newblock {\em arXiv preprint arXiv:2207.06208}, 2022.

\end{thebibliography}

\end{document}